\documentclass{article}
\usepackage[utf8]{inputenc}
\usepackage[margin = 1in]{geometry}
\usepackage{xparse}
\usepackage{amsmath}
\usepackage{amssymb}
\usepackage{amsfonts}

\usepackage{diagbox}
\usepackage{color}
\usepackage{multicol}
\usepackage{epsfig}
\usepackage{epstopdf}
\usepackage{bm}

\usepackage[labelformat=simple]{subcaption}

\graphicspath{{figures/}}
\usepackage[hidelinks]{hyperref}
\usepackage{amsthm}

\usepackage[format=plain,
            labelfont={it},
            textfont=it]{caption}

\newtheorem{result}{{\bf Proposition}}
\newtheorem{lemma}{Lemma}

\usepackage[title]{appendix}

\renewcommand\theequation{\thesection.\arabic{equation}}

\renewcommand{\theequation}{\arabic{section}.\arabic{equation}}

\title{An Asymptotic Analysis of Spike Self-Replication and Spike
  Nucleation of Reaction-Diffusion Patterns on Growing 1-D Domains}
\author{Chunyi Gai \thanks{Department of Mathematics and Statistics,
    University of Northern British Columbia, Prince George, B.C.,
    Canada, V2N 4Z9 (corresponding author)}, Edgardo
  Villar-Sep\'ulveda \thanks{Department of Engineering Mathematics,
    University of Bristol, Bristol, U.K. BS8 1TW}, Alan Champneys
  \thanks{Department of Engineering Mathematics, University of
    Bristol, Bristol, U.K. BS8 1TW}, Michael J. Ward
  \thanks{Department of Mathematics, University of British Columbia,
    Vancouver, B.C., Canada, V6T 1Z2}} \date{\today}

\numberwithin{equation}{section}
\begin{document}
\maketitle
 
\begin{abstract}
In the asymptotic limit of a large diffusivity ratio, certain two-component reaction-diffusion (RD) systems can admit localized spike solutions on a one-dimensional finite domain in a far-from-equilibrium nonlinear regime. It is known that two distinct bifurcation mechanisms can occur which generate spike patterns of increased spatial complexity as the domain half-length $L$ slowly increases; so-called {\em spike nucleation} and {\em spike self-replication}. Self-replication is found to occur via the passage beyond a saddle-node bifurcation point that can be predicted through linearization around the inner spike profile.  In contrast, spike nucleation occurs through slow passage beyond the saddle-node of a nonlinear boundary-value problem defined in the outer region away from the core of a spike. Here, by treating $L$ as a static parameter, precise conditions are established within the semi-strong interaction asymptotic regime to determine which occurs, conditions that are confirmed by numerical simulation and continuation.

For the Schnakenberg and Brusselator RD models, phase diagrams in parameter space are derived that predict whether spike self-replication or spike nucleation will occur first as $L$ is increased, or whether no such instability will occur. For the Gierer-Meinhardt model with a non-trivial activator background, spike nucleation is shown to be the only possible spike-generating mechanism. From time-dependent PDE numerical results on an exponentially slowly growing domain, it is shown that the analytical thresholds derived from the asymptotic theory accurately predict critical values of $L$ where either spike self-replication or spike-nucleation will occur. The global bifurcation mechanism for transitions to patterns of increased spatial complexity is further elucidated by superimposing time-dependent PDE simulation results on the numerically computed solution branches of spike equilibria in which $L$ is the primary bifurcation parameter.
\end{abstract}

\section{Introduction}

The formation of spatial patterns on growing domains is one of the key issues in developmental biology. Originating in 1952, Turing \cite{turing} pioneered the formulation and linear stability analysis of a system of reaction-diffusion (RD) equations to describe the formation of morphogen gradient patterns on the developing embryo. Since then, there have been numerous mathematical models proposed to study pattern-formation aspects of various developmental systems in biology. In particular, it is well-known that spatial domain growth can significantly influence pattern formation on the skin of the zebrafish \cite{kondo}, and there are now various mechanisms that have been proposed to theoretically explain these skin pigmentation patterns (cf.~\cite{volkening}, \cite{kondo2}). In a broader context, the important role of domain growth for various physical problems, with many related to fluid mechanics, are surveyed in \cite{knobloch_2015}.

With regards to modeling biological patterns through {\em activator-inhibitor} RD systems, it was shown from numerical experiments and a scaling law derived by Crampin {\em et al} \cite{crampin_1999} (see also \cite{crampin_mode}) that domain growth can generate robust mode-doubling spatial patterning in certain RD systems under conditions of a Turing instability. This pioneering work was extended to study the effect of 1-D nonuniform domain growth \cite{crampin_2002}, domain growth in multi-dimensional settings and on manifolds \cite{plaza_2004}, domain growth driven by the concentration of the diffusing species (cf.~\cite{krause_concengrowth,neville}), and the strong sensitivity of spatial patterning when the domain growth is very slow \cite{barrass}. Numerical methods to study RD pattern formation on growing multi-dimensional domains have been developed and implemented in \cite{madzvamuse_2006,madzvamuse_velocity}, and these numerical studies have revealed various dynamical behaviors such as spot self-replication.

From a theoretical viewpoint, Turing's original linear stability analysis in \cite{turing} based on a linearization around a spatially homogeneous state does not apply in a time-dependent setting with uniform domain growth. In a more intricate non-autonomous setting, necessary conditions for the emergence of spatial patterns from a spatially homogeneous state under uniform domain growth in 1-D \cite{madzvamuse_2010} and, more recently, in multi-dimensional contexts and on manifolds \cite{krause_turing} have been derived.

In contrast, for a different `far-from-equilibrium' nonlinear regime that also involves a large diffusivity ratio, it is known that stable spike-like patterns can occur in which the activator field is much more highly localized than the inhibitor (see e.g.~\cite{Gai} and references therein). This leads to a so-called {\em semi-strong interaction} asymptotic scaling in which activator spikes exist on an inner domain, and interact through the far-field inhibitor component \cite{wardfirst,DoelmanKaper}.

Starting with a one-spike steady-state solution on a 1-D finite domain, the main goal of this paper is to use asymptotic analysis and bifurcation theory to study the mechanisms that lead to the generation of spike patterns of increased spatial complexity as the domain length slowly increases in time. Although we will focus on an imposed slow exponential domain growth, our theoretical results will apply to other models of slow uniform domain growth.

To illustrate the rather wide applicability of our identified mechanisms for spike generation, the analysis is undertaken for three prototypical two-component RD systems that allow for the existence of spike solutions in the semi-strong interaction regime of a large diffusivity ratio. Our analysis for this semi-strong regime is distinct from that in \cite{ueda} where spike self-replication was studied for slow linear domain growth in the weak-interaction regime where localized spikes interact primarily through their exponentially small tails.  We emphasize that since spike patterns are {\em far-from-equilibrium} spatial structures (cf. \cite{nishiura_book}), the linear stability analyses of spatially uniform states in \cite{madzvamuse_2010,krause_turing} (see references therein), which extended Turing's original stability theory \cite{turing} to a non-autonomous setting, is not the relevant theoretical framework for analyzing spike-generating mechanisms on a slowly growing domain.

\subsection{Three example reaction-diffusion systems}

The first system we study is  the Schnakenberg activator-substrate model used in \cite{crampin_1999}. In dimensionless form, this model can be written in the form
\begin{subequations}\label{Schnakenberg}
    \begin{align}
        v_t &= \varepsilon^2 \, v_{XX} - v + a + u \, v^2 \, , \quad - L < X < L \,,
        \\
        u_t &= D \, u_{XX} + b - u \, v^2 \, , \quad  -L < X < L \, ,
    \end{align}
\end{subequations}
with Neumann boundary conditions $v_X(\pm L, t) = u_X(\pm L,t) = 0$. The parameters $a, b > 0$ represent constant activator and substrate feed rates, respectively. In the semi-strong regime, we assume that $D = \mathcal O(1)$ and $0 < \varepsilon \ll 1$, so that the activator $v$ diffuses more slowly than does $u$.

The second system is the well-studied Brusselator model \cite{prig}. As non-dimensionalized in Appendix \ref{app:nondim}, this system can be written as 
\begin{subequations}\label{Brusselator}
    \begin{align}
        v_t &= \varepsilon^2 \, v_{XX} - v + a + f u \, v^2 \, , \quad - L < X < L \,; \qquad v_x(\pm L, t) = 0\,,
        \\
        u_t &= D \, u_{XX} + v - u \, v^2 \, , \quad  - L < X < L \,; \qquad u_x(\pm L, t) = 0 \,,     
    \end{align}
\end{subequations}
where $a > 0$ and $0 < f < 1$ are the two key model parameters. Again, we will consider the semi-strong regime $D = \mathcal O(1)$ and $0 < \varepsilon \ll 1$, so that $v$ diffuses slower than $u$.

Finally, we consider the Gierer–Meinhardt (GM) model \cite{gierer} for an activator $\mathcal A$ and inhibitor $\mathcal H$. After non-dimensionalization, (see Appendix \ref{app:nondim}), this can be written in the form
\begin{subequations}\label{GM}
    \begin{align}
      \mathcal A_t &= \varepsilon^2 \, \mathcal A_{XX} - \mathcal A + \frac{\mathcal A^2}{\mathcal H} + \kappa \, , \quad - L < X < L \, ; \qquad \mathcal A_X(\pm L, t) = 0 \,,
      \\
      \tau  \mathcal H_t &= D \, \mathcal H_{XX} - \mathcal H + \mathcal A^2 \, , \quad - L < X < L \, ; \qquad \mathcal H_X(\pm L, t) = 0 \,,
    \end{align}
\end{subequations}
where $\tau > 0$ is a constant, $D = \mathcal O(1)$ and $0 < \varepsilon \ll 1$. The novel feature of the model \eqref{GM}, as compared to the analysis in \cite{iron-ward-wei} of spike patterns, is that in \eqref{GM} we now allow for a non-trivial background $\kappa>0$ for the activator. As we will show below, the inclusion of $\kappa > 0$ is necessary for generating new spikes as the domain length slowly increases.

In the three RD systems \eqref{Schnakenberg}--\eqref{GM}, we assume that the one-dimensional domain of half-length $L(t)$ grows exponentially in time with an asymptotically slow growth rate $\rho$, with $\rho = \mathcal O(\varepsilon^2)$, so that
\begin{equation}\label{domain}
    L(t) = L_0 \, e^{\rho \, t} \,.
\end{equation}
The initial domain half-length is $L_0 = 1$. Upon converting the RD model \eqref{Schnakenberg} to a Lagrangian framework as in \cite{crampin_1999}, where we set $x = X/L$, we obtain the following rescaled system for \eqref{Schnakenberg} given by
\begin{subequations}\label{Sch_Lagrange}
    \begin{align}
        v_t &= \varepsilon_L^2 \, v_{xx} - v - \rho \, v + a + u \, v^2 \, , \quad \, - 1 < x < 1 \, ,
        \\
        u_t &= D_L \, u_{xx} - \rho \, u + b - u \, v^2 \, , \quad - 1 < x < 1 \,.
    \end{align}
\end{subequations}   
In \eqref{Sch_Lagrange}, $- \rho v$ and $- \rho u$ are the so-called {\em dilution terms}, and
% we have defined $\varepsilon_L$ and $D_L$ by
\begin{align}\label{scale}
    \varepsilon_L \equiv \frac{\varepsilon}{L} \, , \quad D_L \equiv \frac{D}{L^2} \, .
\end{align}
Since $\rho = \mathcal O\left(\varepsilon^2\right) \ll 1$, we obtain that, to leading order, \eqref{Sch_Lagrange} simplifies to
\begin{subequations}\label{Sch_core}
    \begin{align}
        v_t &= \varepsilon_L^2 \, v_{xx} - v  + a + u \, v^2 \, , \quad - 1 < x < 1 \, ; \qquad v_x(\pm 1, t) = 0 \, , \label{Sch1}
        \\
        u_t &= D_L \, u_{xx} + b - u \, v^2 \, , \quad - 1 < x < 1 \, ; \qquad u_x(\pm 1, t) = 0 \, . \label{Sch2}
    \end{align}
\end{subequations}
In a similar way, the Lagrangian representation of the Brusselator \eqref{Brusselator}, assuming that $\rho = \mathcal O\left(\varepsilon^2\right) \ll 1$ is negligible, is given by
\begin{subequations}\label{Brusselator_Lagrange}
    \begin{align}
        v_t &= \varepsilon_L^2 \, v_{xx} - v + a + f u \, v^2 \, , \quad - 1 < x < 1 \, ; \qquad v_x(\pm 1, t) = 0 \, , \label{Brusselator1}
        \\
        u_t &= D_L \, u_{xx} + v - u \, v^2 \, , \quad - 1 < x < 1 \, ; \qquad u_x(\pm 1, t) = 0 \, , \label{Brusselator2}
    \end{align}
\end{subequations}
where $\varepsilon_L$ and $D_L$ are defined in \eqref{scale}. Finally, the Lagrangian representation for the GM model \eqref{GM}, upon neglecting the asymptotically small dilution terms, is given by
\begin{subequations}\label{GM_Lagrange}
    \begin{align}
        \mathcal A_t &= \varepsilon_L^2 \, \mathcal A_{xx} - \mathcal A + \frac{\mathcal A^2}{\mathcal H} + \kappa \, , \quad - 1 < x < 1 \, ; \qquad \mathcal A_x(\pm 1, t) = 0 \, ,
        \\
        \tau  \mathcal H_t &= D_L \, \mathcal H_{xx} - \mathcal H + \mathcal A^2 \, , \quad - 1 < x < 1 \, ; \qquad \mathcal H_x(\pm 1, t) = 0 \, .
    \end{align}
\end{subequations}

\subsection{The main results}

For these three RD systems, in the limit $\varepsilon_L\to 0$ we will use asymptotic analysis and numerical-path following methods to identify and analyze in detail two distinct mechanisms for generating new spikes as the domain length slowly increases. One mechanism is {\em spike self-replication}, whereby each individual spike self-replicates into two, which doubles the number of spikes on the domain. The other mechanism is {\em spike nucleation} or insertion, whereby a new spike is generated at the domain boundaries or from the quiescent background between adjacent spikes as $L$ increases past a threshold.

By treating $L$ as a static bifurcation parameter, we will show that spike self-replication can occur for the Schnakenberg and Brusselator models owing to the slow passage beyond a saddle-node bifurcation point associated with the local profile of a spike. For all three RD models, we will show that the mechanism for spike nucleation is a slow passage beyond an additional saddle-node point that can be predicted by the non-existence threshold of solutions to a nonlinear boundary-value problem defined in the outer region away from the core of a spike. For certain parameter ranges, that we will identify in our analysis, either spike self-replication or spike-nucleation can occur for the Schnakenberg and Brusselator RD systems. In contrast, spike self-replication behavior is not possible for the GM model \eqref{GM_Lagrange}, and spike nucleation will occur only when $0 < \kappa < 1$.  In other parameter ranges, we will find that no spike-generating mechanism can occur for any of these models. For each, we will provide phase diagrams space which delineate parameter values in which the three distinct behaviors can occur. Those state diagrams for the Schnakenberg and Brusselator examples are summarized in Fig.~\ref{fig:phase_diag}.

From time-dependent PDE numerical results on an exponentially slow-growing domain with growth rate $\rho = \varepsilon^2$, we will show that the analytical thresholds derived from our asymptotic theory for each RD model accurately predict the critical values of $L$ where spike self-replication or spike nucleation will occur. By choosing $\rho = \mathcal O(\varepsilon^2)$, we ensure that the domain growth rate is asymptotically smaller than the slow $\mathcal O(\varepsilon)$ growth rate that is associated with any unstable small eigenvalues of the linearization of spike equilibria (cf.~\cite{kolokolnikov2005existence}). In this sense, when $\rho = \varepsilon^2$, the time-dependent solution should be well-approximated by quasi-steady solutions parametrized by $L$. Indeed, we shall find, through numerical continuation of steady multi-spike equilibria, that this approximation is close, and that the eigenfunctions of the fold equilibria of the steady equilibria also predict which of nucleation or self-replication occurs. 

\begin{figure}[h!tbp]
    \centering
    \begin{subfigure}[b]{0.48\textwidth}  
        \includegraphics[width=\textwidth, height=6.0cm]{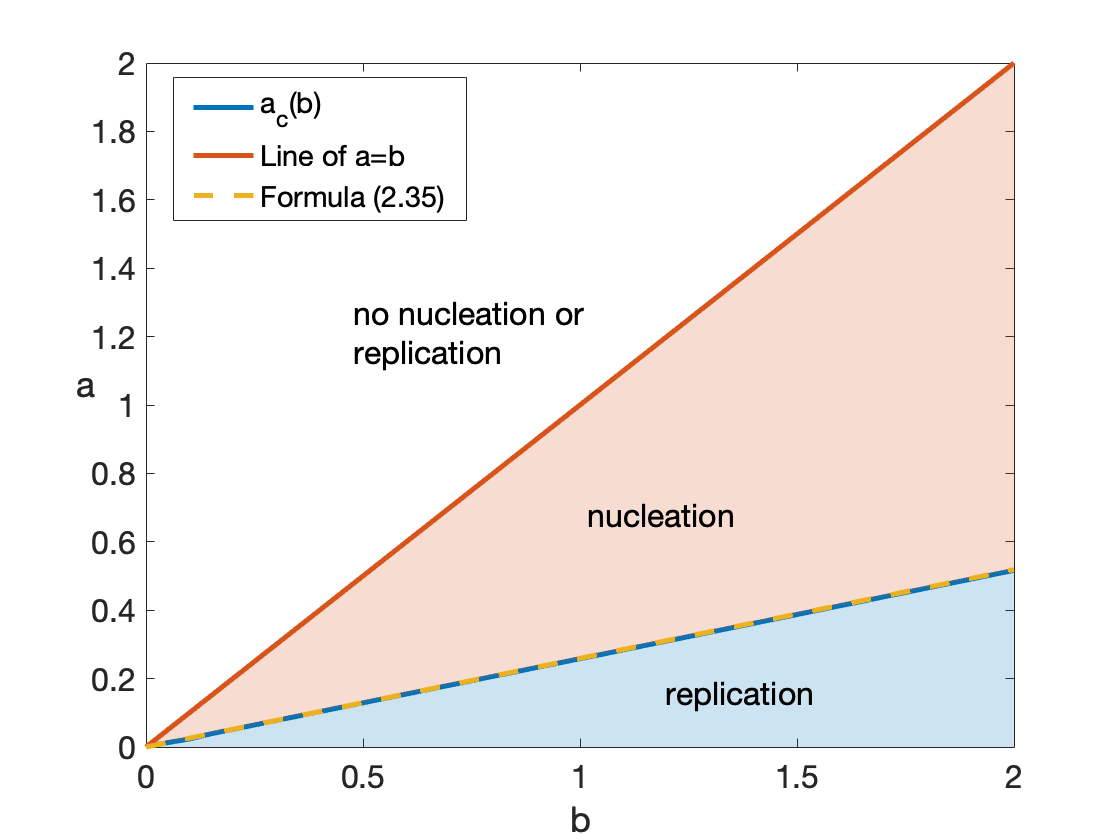}
        \caption{Schnakenberg model}
        \label{fig:Sch_phase_diag}
    \end{subfigure}
    \begin{subfigure}[b]{0.48\textwidth}
        \includegraphics[width=\textwidth, height=6.0cm]{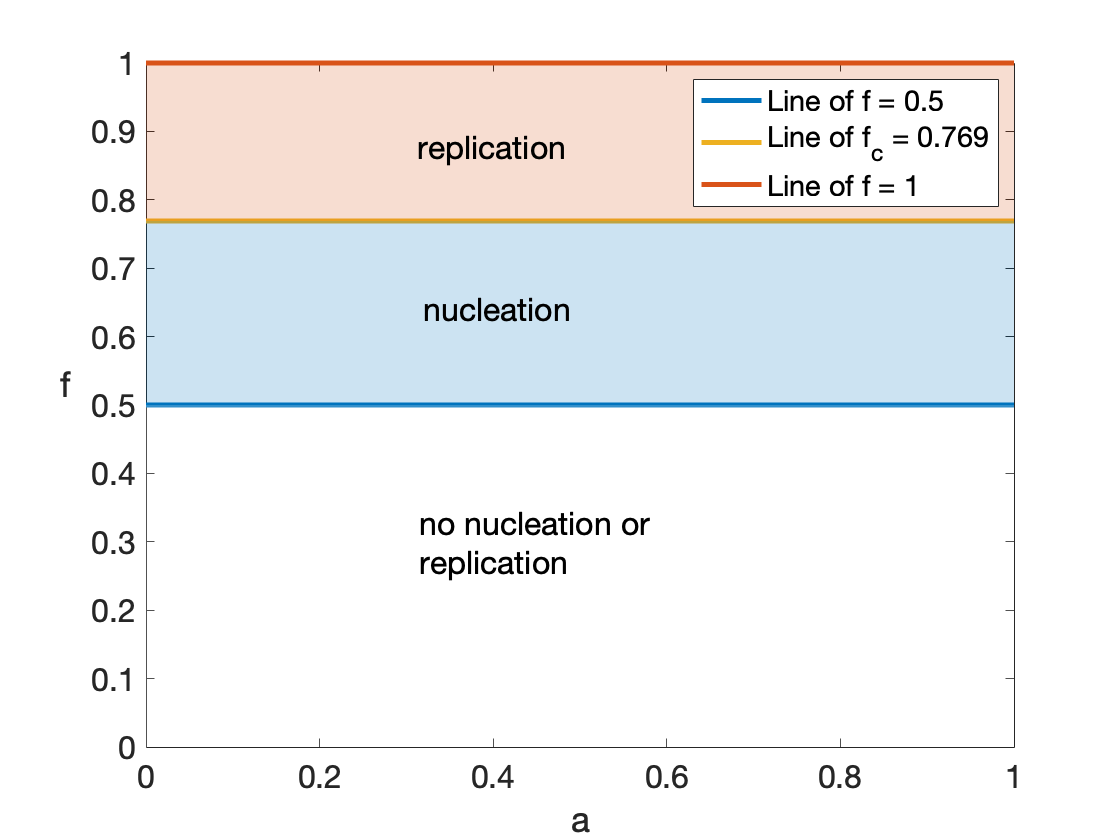}
        \caption{Brusselator model}
        \label{fig:Br_phase_diag}
    \end{subfigure}
    \caption{Phase diagrams in parameter space indicating the type of spike-generating mechanism that will occur as the domain half-length $L$ increases for the Schnakenberg (a) and Brusselator (b) models when $\varepsilon/\sqrt{D} \ll 1$. See \S \ref{sec:sch}--\S \ref{sec:no_nucleation} for the details. The solid blue line $a_c(b)$ in (a), computed numerically from \eqref{sch:integral_Bfinal} for $\varepsilon = 0.01$ and $D = 2$, agrees closely with \eqref{sch:ac_approx}, which neglects the $\varepsilon/\sqrt{D}$ term. For the Brusselator in (b), the thresholds in $f$ are independent of the parameter $a$ when $\varepsilon/\sqrt{D} \ll 1$.}
    \label{fig:phase_diag}
\end{figure}

Recent work \cite{Nico,FahadWoods,FahadGrayScott,Degenerate} has considered stable spikes in a class of RD systems including \eqref{Schnakenberg}, \eqref{Brusselator} and \eqref{GM} as part of a wider zoo of localized patterns for the infinite domain problem. Here, semi-strong interaction analysis predicts single-spike solutions that have monotonically decaying exponential tails. The limit of their existence is a saddle-node bifurcation, which can be continued numerically into the broader {\em homoclinic snaking} regime where complex localized patterns occur that have oscillatory decaying tails \cite{Beck,kno}. As we shall see, the parameter boundary we find here beyond which there is no nucleation or self-replication of the single spike on the growing domain, corresponds precisely to this saddle-node boundary of spikes of the infinite-domain problem. We shall return to this equivalence in the Discussion (\S \ref{sec:discussion}) below, and how the results here might extend beyond the semi-strong interaction regime towards the snaking region.

\subsection{Outline}

The rest of this paper is organized as follows. Treating the domain half-length $L$ as a static parameter, in \S \ref{sec:sch} we use the method of matched asymptotic expansions in the limit $\varepsilon_L \to 0$ to construct quasi-steady state spike solutions for \eqref{Sch_core}. In our analysis, we will derive critical values for the domain half-length $L$ where either spike self-replication or spike nucleation will occur depending on the parameter range of $a$ and $b$. In \S \ref{sec:bruss} we present a similar analysis for the Brusselator model, which can undergo either spike self-replication or spike nucleation events as the domain grows when $1/2 < f < 1$. Section \ref{sec:gm} then proceeds to analyze the possibility of spike nucleation for the GM model \eqref{GM_Lagrange}, which is shown to occur as $L$ exceeds a certain threshold only when $0 < \kappa < 1$. For all three RD models, in \S \ref{sec:no_nucleation} we will show that neither spike self-replication nor spike nucleation can occur in the semi-strong regime whenever there is a spatially homogeneous equilibrium state that is within the range of well-posedness for the nonlinear boundary value problem that is formulated in the outer region between adjacent spikes. Section \ref{sec:transitions} then turns to numerical results to illustrate the global bifurcation mechanism for transitions to patterns of increased spatial complexity. This is done by superimposing time-dependent PDE simulation results for each of our three RD models on the branches of spike equilibria computed using numerical continuation in $L$. The continuation is performed in \textit{pde2path} \cite{pde2path}, which can also analyse stability and identify bifurcation points. Section \ref{sec:discussion} contains a summary of our results and uses our parameter phase diagrams to theoretically explain previous numerical results \cite{crampin_1999,methods,crampin_2002}, in which spike self-replication or spike nucleation behavior was observed. We also show the connection of our results to what is known about the problem on the infinite domain $L \to \infty$, and briefly discuss open problems in higher spatial dimensions. 

\section{The Schnakenberg model}\label{sec:sch}

In this section, we analyze two distinct types of instabilities of spike patterns for the Schnakenberg model \eqref{Sch_core} that lead to the creation of additional spikes as the domain length increases slowly. One instability is referred to as spike self-replication, where the existence of a quasi-steady-state spike pattern is lost owing to the occurrence of a saddle-node bifurcation associated with an inner or {\em core problem}. The ghost effect of this bifurcation triggers the onset of self-replication events that ultimately lead to the doubling of the number of spikes on the domain. For some range of the parameter $a$ in \eqref{Sch_core}, we will show that this instability first occurs when $D_L \equiv D/L^2$ decreases below a threshold, or equivalently when the domain half-length $L$ in \eqref{Schnakenberg} increases above a threshold. The mechanism underlying spike self-replication is related to that studied previously for the Gray–Scott model \cite{doelman1997pattern,kolokolnikov2005existence,muratov2000static,muratov2002stability,kaper} on a fixed domain as a feed-rate parameter increases. In \cite{kolokolnikov2005existence}, it was verified that the criteria for spike self-replication that were formulated in \cite{skeleton} are satisfied for the Gray-Scott model in the semi-strong interaction regime.

The other type of instability that we will analyze is referred to as spike insertion or nucleation, where new spikes can emerge near the domain boundary or the midpoint between two adjacent spikes. For some range of the parameter $a$ in \eqref{Sch_core}, we will show that this instability first arises as $D_L$ decreases below a saddle-node point associated with a nonlinear outer problem as the domain half-length $L$ increases. The underlying mechanism for this behavior is found to be closely related to the phenomena of mesa-splitting in RD systems as studied in \cite{mesa_2007}.

We will derive explicit thresholds for these two distinct instabilities in the semi-strong interaction limit $\varepsilon \ll 1$ with $D = \mathcal O(1)$, and will delineate the range of the parameters in \eqref{Sch_core} where they occur. From a mathematical viewpoint, the novelty of the analysis in contrast to that in \cite{kolokolnikov2005existence} is that we now must couple a nonlinear core, or inner, problem for the spike profile to a nonlinear reduced scalar BVP defined in the outer region away from the spike.

\subsection{Asymptotic construction of quasi-equilibria: The core problem}\label{sec:sch_equil}

We begin by constructing a quasi steady-state spike solution for the Schnakenberg model \eqref{Sch_core} on the canonical interval $|x| \leq \ell$ with a single spike centered at $x = 0$. To construct $K$-spike equilibria of \eqref{Sch_core} on the domain $[- 1, 1]$, with equidistantly-spaced spikes, we only need to set $\ell = 1/K$ and perform a periodic extension of the results obtained for the canonical interval $|x| \leq \ell$.

In the inner region near $x = 0$, we introduce the inner variables $y$, $V$ and $U$ as
\begin{equation}\label{sch:inner}
    y = \frac{x}{\varepsilon_L} \, , \qquad v = \frac{\sqrt{D_L}}{\varepsilon_L} \, V(y) \, , \qquad u = \frac{\varepsilon_L}{\sqrt{D_L}} \, U(y) \, .
\end{equation}
In terms of $V$ and $U$, the steady state problem for \eqref{Sch_core} transforms, on $y \geq 0$, to
\begin{equation}\label{rep:sc_main}
    V_{yy} - V + \frac{\varepsilon_L}{\sqrt{D_L}} \, a + U \, V^2 = 0 \, , \qquad U_{yy} + \frac{\varepsilon_L}{\sqrt{D_L}} \, b - U \, V^2 = 0 \, ,
\end{equation}
and, from \eqref{scale}, we observe that
\begin{equation}\label{sch:scale}
    \frac{\varepsilon_L}{\sqrt{D_L}} = \frac{\varepsilon}{\sqrt{D}} \, .
\end{equation}
Since we will seek even solutions, we consider the half-line $0 \leq y < \infty$. Upon substituting the expansions
\begin{equation}\label{sch:expan}
    V = V_0 + \frac{\varepsilon}{\sqrt{D}} \, V_1 + \ldots \, , \qquad U = U_0 + \frac{\varepsilon}{\sqrt{D}} \, U_1 + \ldots \, ,
\end{equation}
into \eqref{rep:sc_main} we obtain, to leading order, the following coupled {\em core problem} that characterizes the spike profile:
\begin{equation}\label{sch:rep_core}
    V_{0yy} - V_0 + U_0 \, V_0^2 = 0 \, , \qquad U_{0yy} - U_0 \, V_0^2 = 0 \, , \quad y \geq 0 \, .
\end{equation}
This core problem is identical to that derived in \cite{kolokolnikov2005existence} for the Gray-Scott model. We will seek to path-follow the branch of solutions to \eqref{sch:rep_core} that has a single local maximum in $V_0$ and satisfies the symmetry condition $V_{0y}(0) = U_{0y}(0) = 0$. At the next order, the problem for $V_1$ and $U_1$ on $y \geq 0$ becomes
\begin{equation}\label{sch:u1v1}
    V_{1yy} - V_1 + a + 2 U_0 V_0 V_1 + U_1 V_0^2 = 0 \, , \qquad U_{1yy} + b - 2 U_0 V_0 V_1 - U_1 V_0^2 = 0 \, .
\end{equation}
To parameterize solution branches of \eqref{sch:rep_core} we define $B$ as
\begin{equation}\label{B1}
    B \equiv \lim_{y \to \infty} U_{0y} = \int_0^\infty U_0 V_0^2 \, dy \, .
\end{equation}
The far-field behavior of \eqref{sch:rep_core} is that $V_0 \to 0$ exponentially as $y \to \infty$ and
\begin{equation}\label{sch:far_u0}
  U_0 \sim B y + C_s + o(1) \, , \quad \mbox{as} \quad y \to \infty \, .
\end{equation}
The constant $C_s$ in \eqref{sch:far_u0}, depending on $B$ and the specific solution branch, must be computed numerically from the solution to \eqref{sch:rep_core}. For \eqref{sch:u1v1}, the far-field behavior is
\begin{equation}\label{sch:far_u1}
    V_1 \sim a \, , \quad U_1 \sim - {b y^2/2} \, , \quad \mbox{as} \quad y \to \infty \, .
\end{equation}
By using the numerical bifurcation software \textit{pde2path} \cite{pde2path}, we plot the bifurcation diagram of \eqref{sch:rep_core} for single-spike solutions as $B > 0$ is varied. The result is shown in Fig.~\ref{fig:bif_core}, where the vertical axis is
\begin{equation*}
    \beta \equiv U_0(0) V_0(0) \, .
\end{equation*}
Our results show that, along the upper solution branch, the single-spike solution disappears as a result of a saddle-node bifurcation as $B$ is increased past $B_c \sim 1.347$. Along the lower branch, the solution has an unstable volcano-shaped profile. The saddle-node point $B_c$ signifies the onset of self-replication (cf.~\cite{ward_2005,kolokolnikov2005existence}), and is consistent with the threshold computed in \cite{kolokolnikov2005existence,muratov2000static} for the Gray-Scott model.
   
\begin{figure}[htbp]
    \centering
    \includegraphics[width=0.95\textwidth, height=5.0cm]{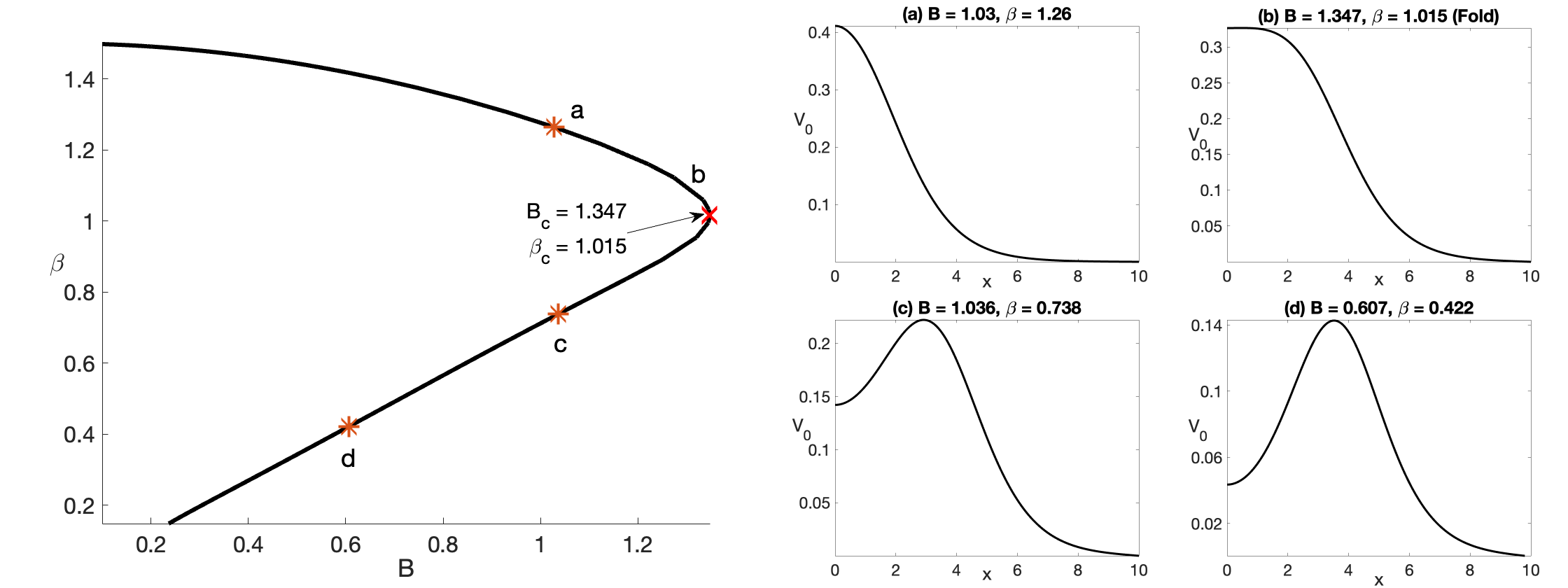}
    \caption{Left panel: Bifurcation diagram of spike solutions for the core problem \eqref{sch:rep_core} obtained using \textit{pde2path} \cite{pde2path}. A saddle-node bifurcation occurs at $B = B_c \approx 1.347$. The single-spike pattern is linearly stable only on the upper branch, which we refer to as the {\em primary branch}. Right panel: Spike profile of the activator $V_0$ at the four indicated points shown in the left panel. The profile has a volcano shape on the unstable lower branch.}
    \label{fig:bif_core}%
\end{figure}

In Fig.~\ref{fig:Sch_Cs} we plot $C_s$ and $U_0(0)$ versus $B$ along the upper branch of the bifurcation diagram in Fig.~\ref{fig:bif_core}, as computed from the numerical solution of the core problem \eqref{sch:rep_core} with far-field behavior \eqref{sch:far_u0}. We observe that $C_s$ is positive on this branch.

To determine the linear stability of the two solution branches in the left panel of Fig.~\ref{fig:bif_core}, in Appendix \ref{app:stab} we derive the following eigenvalue problem on $0<y<\infty$ that characterizes instabilities that can occur on an order $\mathcal O(1)$ time-scale:
\begin{subequations}\label{schnak:eig_prob}
    \begin{align}
        \Phi_{0yy} - \Phi_0 + 2 U_0 V_0 \Phi_0 + V_0^2 N_0 &= \lambda \Phi \,; \qquad \Phi_{0y}(0)=0 \,, \quad \Phi_0\to 0 \quad \mbox{as} \quad y \to \infty \,,
        \\
        N_{0yy} - 2 U_0 V_0 \Phi_0 - V_0^2 N_0 &=0 \,; \qquad N_{0y}(0) = 0 \,, \quad N_{0y} \to 0  \quad \mbox{as} \quad y \to \infty\,.
    \end{align}  
\end{subequations}
From a numerical computation of the dominant eigenvalues of \eqref{schnak:eig_prob}, in Fig.~\ref{fig:sch_eig} we show that the primary solution branch in the left panel of Fig.~\ref{fig:bif_core} is linearly stable on an $\mathcal O(1)$ time-scale, while the other branch is unstable. At the saddle-node bifurcation point $B = B_c$, the neutral mode for \eqref{schnak:eig_prob} where $\lambda=0$ is simply $\Phi_0 = V_{0\beta}$ and $N_0 = U_{0\beta}$, as seen by differentiating the core problem \eqref{sch:rep_core} with far-field behavior \eqref{sch:far_u0} with respect to $\beta \equiv U_0(0) V_0(0)$ and using $dB/d\beta = 0$ at $\beta = \beta_c$. We observe from Fig.~\ref{fig:sch_dimple} that $V_{0\beta}$ has a dimple-shape, which is a necessary condition for the initiation of spike self-replication as formulated in \cite{skeleton} and further studied in \cite{ward_2005,kolokolnikov2005existence}.

\begin{figure}[htbp]
    \centering
    \includegraphics[width=0.55\textwidth, height=5.0cm]{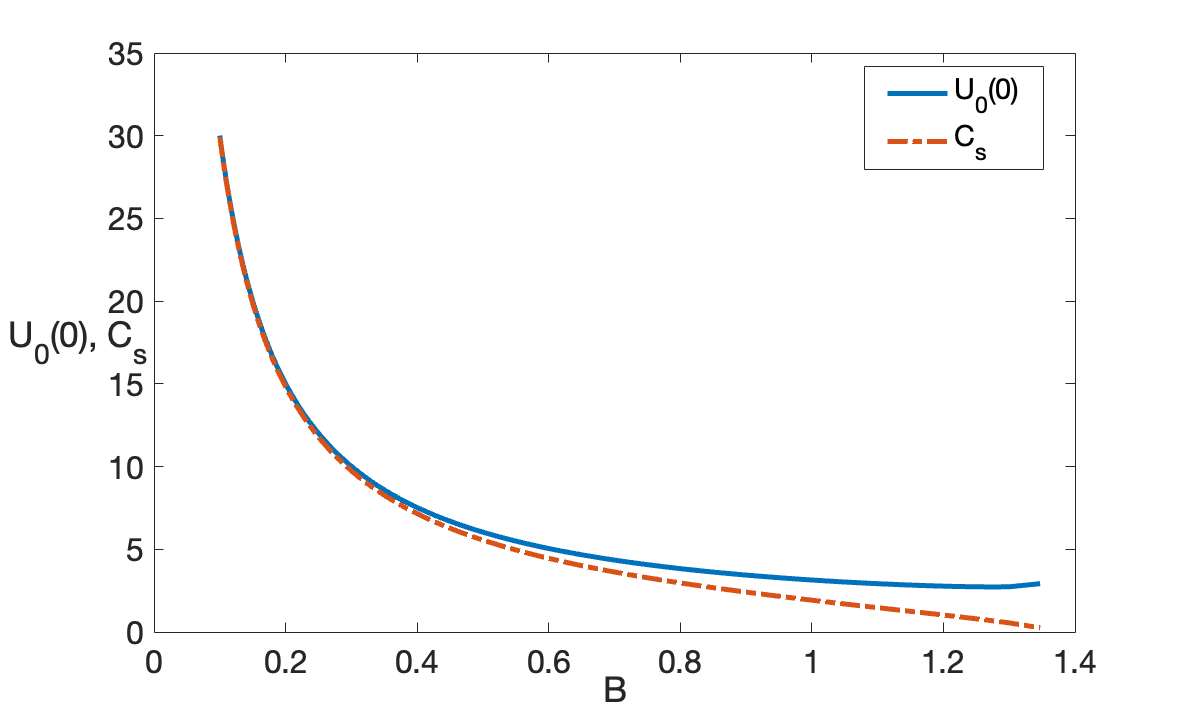}
    \caption{Plots of $C_s$ and $U_0(0)$ as $B$ varies on the upper branch of the bifurcation diagram shown in Fig.~\ref{fig:bif_core}, computed from the numerical solution to the core problem \eqref{sch:rep_core} with \eqref{sch:far_u0}. Observe that $C_s > 0$ up to the fold point $B_c = 1.347$ where $C_{sc} \equiv C_s(B_c) \approx 0.247$.}
    \label{fig:Sch_Cs}
\end{figure}

\begin{figure}[h!tbp]
    \centering
    \begin{subfigure}[b]{0.48\textwidth}  
        \includegraphics[width=\textwidth, height=6.0cm]{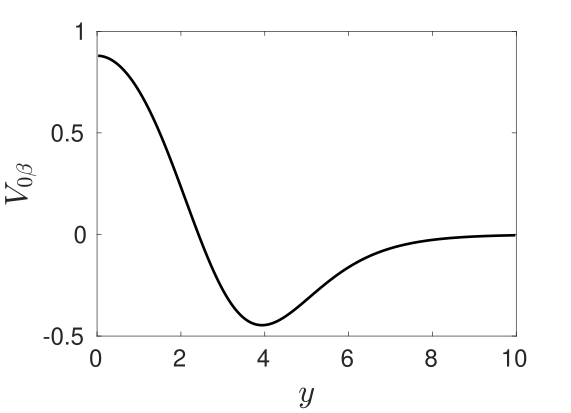}
        \caption{Dimple eigenfunction at the fold point}
        \label{fig:sch_dimple}
    \end{subfigure}
    \begin{subfigure}[b]{0.48\textwidth}
        \includegraphics[width=\textwidth,height=6.0cm]{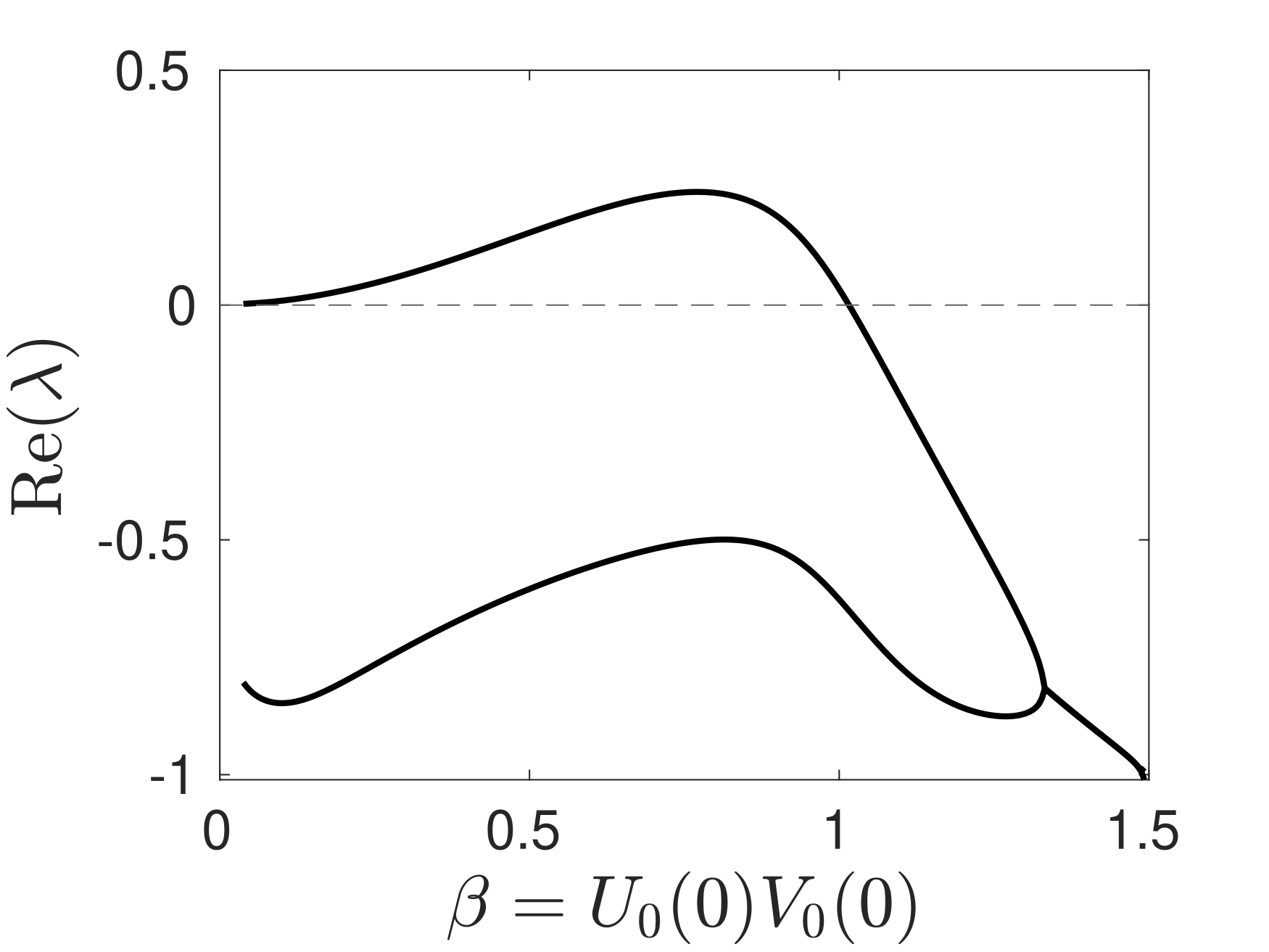}
        \caption{$\mbox{Re}(\lambda)$}
        \label{fig:sch_eig}
    \end{subfigure}
    \caption{Left panel: The dimple eigenfunction component $\Phi = V_{0\beta}$ at the saddle-node bifurcation point $B = B_c \approx 1.347$ that corresponds to $\lambda = 0$. Right panel: $\mbox{Re}(\lambda)$ versus $\beta \equiv U_0(0) V_0(0)$ for the numerically computed dominant eigenvalues of \eqref{schnak:eig_prob}. We observe that $\mbox{Re}(\lambda) < 0$ along the primary solution branch where $\beta_c<\beta<1.5$ with $\beta_c\approx 1.015$. In contrast, the lower branch in the left panel of Fig.~\ref{fig:bif_core}, corresponding to the range $\beta<\beta_c$, is unstable owing to a real positive eigenvalue for \eqref{schnak:eig_prob}.}
    \label{fig:sch_dimple_eig}
\end{figure}

\subsection{Asymptotic construction of quasi-equilibria: The outer solution}\label{sec:sch_equil_out}

To determine $B$ in terms of the parameters of \eqref{Sch_core}, we must match the far-field behavior of the inner solution to an outer solution valid on $\mathcal{O}(\varepsilon_L) \ll |x|<l$. By symmetry, we can extend this outer solution to be even on $- \ell < x < \ell$. Our focus will be to match the inner and outer solutions along the linearly stable primary branch of Fig.~\ref{fig:bif_core} as $B$ is varied on $0 < B < B_c$.

In the outer region, the outer solution for \eqref{Sch_core} on $0^+ < x < \ell$ satisfies
\begin{subequations}
    \begin{align}
        - v +  a + u v^2 &= 0 \,, \label{Sch_out3}
        \\
        D_L u_{xx} + b - u v^2 &= 0 \, , \quad u_x(\ell) = 0 \, . \label{Sch_out4}
    \end{align}
\end{subequations}
From \eqref{Sch_out3}, we obtain
\begin{equation}\label{Sch_utov}
    u = \frac{(v - a)}{v^2}, \quad \mbox{for} \quad v > a \, ,
\end{equation}
which yields
\begin{equation}\label{Sch_ux2}
    u_x = \frac{(2 a - v)}{v^3} \, v_x \, .
\end{equation}
Moreover, by using \eqref{Sch_utov} and \eqref{Sch_ux2} in \eqref{Sch_out4} we obtain
\begin{subequations}\label{Sch_fullouter}
    \begin{equation}\label{Sch_nonlinear}
        D_L \, \left(g(v) \, v_x\right)_x = R_{s}(v) \, , \quad
        0^+ < x < \ell \, ; \quad v_x(\ell) = 0 \, ,
    \end{equation}
    where
    \begin{equation}
          g(v) \equiv \frac{(2 a - v)}{v^3} \, , \qquad R_s(v) \equiv v - a - b \, . \label{Sch_fR}
    \end{equation}
\end{subequations}
For the well-posedness of \eqref{Sch_fullouter} we need $u > 0$ and $u_x > 0$ on $0^+ < x < \ell$, which requires that $a < v(x) < 2 a$ on $0^+ < x < \ell$.

To derive the matching condition between the inner and outer solutions we first use \eqref{sch:inner}, \eqref{sch:expan}, \eqref{sch:far_u0} and \eqref{sch:far_u1} to obtain that the outer solution must satisfy
\begin{equation} \label{sch:match_u}
    u(x) \sim \frac{B x}{\sqrt{D_L}} + \frac{C_s \varepsilon}{\sqrt{D}} \, , \qquad v(x) \sim a + o(1) \, , \quad \mbox{as} \quad x \to 0^+ \, ,
\end{equation}
where $C_s = C_s(B)$. By using \eqref{Sch_ux2} and \eqref{Sch_fR}, and calculating $u_x(0^+)$ from \eqref{sch:match_u}, we conclude that
\begin{equation}\label{sch:match_1}
    \lim_{x \to 0^+} u_x =  \lim_{x \to 0^+} g(v) v_x = \frac{B}{\sqrt{D_L}} \, .
\end{equation}
Then, by using \eqref{Sch_utov} and \eqref{sch:match_u}, we obtain
\begin{equation}\label{sch:match_1.5}
    u(0^+) =  \frac{v(0^+) - a}{\left[v(0^+)\right]^2} \sim \frac{C_s \varepsilon}{\sqrt{D}} \, ,
\end{equation}
which can be solved for $\varepsilon \ll 1$ to obtain
\begin{equation} \label{sch:match_2}
    v(0^+) \sim a + a^2 C_{s} \frac{\varepsilon}{\sqrt{D}} \, .
\end{equation}
Here $C_s = C_s(B)$ is defined from the far-field of the core problem \eqref{sch:far_u0}. Since $C_s(B) > 0$, as shown in Fig.~\ref{fig:Sch_Cs}, we conclude from \eqref{sch:match_2} that $v(0^+) > a$.

We now discuss the solvability of the outer problem \eqref{Sch_fullouter} subject to the two matching conditions \eqref{sch:match_1} and \eqref{sch:match_2}, which will determine $B$. We first require the following lemma:

\begin{lemma}\label{lemma:sch}
    Suppose that $a < b$. Then, on the range of $x$ where $a < v(x) < 2 a$, we have $R_s(v) < 0$ and, consequently, $dv/dx > 0$. 
\end{lemma}
\begin{proof}
    If $a < b$, we observe from \eqref{Sch_fR} that $R_s(a) = - b < 0$, $R_s(2 a) = a - b < 0$ and $dR_s/dv > 0$. Therefore, $R_s< 0$ on the range $a < v < 2 a$. Next, upon integrating \eqref{Sch_nonlinear}, and imposing $v_x(\ell) = 0$, we obtain on $0 < x < \ell$ that
    \begin{equation}\label{lemma:sch_eq}
        D_L \left[g(v) v_x\right]\vert_x^\ell = - D_L g(v) v_x =
        \int_x^\ell R_s\left[v(\eta)\right] \, d\eta < 0 \, ,
    \end{equation}
    whenever $a < v(x) < 2 a$. Since $g(v) > 0$ for $a < v < 2 a$, we conclude that $v_x > 0$ when $a < v(x) < 2 a$.
\end{proof}
Next, we multiply \eqref{Sch_nonlinear} by $g(v) v_x$ and integrate. Upon using the monotonicity of $v(x)$ on $a < v < 2 a$ together with $v_x(\ell) = 0$, we obtain
\begin{equation}
    \frac{D_L}{2} \, \left[g(v) \, v_x\right]^2 = \int_{\mu}^{v(x)} \, R_{s}(\xi) \, g(\xi) \, d\xi \, , \label{sch:integral}
\end{equation}
where we have labeled the boundary value as $\mu \equiv v(l)$. Setting $x = 0^+$ in \eqref{sch:integral}, and using \eqref{sch:match_1}, we conclude that
\begin{equation}\label{sch:integral_B}
    B^2 = - 2 \int_{v(0^+)}^{\mu} \, R_s(\xi) \, g(\xi) \, d\xi \, ,
\end{equation}
where $v(0^+) \sim a$, but with $v(0^+) > a$ given in \eqref{sch:match_2}. Since $R_s < 0$ when $a < \mu = v(\ell) < 2 a$, we observe that the right-hand side of \eqref{sch:integral_B} is positive on this range.

To proceed further, we define $\mathcal G_s^{\prime}(\xi)$ by
\begin{equation}\label{Gprime_sch}
    \mathcal G_s^{\prime}(\xi)\equiv - R_s(\xi) \, g(\xi) =
    \frac{2 a (a + b)}{\xi^3} - \frac{(3 a + b)}{\xi^2} + \frac{1}{\xi} > 0 \, , \quad \mbox{on} \quad a < \xi < 2a \, ,
\end{equation}
when $a < b$. A first integral of \eqref{Gprime_sch} is given by
\begin{equation}\label{G_Sch}
    \mathcal G_s(\xi) = - \frac{a (a + b)}{\xi^2} + \frac{(3 a + b)}{\xi} + \log{\xi} \, .
\end{equation}
In this way, \eqref{sch:integral_B} reduces to the explicit relation
\begin{equation}\label{sch:integral_Bfinal}
    B^2 = 2 \left[\mathcal G_s(\mu) - \mathcal G_s(v(0^+)\right] \, , \quad \mbox{with} \quad v(0^+) \sim a + a^2 C_s \frac{\varepsilon}{\sqrt{D}} \, .
\end{equation}
Since $\mathcal G_s^{\prime}(\mu) > 0$ on $a < \mu < 2 a$, it follows that $B$ is an increasing function of $\mu = v(\ell)$.

By Lemma \ref{lemma:sch}, on the range of $x$ where $a < v < 2 a$, we use \eqref{Gprime_sch} in \eqref{sch:integral} to obtain that
\begin{equation} \label{vx_sch}
    g(v) \, v_x = \sqrt{\frac{2}{D_L}} \, \sqrt{\mathcal G_s(\mu) - \mathcal G_s(v(x))} \, .
\end{equation}
By integrating this separable ODE, we get an implicit relation for $v(x)$ on $0^+ < x < \ell$ given by
\begin{equation}\label{sepa_Sch}
    \chi_s\left[v(x)\right] = \sqrt{\frac{2}{D_L}} \, x \, , \qquad \mbox{where} \qquad \chi_s\left[v(x)\right] \equiv \int_{v\left(0^+\right)}^{v(x)} \frac{g(\xi)}{\sqrt{\mathcal G_s(\mu) - \mathcal G_s(\xi)}} \, d\xi \, .
\end{equation}
Then, by setting $x = \ell$ and  $v(\ell) = \mu$ in \eqref{sepa_Sch}, we obtain an implicit equation for $\mu$ given by
\begin{equation}\label{chi2_Sch}
    \chi_s(\mu) = \int_{v\left(0^+\right)}^\mu \frac{g(\xi)}
    {\sqrt{\mathcal G_s(\mu) - \mathcal G_s(\xi)}} \, d\xi =
    \sqrt{\frac{2}{D_L}} \ell \, ,
\end{equation}
where $v(0^+)$, with $v(0^+) > a$ given by \eqref{sch:match_2}.

Since the integral in \eqref{sepa_Sch} is improper at $\xi = \mu$, to obtain a more tractable formula for $\chi_s(\mu)$ we integrate the expression in \eqref{chi2_Sch} by parts by using $g(\xi) = -\mathcal G_s^{\prime}(\xi)/R_s(\xi)$. This yields the proper integral
\begin{equation}\label{properchi_sch}
    \chi_s(\mu) = - 2 \frac{\sqrt{\mathcal G_s(\mu) - \mathcal G_s(v(0^+))}}{R_s(v(0^+))} + 2 \int_{v(0^+)}^\mu \, \frac{\sqrt{\mathcal G_s(\mu) - \mathcal G_s(\xi)}}{\left[R_s(\xi)\right]^2} \, R_s^{\prime}(\xi) \, d\xi \, ,
\end{equation}
where $R_s(\xi) = \xi - a - b$ and $\mathcal G_s(\xi)$ is defined in \eqref{G_Sch}. On the range $\mu > v(0^+)$, we observe that $\chi_s(\mu)$ is positive since $R_s(\xi) < 0$ and $R_s^{\prime}(\xi) > 0$ on $a < \xi < 2 a$. Moreover, by differentiating \eqref{properchi_sch} we conclude that
\begin{equation}\label{properchip_sch}
    \chi_s^{\prime}(\mu) = \frac{g(\mu)}{\sqrt{\mathcal G_s(\mu) - \mathcal G_s(v(0^+))}} \frac{R_s(\mu)}{R_s(v_0^+)} + \mathcal G_s^{\prime}(\mu) \int_{v(0^+)}^{\mu} \frac{R_s^{\prime}(\xi)}{\sqrt{\mathcal G_s(\mu) - \mathcal G_s(\xi)} \left[R_s(\xi)\right]^2} \, d\xi \, .
\end{equation}
Since $g(\xi) > 0$, $\mathcal G_s^{\prime}(\xi) > 0$ and $R_s^{\prime}(\xi) > 0$ on $a < \xi < 2 a$, we conclude that $\chi_s(\mu)$ is a monotone increasing function of $\mu$ on $v(0^+) < \mu < 2 a$ that reaches its maximum value at $\mu = 2 a$. As such, recalling that $\mu = v(\ell)$, we define
\begin{equation}\label{sch:xmax}
    \chi_{s, \max} \equiv \chi_s(2 a) \, , \quad \mu_{\max} \equiv 2a \, .
\end{equation}
We summarize our asymptotic construction of a one-spike solution to \eqref{Sch_core} as follows:

\begin{result}
On the range $0 < a < b$, our asymptotic construction of a one-spike solution to \eqref{Sch_core} on $|x| \leq \ell$ when $\varepsilon_L \ll 1$ reduces to solving the coupled nonlinear algebraic system
\begin{equation}\label{sch:nas}
    \chi_s(\mu) = \sqrt{\frac{2}{D_L}} \ell \, , \qquad
    B^2 = 2 \left[\mathcal G_s(\mu) - \mathcal G_s(v(0^+)\right]\,,
\end{equation}
for $B$ and $\mu = v(\ell)$ in terms of the parameters $\ell/\sqrt{D_L}$, $a$, $b$ and $\varepsilon/\sqrt{D}$. Here, $v(0^+)$ is given in \eqref{sch:match_2}, while $\chi_s(\mu)$ and $\mathcal G_s(\xi)$ are defined in \eqref{properchi_sch} and \eqref{G_Sch}, respectively. For a $K$-spike solution on $- 1 \leq x \leq 1$, we identify that $\ell = 1/K$.
\end{result}

Our asymptotic theory predicts that spike self-replication will occur for the range of the parameter $a$ where \eqref{sch:nas} has a solution with $B > B_c \approx 1.347$ but $\mu < 2 a$. On such a range for $a$, we predict that spike-replication occurs when
\begin{equation}\label{sch:thresh_repD}
    D_L < D_{L, K}^{rep} \equiv \frac{2}{K^2 \left[\chi_s(\mu)\right]^2} \, .
\end{equation}
Here, $\mu < 2a$, is calculated numerically from solving \eqref{sch:integral_Bfinal} with $B = B_c\approx 1.347$. Equivalently, in terms of the domain half-length $L$, we use $D_L={D/L^2}$ to predict that spike self-replication occurs when
\begin{equation}\label{sch:thresh_repL}
    L > L_K^{rep} \equiv \sqrt{\frac{D}{2}} K \chi_s(\mu) \, .
\end{equation}
Alternatively, spike nucleation near $x = \pm \ell$ is predicted to occur for the range of $a$ where \eqref{sch:integral_Bfinal} has a solution with $\mu = 2 a$ and $B < B_c \approx 1.347$. For this range of $a$, we predict that spike nucleation occurs when
\begin{equation}\label{sch:thresh_nucD}
    D_L < D_{L, K}^{nuc} \equiv \frac{2}{K^2 \left[\chi_{s, \max}\right]^2} \, , \quad \mbox{where} \quad   \chi_{s, \max} \equiv \chi_s(2 a) \, ,
\end{equation}
or, equivalently, when the domain half-length $L$ is sufficiently large so that
\begin{equation}\label{sch:thresh_nucL}
    L > L_K^{nuc} \equiv \sqrt{\frac{D}{2}} K \chi_{s, \max} \, .
\end{equation}

\begin{figure}[htbp]
    \centering
    \includegraphics[width=0.55\textwidth, height=5.0cm]{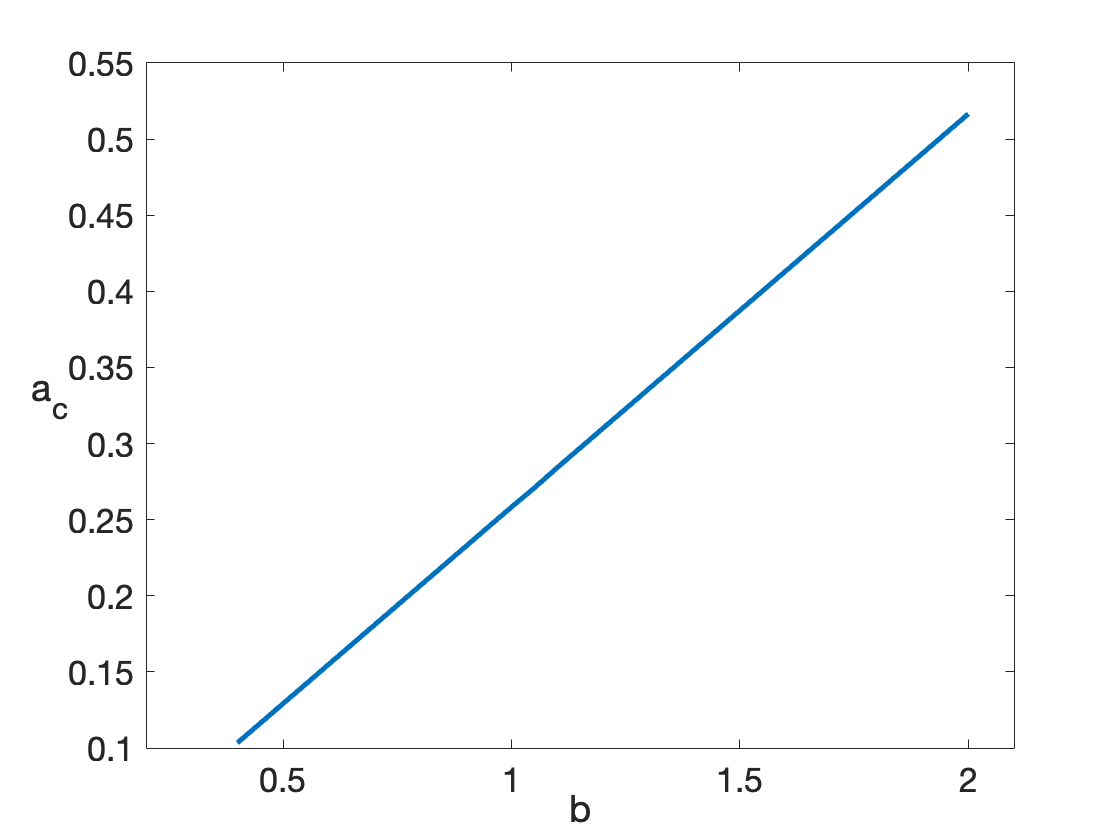}
    \caption{Plot of critical threshold $a = a_c$ versus $b$ separating whether spike self-replication or spike-nucleation will occur, as computed by setting $B=B_c\approx 1.347$ and $\mu = 2 a$ in \eqref{sch:integral_Bfinal}. Parameters: $\varepsilon = 0.01$ and $D = 2$.}
    \label{fig:sch:ac_vs_b}
\end{figure}

The thresholds in \eqref{sch:thresh_repD}--\eqref{sch:thresh_nucL} are calculated by using Newton's method on the nonlinear algebraic system \eqref{sch:nas}. The critical value of $a$, denoted by $a_c$, that distinguishes whether spike-replication or spike nucleation behavior occurs first as $D_L$ is decreased, is obtained by numerically solving \eqref{sch:integral_Bfinal} for $a$, where we set $B = B_c\approx 1.347$ and $\mu = 2 a$. On the range $0 < a < a_c$ spike self-replication occurs, while spike nucleation occurs on the range $a_c < a < b$. Since the threshold $a_c$ is independent of $L$, we conclude that either only spike nucleation events or only spike self-replication events will occur as the domain half-length $L$ increases. In Fig.~\ref{fig:sch:ac_vs_b} we plot the numerically computed threshold $a_c$ versus $b$ for $\varepsilon = 0.01$ and $D = 2$. For the range of $b$ shown, $a_c$ is approximately linear in $b$. To obtain a good approximation for $a_c$, valid for $\varepsilon/\sqrt{D} \ll 1$, we simply set $B = B_c$, $\mu = 2 a$ and $v_0(0^+) = a$ in \eqref{sch:integral_Bfinal} and solve for $a$. By using \eqref{G_Sch} for $\mathcal G_s(\xi)$, we conclude, after some algebra, that
\begin{equation*}
    \frac{B_c^2}{2} \sim \mathcal G_s(2 a) - \mathcal G_s(a) = -\frac{3}{4} + \frac{b}{4 a} + \log{2} \, .
\end{equation*}
By solving for $a$, we get that $a_c=a_c(b)$ is given by
\begin{equation}\label{sch:ac_approx}
    a_c \sim \frac{b}{2 B_c^2 + 3 - 4 \log{2}} \, , \quad \mbox{with} \quad B_c \approx 1.347 \, ,
\end{equation}
which very closely approximates the transition curve shown in Fig.~\ref{fig:sch:ac_vs_b}.

\begin{figure}[h!tbp]
    \centering
    \begin{subfigure}[b]{0.48\textwidth}  
        \includegraphics[width=\textwidth, height=6.0cm]{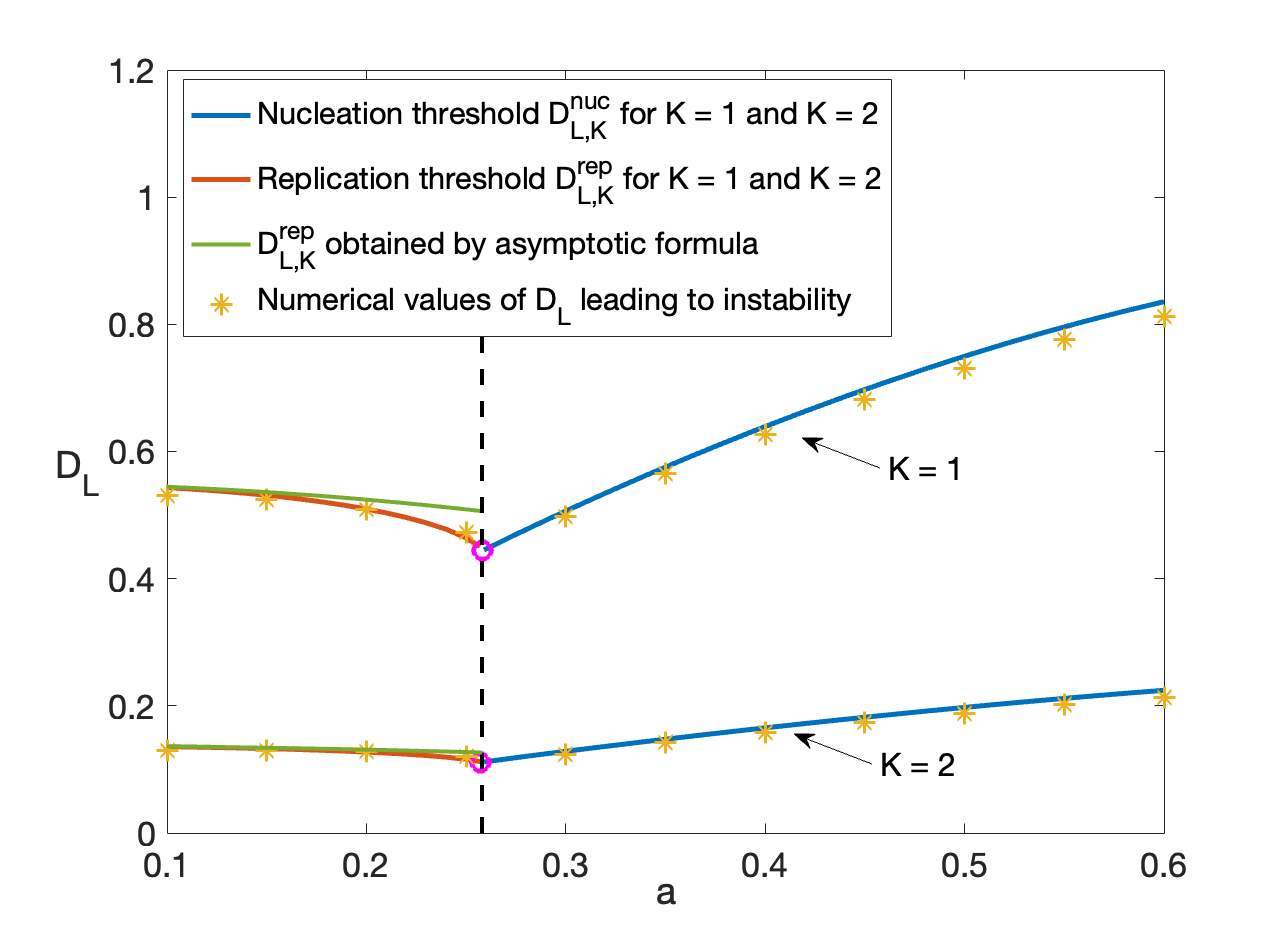}
        \caption{$D_{L, K}^{rep}$ and $D_{L, K}^{nuc}$}
        \label{fig:Sch_Dc}
    \end{subfigure}
    \begin{subfigure}[b]{0.48\textwidth}
        \includegraphics[width=\textwidth, height=6.0cm]{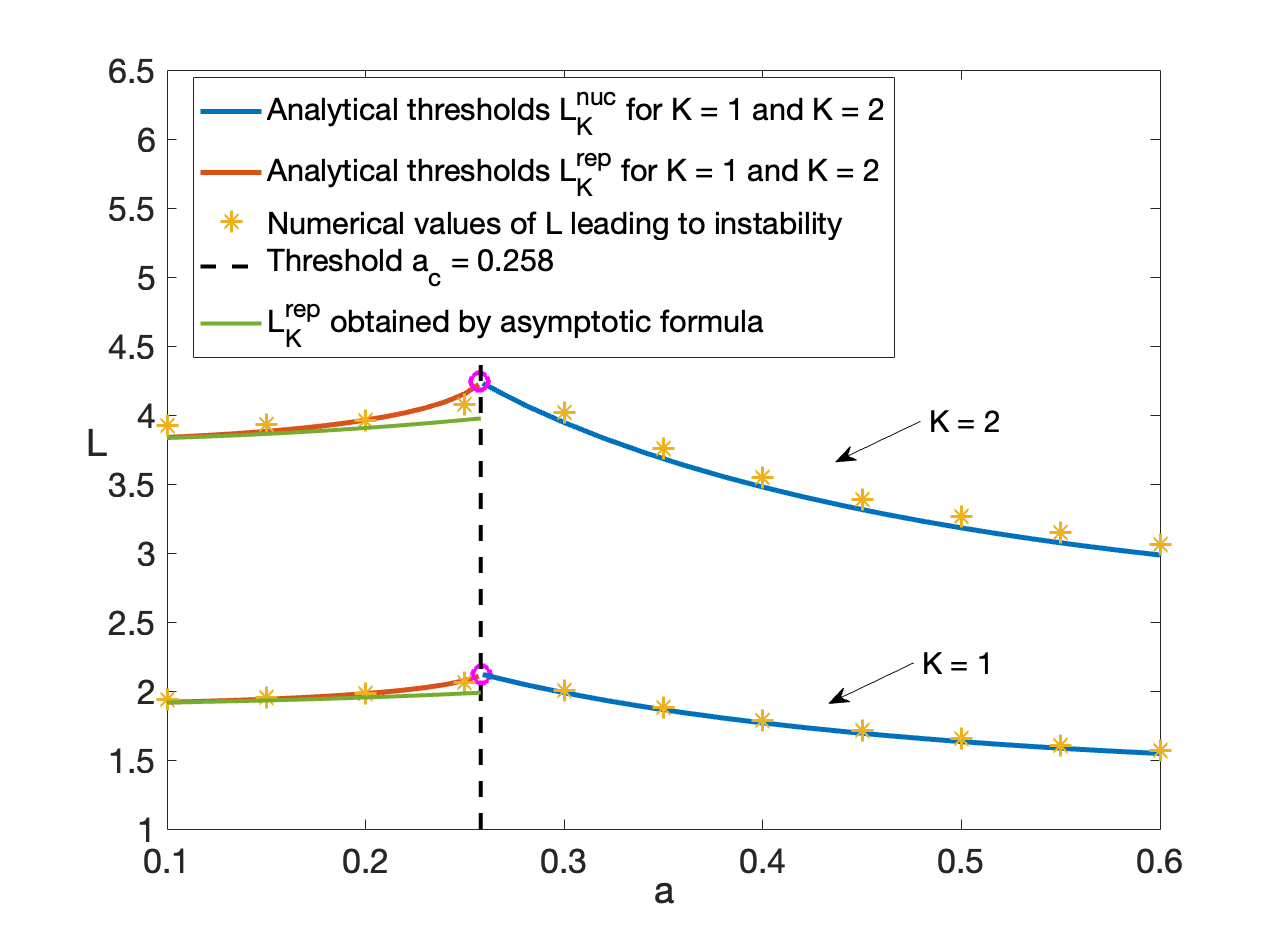}
        \caption{$L_K^{rep}$ and $L_K^{nuc}$}
        \label{fig:Sch_Lc}
    \end{subfigure}
    \caption{Left panel: Comparison between numerical and analytical results of the spike-replication threshold $D_{L, K}^{rep}$ and the spike nucleation threshold $D_{L, K}^{nuc}$ for the Schnakenberg model. The solid curves are the asymptotic results given in \eqref{sch:thresh_repD} and \eqref{sch:thresh_nucD}, based on the nonlinear algebraic system \eqref{sch:nas}. The yellow stars are obtained by full numerical simulations of the Schnakenberg model \eqref{Sch_core} using \textit{FlexPDE} \cite{flexpde2015}. The pink circle is the threshold $a_c \approx 0.258$. Parameters: $\varepsilon = 0.01$, $b = 1$. Right panel: Same plot but for the critical lengths $L_K^{rep}$ and $L_K^{nuc}$ when $D = 2$, given by the asymptotic results in \eqref{sch:thresh_repL} and \eqref{sch:thresh_nucL}, respectively. In both panels, the solid green curves are the closed-form asymptotic results in \eqref{sch:small_DLfinal} and \eqref{sch:small_Lfinal} for the self-replication thresholds valid for $a \ll 1$.}
    \label{fig:Sch_DcLc}
\end{figure}

For $b = 1$, $\varepsilon = 0.01$ and $D = 2$, and $K = 1, 2$, in Fig.~\ref{fig:Sch_Dc} we plot the spike self-replication threshold $D_{L, K}^{rep}$ and the spike nucleation threshold $D_{L, K}^{nuc}$ versus $a$, as computed from \eqref{sch:thresh_repD} and \eqref{sch:thresh_nucD}, respectively, as $a$ is varied. The computed threshold $a = a_c \approx 0.258$ for $b = 1$ (see Fig.~\ref{fig:sch:ac_vs_b}), separating whether spike self-replication or spike nucleation occurs, is identified in the plot. We conclude that, as $L$ increases, which corresponds to decreasing $D_L = D/L^2$ and $\varepsilon_L = \varepsilon/L$, spike self-replication owing to the non-existence of the core solution is the dominant instability when $a < a_c$. In contrast, on the range $a_c < a < b$, spike nucleation events owing to the non-existence of the outer solution are predicted to be the dominant instability as $D_L$ is decreased. As indicated by the yellow stars in Fig.~\ref{fig:Sch_Dc}, we show a very close agreement between our prediction of the two stability thresholds with the corresponding results computed from full numerical time-dependent simulations of \eqref{Sch_core} using \textit{FlexPDE} \cite{flexpde2015}. Setting $D = 2$, in Fig.~\ref{fig:Sch_Lc} we show a similar plot but in terms of the two critical length thresholds $L_K^{rep}$ and $L_K^{nuc}$ as defined in \eqref{sch:thresh_repL} and \eqref{sch:thresh_nucL}, respectively.

The global bifurcation diagrams shown below in \S \ref{sch:global} confirm that the critical values for the non-existence of the core solution or the non-existence of the outer solution correspond to saddle-node bifurcations for single-spike steady-state solution branches for the full Schnakenberg model \eqref{Sch_core}.

\subsection{Asymptotics of the outer solution for $a$ small}
\label{sec:sch_small_a}

For $a \ll 1$, we now derive an explicit analytical prediction for the spike self-replication threshold. In the limit $a \to 0$, spike nucleation behavior no longer occurs.

When $a \ll 1$, we obtain from \eqref{Sch_out3} that on $0^+ < x < \ell$,
\begin{equation}\label{sch:a-small-v}
    v(x) = a + \mathcal O(a^2) + \ldots \, .
\end{equation}
Upon substituting \eqref{sch:a-small-v} into \eqref{Sch_out4} and recalling the two matching conditions \eqref{sch:match_1} and \eqref{sch:match_1.5}, we obtain that $u(x)$ satisfies the following approximating linear problem for $a \ll 1$:
\begin{subequations}\label{sch_nonlinear_asmall}
    \begin{align}
        D_L \, u_{xx} &= - b + a^2 u \, , \quad 0^+ < x < \ell \, ; \quad u_x(\ell) = 0 \, ,
        \\
        u(0^+) &= C_s \frac{\varepsilon}{\sqrt{D}} \, , \quad
        u_x(0^+) = \frac{B}{\sqrt{D_L}} \, ,
    \end{align}
\end{subequations}
Here, $C_s = C_s(B)$ is obtained from the far-field \eqref{sch:far_u0} of the core problem \eqref{sch:rep_core}.

The solution to \eqref{sch_nonlinear_asmall} with $u_{x}(\ell) = 0$ after imposing the condition for $u(0^+)$ is given by
\begin{equation}\label{sch:v1-solve}
    u = \frac{b}{a^2} + \left(\frac{C_s \varepsilon}{\sqrt{D}} -
    \frac{b}{a^2}\right) \frac{\cosh\left[{a (\ell - |x|)/\sqrt{D_L}}\right]}{\cosh\left(a \ell/\sqrt{D_L}\right)} \, .
\end{equation}
Then, upon satisfying the condition for $u_x(0^+)$ in \eqref{sch_nonlinear_asmall}, we conclude that
\begin{equation}\label{sch:v1-B}
    B = \frac{b}{a} \left(1 - \frac{C_s a^2 \varepsilon}{b \sqrt{D}}\right) \tanh\left(\frac{a \ell}{\sqrt{D_L}}\right) \, .
\end{equation}
Moreover, upon solving \eqref{sch:v1-B} for $D_L$ we obtain, for $a \ll 1$, that
\begin{equation}\label{sch:small_DL1}
    D_L \sim a^2 \ell^2 \left[\tanh^{- 1}\left(\frac{a B}{b \left(1 - C_s {a^2 \varepsilon/(b \sqrt{D})}\right)}\right)\right]^{- 2} \, ,
\end{equation}
where $C_s = C_s(B)$. Setting $B = B_c \approx 1.347$ and $\ell = 1/K$ in \eqref{sch:small_DL1}, and neglecting the $\mathcal O\left(\varepsilon/\sqrt{D}\right)$ correction term, we conclude that the spike self-replication threshold $D_{L, K}^{rep}$ is approximated as
\begin{equation}\label{sch:small_DLfinal}
    D_{L, K}^{rep} \sim \frac{a^2}{K^2} \left[\tanh^{- 1}\left(\frac{a B_c}{b}\right)\right]^{- 2} \, , \quad \mbox{for} \quad a \ll 1 \, ,
\end{equation}
where $B_c \approx 1.347$. In terms of the domain half-length, the spike self-replication threshold is estimated as
\begin{equation}\label{sch:small_Lfinal}
    L_K^{rep} \sim \frac{\sqrt{D} K}{a} \tanh^{- 1}\left(\frac{a B_c}{b}\right) \, , \quad \mbox{for} \quad a \ll 1 \, .
\end{equation}
As indicated by the solid green curves in Fig.~\ref{fig:Sch_DcLc} the approximate thresholds \eqref{sch:small_DLfinal} and \eqref{sch:small_Lfinal} are moderately accurate only when $a$ is small.
  
\subsection{Global bifurcation diagrams}\label{sch:global}

\begin{figure}[htbp]
    \centering
    \includegraphics[width=0.99\textwidth, height=6.0cm]{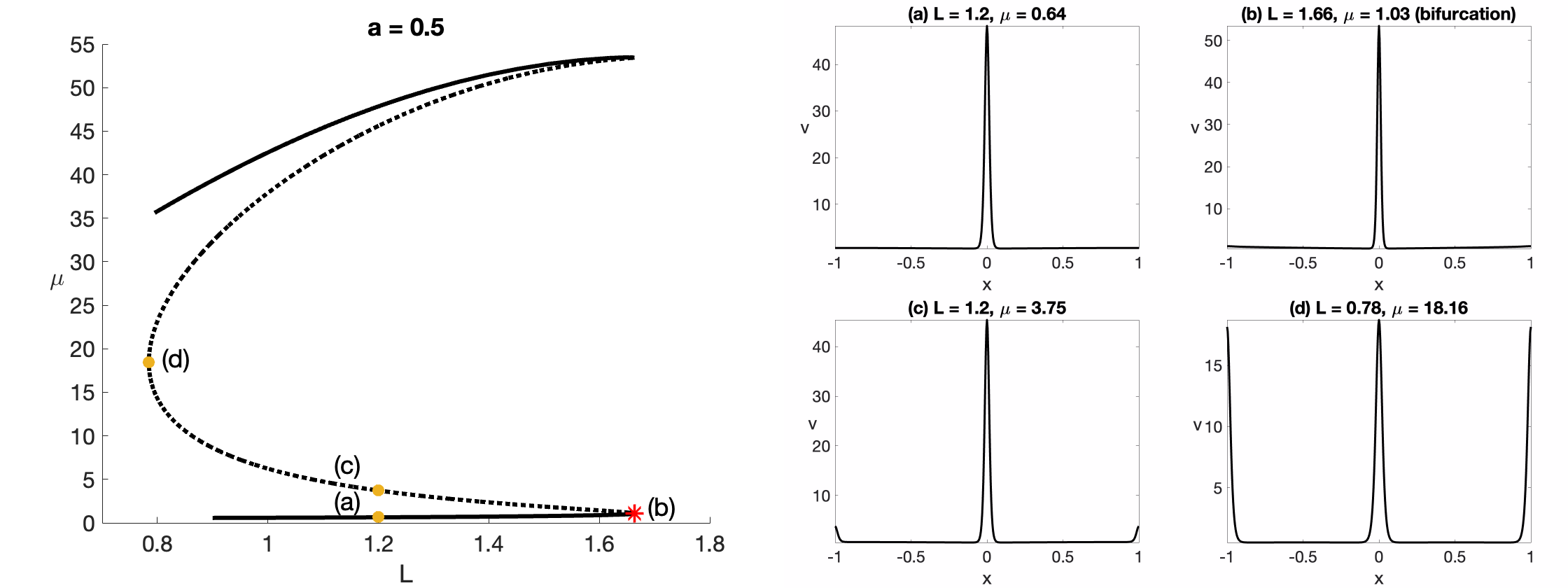}
    \caption{Left panel: Global bifurcation diagram of single-spike steady-states for the Schnakenberg model \eqref{Sch_core} as computed using \textit{pde2path} \cite{pde2path}, where $L$ is the bifurcation parameter and the vertical axis is $\mu = v(1)$. Parameters: $\varepsilon = 0.01$, $a = 0.5$, $b = 1$
    and $D = 2$. Since $a_c < a < b$, where $a_c = 0.258$, we predict that spike nucleation behavior occurs as $L$ is increased starting from point (a) on the lower branch. The red star marks a saddle-node point that is connected to an unstable middle branch that has an interior spike together with two emerging boundary spikes. Solid portions are linearly stable while dashed portions are unstable. Right panel: Spike profile $v(x)$ and bifurcation values (top of subfigures) at the indicated points in the left panel.}
    \label{fig:Sch_bif_nuc}
\end{figure}

Next, we numerically compute global bifurcation diagrams of single-spike steady-state solutions for \eqref{Sch_core} using \textit{pde2path} \cite{pde2path}, where $L$ is taken as the bifurcation parameter and we fix $\varepsilon = 0.01$, $b = 1$ and $D = 2$. When $a = 0.5$ is chosen on the range $a_c < a < b$, where $a_c \approx 0.258$, the global bifurcation diagram of single-spike steady-states shown in Fig.~\ref{fig:Sch_bif_nuc} provides a detailed view on how new spikes are created near the domain boundaries when $L$ approaches the saddle-node bifurcation point, which we denote by $L^{num}$. In this bifurcation diagram, the vertical axis in the left panel of Fig.~\ref{fig:Sch_bif_nuc} is the boundary value $\mu \equiv v(1)$. Solid and dashed portions of the branches indicate linearly stable and unstable solutions, respectively. In the right panel of Fig.~\ref{fig:Sch_bif_nuc} the solution $v(x)$ at the marked points in the left panel is shown. Starting from the bottom solution branch at point (a), this branch that has one interior spike ceases to exist as $L$ increases above $L^{num}$, where $L^{num} \approx 1.66$ as indicated by the red star. We emphasize that this saddle-node point is well-approximated by the value $L_1^{nuc}$ in \eqref{sch:thresh_nucL}, at which the outer solution ceases to exist (see Fig.~\ref{fig:Sch_Lc} for $a = 0.5$ and $K = 1$). Traversing the middle branch on Fig.~\ref{fig:Sch_bif_nuc} starting from the saddle-node point, we observe at point (c) that new boundary spikes are nucleated at the domain endpoints.

We remark that a normal form analysis for \eqref{Sch_core} valid near the saddle-node bifurcation at point (b) can be undertaken in a similar way as done in \cite{gai2024police,tse2016hotspot} for characterizing the nucleation of hotspots for RD systems of urban crime. A closely related normal analysis, and the existence of a dimple-shaped neutral eigenfunction for the linearization at the fold point, were first established in \S 3 of \cite{mesa_2007} in the context of analyzing the onset of self-replication of 1-D mesa patterns in RD systems.

In Fig.~\ref{fig:Sch_bif_rep} we use \textit{pde2path} \cite{pde2path} to obtain a global bifurcation diagram for single-spike steady-states of \eqref{Sch_core} for the value $a = 0.2 < a_c$, where we predict that spike self-replication will occur as $L$ is increased. In the left panel of Fig.~\ref{fig:Sch_bif_rep}, the vertical axis is now $u(0) v(0)$, which is a good measure of the spike amplitude. The computed saddle-node value is $L^{num} \approx 1.99$, which is well approximated by the critical threshold $L_1^{rep}$ in \eqref{sch:thresh_repL}, for which the core problem \eqref{sch:rep_core} with far-field behavior \eqref{sch:far_u0} has a saddle-node point when coupled to the outer solution (see Fig.~\ref{fig:Sch_Lc} for $a = 0.2$ and $K = 1$).

\begin{figure}[htbp]
    \centering
    \includegraphics[width=0.99\textwidth, height=6.0cm]{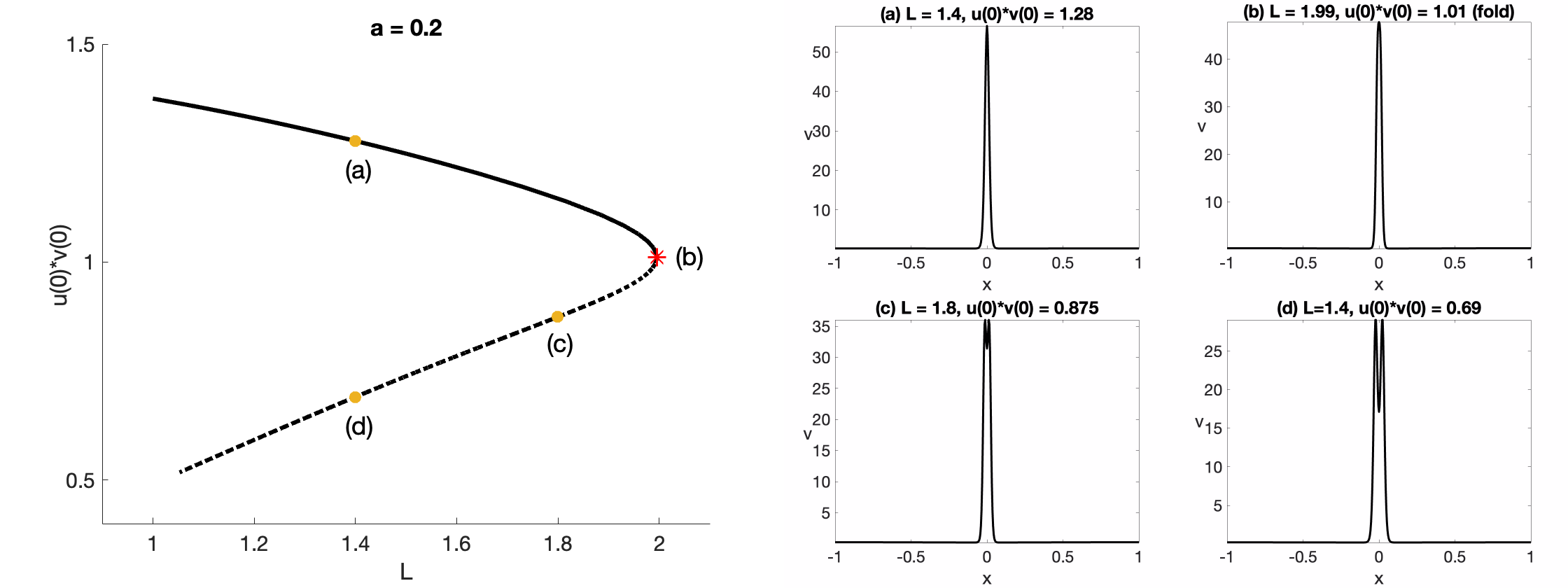}
    \caption{Left panel: Same plot and parameters as in Fig.~\ref{fig:Sch_bif_nuc} except that $a = 0.2 < a_c$. The vertical axis is now $u(0) v(0)$, which better measures the spike amplitude. Since $a < a_c$, we predict that spike self-replication will occur as $L$ is increased starting from point (a) on the upper branch. The red star marks a saddle-node point that is a signature for spike self-replication. The linearly stable upper branch is connected to an unstable lower branch where the spike profile has a volcano shape. Observe that the global bifurcation diagram has a qualitatively similar form to that for the core problem in Fig.~\ref{fig:bif_core}. Right panel: Spike profile $v(x)$ and bifurcation values (top of subfigures) at the indicated points in the left panel.}
    \label{fig:Sch_bif_rep}
\end{figure}

\subsection{Validation of asymptotic theory: Full PDE simulations}\label{sec:sh_fullpde}

\begin{figure}[htbp]
    \centering
    \includegraphics[width=0.99\textwidth, height=6.0cm]{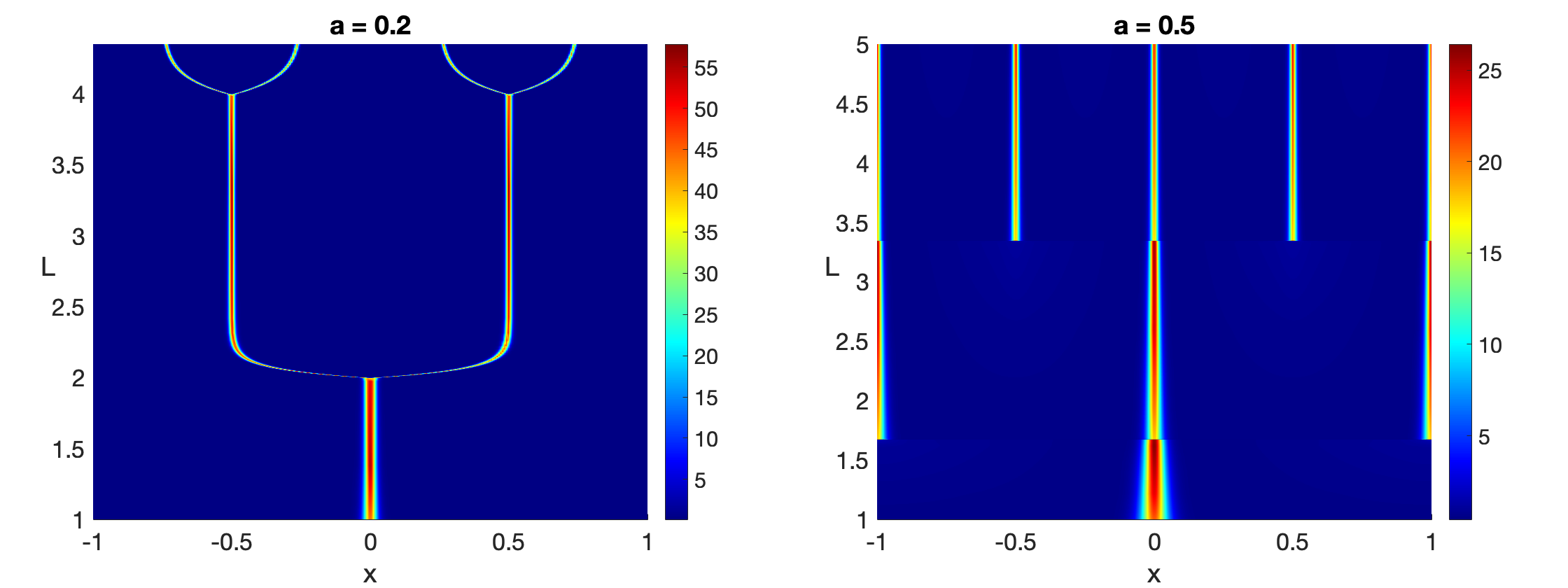}
    \caption{Time-dependent PDE simulations of \eqref{Sch_core} using \textit{FlexPDE} \cite{flexpde2015} illustrating the two spike generation mechanisms as the domain half-length $L$ slowly increases as $L = L_0 e^{\rho t}$ with $L_0 = 1$ and $\rho = 10^{- 4}$. A color plot of the magnitude of the solution $v$ is shown as a function of $L$, which is a proxy measure for time. Parameters: $\varepsilon = 0.01$, $b = 1$ and $D = 2$. Left panel: For $a = 0.2 < a_c \approx 0.258$, spike self-replication dynamics occurs as $L$ increases. Right panel: For $a = 0.5$, where $a_c < a < b = 1$, spike nucleation occurs as $L$ increases. Observe that the transient dynamics that appear immediately after a self-replication or nucleation event occurs on a much shorter time scale than the quasi-steady behavior described by our asymptotic theory.}
    \label{fig:Sch_flexpde_plot}%
\end{figure}

In Fig.~\ref{fig:Sch_flexpde_plot}, we show full-time-dependent PDE simulations of \eqref{Sch_core} computed using \textit{FlexPDE} \cite{flexpde2015} illustrating the two distinct types of spike generation mechanisms as the domain half-length $L$ slowly increases in time according to $L = L_0 e^{\rho t}$ with $L_0 = 1$ and $\rho = 10^{- 4}$. The vertical axes in Fig.~\ref{fig:Sch_flexpde_plot} is $L$, which is a proxy measure for time. In the left panel of Fig.~\ref{fig:Sch_flexpde_plot}, where $a = 0.2 < a_c = 0.258$, a single spike splits into two as $L$ increases above $L \approx 1.98$. Each of these new spikes splits again as $L$ further increases above $L \approx 3.96$. In contrast, in the right panel, where $a = 0.5 > a_c$, we observe that instead of having spike self-replication, new spikes now nucleate from quiescent regions. This nucleation behavior occurs twice: first at the domain boundaries when $L \approx 1.65$ and then later at the midpoint locations between an interior and a boundary spike as $L$ increases above $L \approx 3.26$. We emphasize that all of these numerically computed transition values of $L$ are well-approximated by the critical thresholds predicted in Fig.~\ref{fig:Sch_Lc} from our asymptotic theory.

\section{The Brusselator model}\label{sec:bruss}

We now perform a similar analysis for the Brusselator model \eqref{Brusselator_Lagrange} showing that, depending on the parameter range of $a$ and $f$, a single-spike pattern can either undergo a spike self-replication instability or a spike nucleation event as the domain half-length $L$ grows. In contrast to the analysis of the Schnakenberg model \eqref{Sch_core}, there has been no previous work in a 1-D setting on analyzing the possibility of spike self-replication behavior for the Brusselator \eqref{Brusselator_Lagrange}.

\subsection{Asymptotic construction of quasi-equilibria: The core problem}\label{sec:bruss_equil}

We first construct a quasi-steady-state spike solution for \eqref{Brusselator_Lagrange} on the canonical interval $|x| \leq \ell$ with a single spike centered at $x = 0$. A $K$-spike pattern on $[- 1, 1]$ is obtained by setting $\ell = 1/K$ as in \S \ref{sec:sch_equil}.

In the inner region near $x = 0$, we introduce the inner variables
$y$, $V$ and $U$ in
\begin{equation}\label{bruss:inner}
    y = \frac{x}{\varepsilon_L} \, , \qquad v = \frac{\sqrt{D_L}}{\varepsilon_L} \, V(y) \, , \qquad u = \frac{\varepsilon_L}{\sqrt{D_L}} \, U(y) \, .
\end{equation}
In terms of $V$ and $U$, the steady problem for \eqref{Brusselator_Lagrange} on $y \geq 0$ becomes
\begin{equation}\label{rep:bruss_main}
    V_{yy} - V + \frac{\varepsilon_L}{\sqrt{D_L}} \, a + f U \, V^2 = 0 \, , \qquad U_{yy} + V - U \, V^2 = 0 \, .
\end{equation}
We expand $U$ and $V$ as in \eqref{sch:expan} and substitute the expansion into \eqref{rep:bruss_main}, where we observe that the scaling relation \eqref{sch:scale} still holds. At leading order, we obtain the Brusselator {\em core problem}
\begin{equation}\label{bruss:rep_core}
    V_{0yy} - V_0 + f U_0 \, V_0^2 = 0 \,, \qquad U_{0yy} + V_0 - U_0 \, V_0^2 = 0 \, , \quad y \geq 0 \, ; \qquad
    U_{0y}(0) = V_{0y}(0) = 0 \, .
\end{equation}
We seek to path-follow a solution branch for \eqref{bruss:rep_core} that has a single local maximum in $V_0$ at $y = 0$. At the next order, the problem for $V_1$ and $U_1$ on $y \geq 0$ is given by
\begin{equation}\label{bruss:u1v1}
    V_{1yy} - V_1 + a + 2 f U_0 V_0 V_1 + f U_1 V_0^2 = 0 \, , \qquad U_{1yy} + V_1 - 2 U_0 V_0 V_1 - U_1 V_0^2 = 0 \, .
\end{equation}
We parameterize solution branches to \eqref{bruss:rep_core} in terms of $B$, which we now redefine as $B = \int_0^\infty \left(U_0 V_0^2 - V_0\right) \, dy$. Upon using \eqref{bruss:rep_core}, we can write $B$ equivalently as
\begin{equation}\label{bruss:B1}
    B \equiv \lim_{y \to \infty} U_{0y} = (1 - f) \int_0^\infty
    U_0 V_0^2 \, dy \, ,
\end{equation}
which shows that $B > 0$ when $0 < f < 1$. For \eqref{bruss:rep_core}, $V_0 \to 0$ exponentially as $y \to \infty$, while $U_0$ has the far-field behavior
\begin{equation}\label{bruss:far_u0}
    U_0 \sim B y + C_b + o(1) \, , \quad \mbox{as} \quad y \to \infty \, ,
\end{equation}
where the constant $C_b = C_b(B, f)$ must be computed numerically along each particular solution branch to \eqref{bruss:rep_core}. For \eqref{bruss:u1v1}, the far-field behavior is
\begin{equation}\label{bruss:far_u1}
    V_1 \sim a \, , \quad U_1 \sim - {a y^2/2} \, , \quad \mbox{as} \quad y \to \infty \, .
\end{equation}
The global bifurcation diagram of $\beta \equiv U_0(0) V_0(0)$ versus $B$ shown in the left panel of Fig.~\ref{fig:Bru_core_bif}, as computed by \textit{pde2path}, has qualitatively the same behavior at each fixed $f$ in $0 < f < 1$ as for the Schnakenberg model shown previously in the left panel of Fig.~\ref{fig:bif_core}. For each fixed $f$ in $0 < f < 1$, the upper linearly stable {\em primary} branch of the bifurcation diagram admits a single spike solution that disappears at a saddle-node bifurcation as $B$ is increased past some critical value $B_c = B_c(f)$. The numerical results for $B_c(f)$ given in Fig.~\ref{fig:Bru_core_bif_c}, show that $B_c$ decreases as $f$ increases. Along each lower branch in Fig.~\ref{fig:Bru_core_bif} the solution has an unstable volcano-shaped profile for $V_0$. For several fixed values of $f$, in Fig.~\ref{fig:Bru_Cb_vs_B} we show the numerically computed $C_b = C_b(B,f)$, defined in \eqref{bruss:far_u0}, along the corresponding upper branches in Fig.~\ref{fig:Bru_core_bif}. We observe that $C_b(B) > 0$.

\begin{figure}[htbp]
    \centering
    \begin{subfigure}[b]{0.48\textwidth}  
        \includegraphics[width=\textwidth, height=5.0cm]{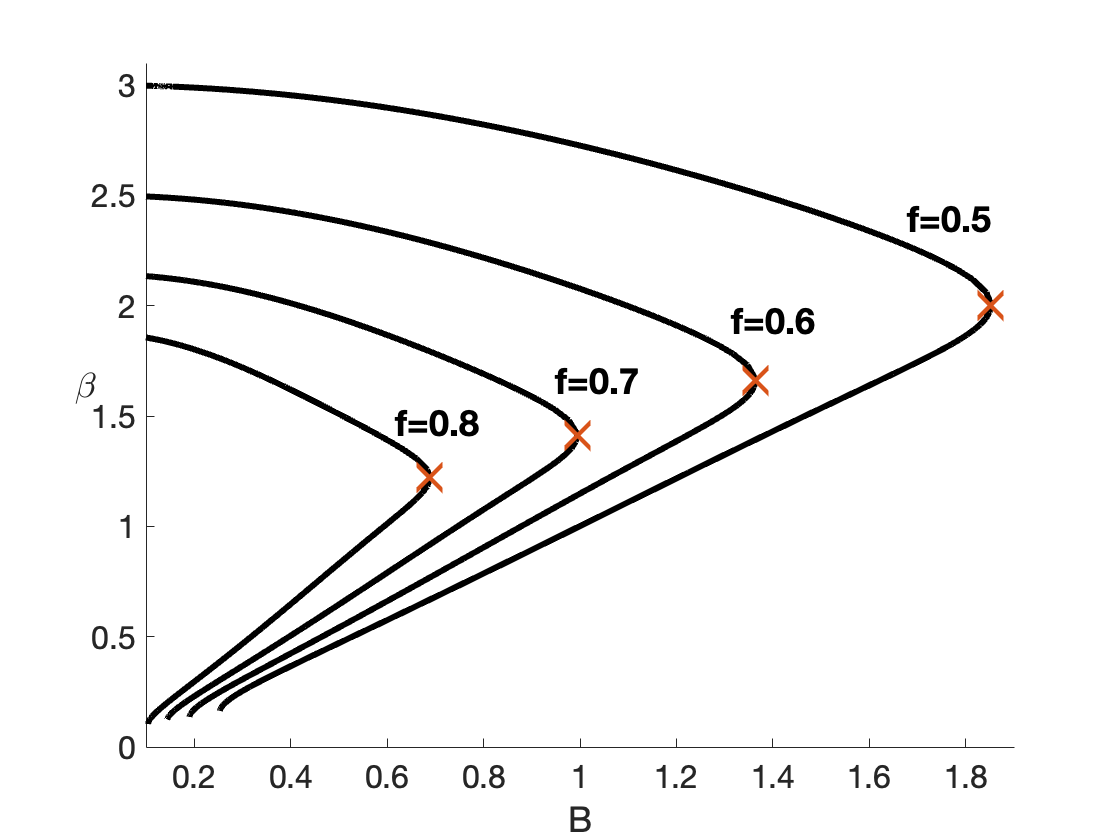}
        \caption{$U_0(0)V_0(0)$ versus $B$}
        \label{fig:Bru_core_bif}
    \end{subfigure}
    \begin{subfigure}[b]{0.48\textwidth}
        \includegraphics[width=\textwidth, height=5.0cm]{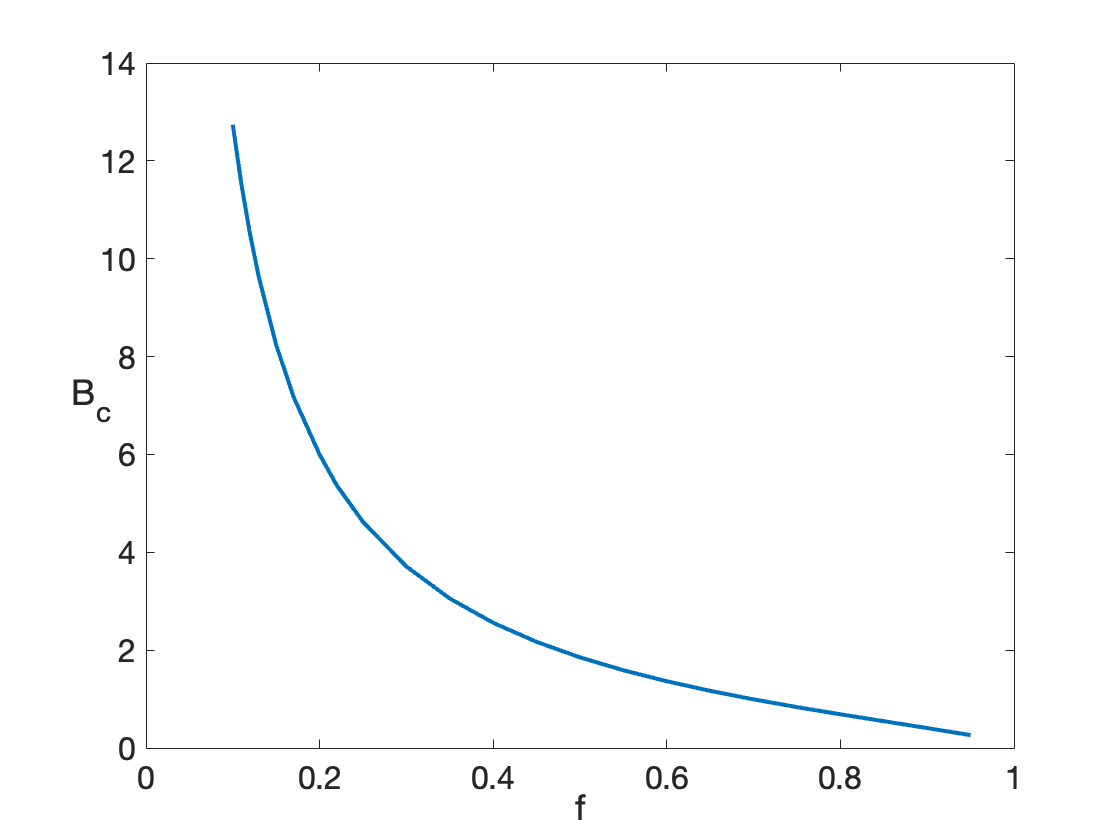}
        \caption{$B_c$ versus $f$}
        \label{fig:Bru_core_bif_c}
    \end{subfigure}
    \caption{Left panel: Bifurcation diagram of $\beta \equiv U_0(0) V_0(0)$ versus $B$ for the Brusselator core problem \eqref{bruss:rep_core} with \eqref{bruss:far_u0} for values of $f$ as indicated. The crosses correspond to the fold points $B = B_c$ near where spike self-replication is predicted. Right panel: Plot of fold points $B_c$ versus $f$ for the core problem.}
    \label{fig:Bru_core_bif_all}
\end{figure}

\begin{figure}[htbp]
    \centering
    \includegraphics[width=0.55\textwidth,height=5.0cm]{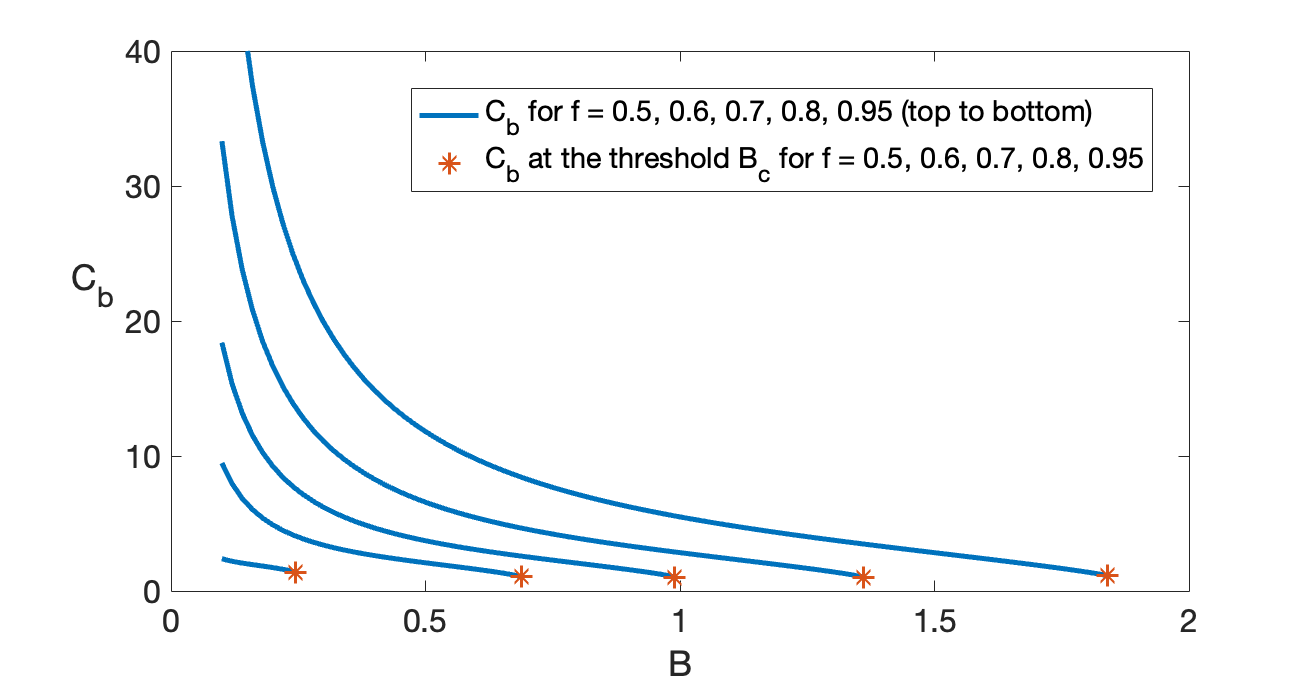}
    \caption{Numerically computed value of $C_b$ as $B$ varies on the upper (primary) branch of the bifurcation diagram shown in the left panel of Fig.~\ref{fig:Bru_core_bif} with a few fixed values of $f$. Note that $C_b > 0$ along the entire branch up to the fold point $B_c$.}
    \label{fig:Bru_Cb_vs_B}%
\end{figure}

To determine the linear stability of the two solution branches in Fig.~\ref{fig:Bru_core_bif} to $\mathcal O(1)$ time-scale instabilities, we must determine the dominant eigenvalues of \begin{subequations}\label{bruss:eig_prob}
    \begin{align}
        \Phi_{0yy} - \Phi_0 + 2f U_0 V_0 \Phi_0 + fV_0^2 N_0 &= \lambda \Phi \,; \qquad \Phi_{0y}(0) = 0 \,, \quad \Phi_0 \to 0 \quad \mbox{as} \quad y \to \infty \,,
        \\
        N_{0yy} + \Phi_0 - 2 U_0 V_0 \Phi_0 - V_0^2 N_0 &=0 \,; \qquad N_{0y}(0) = 0 \,, \quad N_{0y} \to 0 \quad \mbox{as} \quad y \to \infty\,,
    \end{align}  
\end{subequations}
where $U_0$ and $V_0$ satisfy the core problem \eqref{bruss:rep_core}. The derivation of \eqref{bruss:eig_prob} is very similar to that done in Appendix \ref{app:stab} for the Schnakenberg model and is omitted. For $f = 0.8$, in Fig.~\ref{fig:bruss_eig} we show that the primary branch in Fig.~\ref{fig:Bru_core_bif} is linearly stable, while the lower branch is unstable. At the saddle-node point $B = B_c \approx 0.685$ when $f = 0.8$, in Fig.~\ref{fig:bruss_dimple} we show that the neutral mode $\Phi_0 = V_{0\beta}$, corresponding to $\lambda = 0$, has a dimple shape. Qualitatively identical results hold for other values of $0 < f < 1$.

\begin{figure}[h!tbp]
    \centering
    \begin{subfigure}[b]{0.48\textwidth}  
        \includegraphics[width=\textwidth, height=6.0cm]{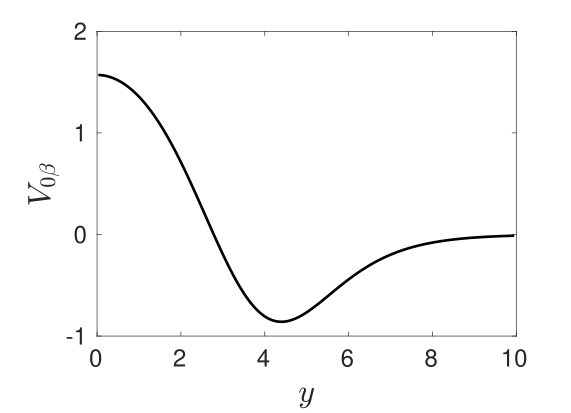}
        \caption{Dimple eigenfunction at the fold point}
        \label{fig:bruss_dimple}
    \end{subfigure}
    \begin{subfigure}[b]{0.48\textwidth}
        \includegraphics[width=\textwidth,height=6.0cm]{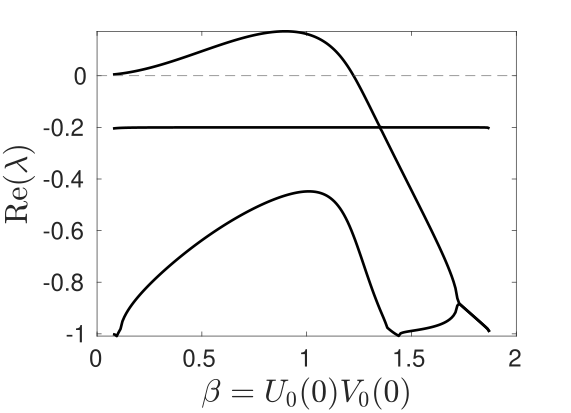}
        \caption{$\mbox{Re}(\lambda)$}
        \label{fig:bruss_eig}
    \end{subfigure}
    \caption{Left panel: The dimple eigenfunction component $\Phi = V_{0\beta}$ at the saddle-node bifurcation point $B = B_c \approx 0.685$ that corresponds to $\lambda = 0$ when $f = 0.8$. Right panel: $\mbox{Re}(\lambda)$ versus $\beta\equiv U_0(0)V_0(0)$ for the numerically computed dominant eigenvalues of \eqref{bruss:eig_prob} when $f = 0.8$. We observe that $\mbox{Re}(\lambda) < 0$ along the primary solution branch of Fig.~\ref{fig:Bru_core_bif} when $f = 0.8$. The lower branch is unstable due to a real positive eigenvalue for \eqref{bruss:eig_prob}.}
    \label{fig:bruss_dimple_eig}
\end{figure}

\subsection{Asymptotic construction of quasi-equilibria: The outer solution}\label{sec:bruss_equil_out}

Next, we show how to determine $B$ by matching the far-field behavior of the Brusselator core solution along the primary branch of an outer solution valid on $\mathcal O(\varepsilon_L) \ll |x| < l$.

In the outer region, the outer solution for \eqref{Brusselator_Lagrange} on $0^+ < x < \ell$ satisfies
\begin{subequations}
    \begin{align}
        - v +  a + f u v^2 &= 0 \, , \label{Bruss_out3}
        \\
        D_L u_{xx} + v - u v^2 &= 0 \, ,  \quad u_x(\ell) = 0 \, . \label{Bruss_out4}
    \end{align}
\end{subequations}
By solving for $u$ in \eqref{Bruss_out3}, we obtain that
\begin{equation}\label{Bruss_utov}
    u = \frac{(v - a)}{f v^2}, \quad \mbox{for} \quad v > a \, ,
\end{equation}
so that
\begin{equation}\label{Bruss_ux2}
    u_x = \frac{(2 a - v)}{f v^3} \, v_x \, .
\end{equation}
Upon substituting \eqref{Bruss_utov} and \eqref{Bruss_ux2} into \eqref{Bruss_out4}, we conclude that
\begin{subequations}\label{Bruss_fullouter}
    \begin{equation}\label{Bruss_nonlinear}
        D_L \, \left(g(v) \, v_x\right)_x = R_b(v) \, , \quad 0^+ < x < \ell \, ; \quad v_x(\ell) = 0 \, ,
    \end{equation}
    where
    \begin{equation}
        g(v) \equiv \frac{(2 a - v)}{v^3} \, , \qquad R_b(v) \equiv (1 - f) v - a \, . \label{Bruss_fR}
    \end{equation}
\end{subequations}
The outer problem \eqref{Bruss_fullouter} is well-posed when $u > 0$ and $u_x > 0$ on $0^+ < x < \ell$. This requires that $a < v(x)< 2 a$ on $0^+ < x < \ell$.

The derivation of the matching condition between the inner and outer solutions parallels that done in \S \ref{sec:sch_equil} and is omitted. We conclude that
\begin{equation}\label{bruss:match}
    f \lim_{x \to 0^+} u_x = \lim_{x \to 0^+} g(v) v_x = \frac{B f}{\sqrt{D_L}} \, , \quad v(0^+) \sim a + a^2 C_b f \frac{\varepsilon}{\sqrt{D}} \, , \quad
    u(0^+) \sim C_b \frac{\varepsilon}{\sqrt{D}} \, , 
\end{equation}
where the constant $C_b = C_b(B, f)$ must be computed from the core problem \eqref{bruss:rep_core} with far-field behavior \eqref{bruss:far_u0}. Since $C_b > 0$, as seen from Fig.~\ref{fig:Bru_Cb_vs_B}, we conclude from \eqref{bruss:match} that $v(0^+) > a$.

Next, we establish a key lemma analogous to that in Lemma \ref{lemma:sch}.

\begin{lemma}\label{lemma:bruss}
    Suppose that $1/2 < f < 1$. Then, on the range of $x$ where $a < v(x) < 2 a$, we have $R_b(v) < 0$ and, consequently, $dv/dx > 0$.
\end{lemma}

\begin{proof}
    From \eqref{Bruss_fR} we have that $R_b(a) = - a f < 0$ and $R_b(2 a) = a (1 - 2 f) < 0$, since $f > 1/2$. Moreover, $dR_b/dv = (1 - f) > 0$, since $f < 1$. Therefore, $R_b < 0$ on the range $a < v < 2 a$, when $1/2 < f < 1$. The proof that $dv/dx > 0$ when $a < v(x) < 2 a$ follows from \eqref{lemma:sch_eq} as in Lemma \ref{lemma:sch} and is omitted.
\end{proof}

Since the remaining steps in the analysis to construct quasi steady-states parallels that in \S \ref{sec:sch_equil} we only highlight the key differences.

In place of \eqref{Gprime_sch}, we define $\mathcal G_b^{\prime}(\xi)$ by
\begin{equation}\label{Bprime_sch}
    \mathcal G_b^{\prime}(\xi) \equiv - R_b(s) \, g(s) = \frac{2 a^2}{\xi^3} - \frac{a (3 - 2 f)}{\xi^2} + \frac{(1 - f)}{\xi} > 0 \, ,  \quad \mbox{on} \quad a < \xi < 2a \, ,
\end{equation}
when $1/2 < f < 1$. A first integral of \eqref{Bprime_sch} is given by
\begin{equation}\label{G_B}
    \mathcal G_b(\xi) = - \frac{a^2}{\xi^2} + \frac{a (3 - 2 f)}{\xi} + (1 - f) \log{\xi} \, .
\end{equation}
With $v(0^+)$ as defined in \eqref{bruss:match}, in place of \eqref{sch:integral_Bfinal}, we obtain that
\begin{equation}\label{Bruss:integral_Bfinal}
    B^2 = \frac{2}{f^2} \left[\mathcal G_b(\mu) - \mathcal G_b
    (v(0^+)\right] \, , \quad \mbox{with} \quad v(0^+) \sim a +
    a^2 C_b f \frac{\varepsilon}{\sqrt{D}} \, .
\end{equation}
Since $\mathcal G_b^{\prime}(\mu) > 0$ on $a < \mu < 2 a$, we observe that $B$ is an increasing function of $\mu = v(\ell)$.

Next, we multiply \eqref{Bruss_nonlinear} by $g(v) v_x$ and integrate. Upon using the monotonicity of $v(x)$ for $a < v < 2 a$ when $1/2 < f < 1$, we obtain that $v(x)$ is now defined implicitly by
\begin{equation}\label{sepa_Bruss}
    \chi_b\left[v(x)\right] = \sqrt{\frac{2}{D_L}} \, x \, ,
\end{equation}
where, in place of \eqref{properchi_sch}, we now have that $\mu = v(\ell)$ is determined by
\begin{equation}\label{properchi_bruss}
    \chi_b(\mu) \equiv - 2 \frac{\sqrt{\mathcal G_b(\mu) - \mathcal G_b(v(0^+))}}{R_b(v(0^+))}+ 2 \int_{v(0^+)}^\mu \, \frac{\sqrt{\mathcal G_b(\mu) - \mathcal G_b(\xi)}}{\left[R_b(\xi)\right]^2} \, R_b^{\prime}(\xi) \, d\xi = \sqrt{\frac{2}{D_L}} \ell \, .
\end{equation}
Here, $R_b(\xi)$ and $\mathcal G_b(\xi)$ are defined in \eqref{Bruss_fR} and \eqref{G_B}, respectively. The positivity of $\chi_b^{\prime}(\mu)$ on $a < \mu < 2 a$ when $1/2 < f < 1$ follows, as in \eqref{properchip_sch}, since $R_b(\xi) < 0$, $R_b^{\prime}(\xi) > 0$ and $\mathcal G_b^{\prime}(\xi) > 0$ on $a < \xi < 2 a$. We conclude that the maximum value of $\mu = v(\ell)$ occurs when $\mu = 2a$, and we define \begin{equation}\label{bruss:xmax}
    \chi_{b, \max} \equiv \chi_b(2a) \, , \quad \mu_{\max} \equiv 2a \, .
\end{equation}
We summarize our asymptotic construction of a one-spike solution to the Brusselator \eqref{Brusselator_Lagrange} as follows:

\begin{result}
    On the range $1/2 < f < 1$, our asymptotic construction of a one-spike solution to \eqref{Brusselator_Lagrange} on $|x| \leq \ell$ when $\varepsilon_L \ll 1$ reduces to solving the coupled nonlinear algebraic system
    \begin{equation} \label{bruss:nas}
        \chi_b(\mu) = \sqrt{\frac{2}{D_L}} \ell \, , \qquad B^2 = \frac{2}{f} \left[\mathcal G_b(\mu) - \mathcal G_b(v(0^+)\right] \, ,
    \end{equation}
    for $B$ and $\mu = v(\ell)$ in terms of the parameters $\ell/\sqrt{D_L}$, $a$, $f$ and $\varepsilon/\sqrt{D}$: Here, $v(0^+)$, with $v(0^+) > a$, is given in \eqref{bruss:match}, while $\chi_b(\mu)$ and $\mathcal G_b(\xi)$ are defined in \eqref{properchi_bruss} and \eqref{G_B}, respectively. For a $K$-spike solution on $- 1 \leq x \leq 1$, we must set $\ell = 1/K$.
\end{result}

Spike self-replication will occur for the range of $f$ in $1/2 < f < 1$ where \eqref{bruss:nas} has a solution with $B > B_c(f)$ and $\mu < 2 a$. For this range of $f$, spike self-replication is predicted to occur when
\begin{equation} \label{bruss:thresh_repD}
    D_L < D_{L, K}^{rep} \equiv \frac{2}{K^2 \left[\chi_b(\mu)\right]^2} \, ,
\end{equation}
where $\mu < 2a$ is calculated from numerically solving \eqref{Bruss:integral_Bfinal} with $B = B_c(f)$. In terms of the domain half-length $L$, we predict that spike self-replication occurs when
\begin{equation}\label{bruss:thresh_repL}
    L > L_K^{rep} \equiv \sqrt{\frac{D}{2}} K \chi_b(\mu) \, .
\end{equation}
Alternatively, spike nucleation near $x = \pm \ell$ will occur for the range of $f$ where \eqref{Bruss:integral_Bfinal} has a solution with $\mu = 2 a$ and $B < B_c(f)$. For this range of $f$, spike nucleation is predicted to occur when
\begin{equation}\label{bruss:thresh_nucD}
    D_L < D_{L, K}^{nuc} \equiv \frac{2}{K^2 \left[\chi_{b, \max}\right]^2} \, , \quad \mbox{where} \quad \chi_{b, \max} = \chi_b(2 a) \, .
\end{equation}
In terms of $L$, we predict that spike nucleation occurs when
\begin{equation} \label{bruss:thresh_nucL}
    L > L_K^{nuc} \equiv \sqrt{\frac{D}{2}} K \chi_{b, \max} \, .
\end{equation}
The critical value $f_c = f_c(a)$ that separates whether spike-replication or spike nucleation behavior occurs is obtained by numerically solving \eqref{Bruss:integral_Bfinal} using Newton's method with $B = B_c(f)$ and $\mu = 2a$ in terms of $f$ (see Fig.~\ref{fig:Bru_core_bif_c}). Spike nucleation occurs for $1/2 < f < f_c(a)$, while spike self-replication occurs on the range $f_c(a) < f < 1$. To estimate $f_c$ when $\varepsilon/\sqrt{D} \ll 1$, we simply set $B = B_c(f)$, $\mu = 2a$ and $v(0^+) = a$ in \eqref{Bruss:integral_Bfinal} to get ${B_c^2 f^2/2} \sim \mathcal G_b(2 a) - \mathcal G_b(a)$, where $\mathcal G_b(\xi)$ is given in \eqref{G_B}. After some algebra, we conclude rather remarkably, for $\varepsilon/\sqrt{D} \ll 1$, that $f_c$ is independent of $a$ and is a root of
\begin{equation}\label{bruss:fc_eq}
    \frac{B_c^2 f^2}{2} = - \frac{3}{4} + f + (1 - f) \log{2} \, .
\end{equation}
By path-following the fold point of the core solution, the curve $B_c(f)$ in Fig.~\ref{fig:Bru_core_bif_c} can be estimated numerically. Then, by numerically solving \eqref{bruss:fc_eq} we estimate for any $a > 0$ that $f_c \approx 0.769 + \mathcal O\left(\varepsilon/\sqrt{D}\right)$.

In summary, we conclude that the transition threshold $f_c$ is independent of the domain half-length $L$ and, within a negligible $\mathcal O\left(\varepsilon/\sqrt{D}\right)$ error, is also independent of the parameter $a$. As a result, for any $a > 0$, as the domain grows either only spike nucleation or only spike self-replication is predicted to occur for the Brusselator.

\begin{figure}[h!tbp]
    \centering
    \begin{subfigure}[b]{0.48\textwidth}
        \includegraphics[width=\textwidth, height=6.0cm]{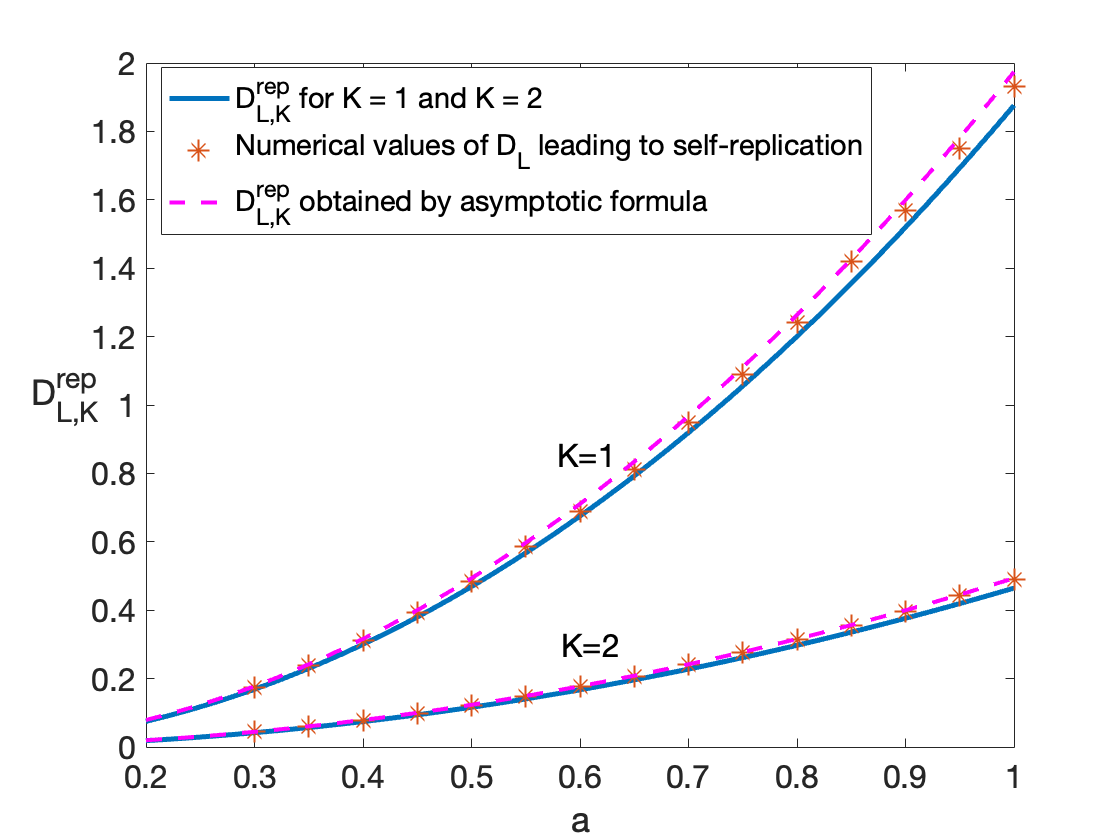}
        \caption{$D_{L,K}^{rep}$}
        \label{fig:Bru_Dc_rep}
    \end{subfigure}
    \begin{subfigure}[b]{0.48\textwidth}
        \includegraphics[width=\textwidth, height=6.0cm]{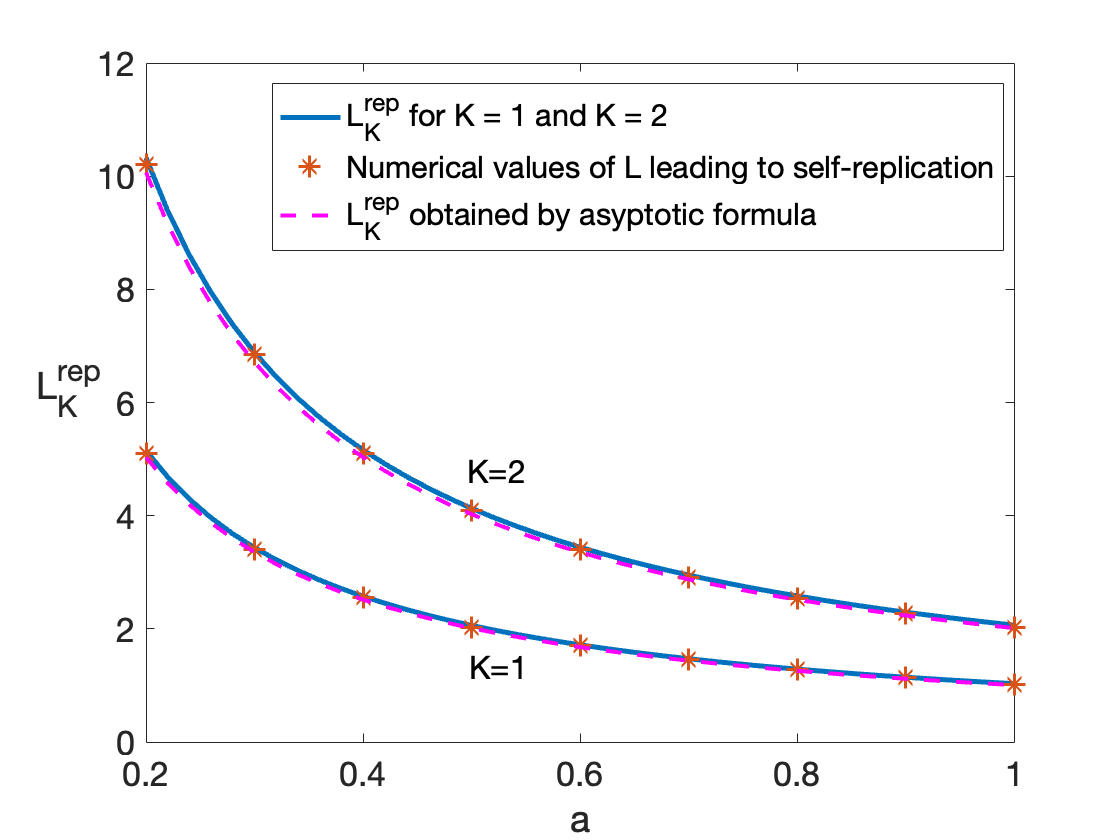}
        \caption{$L_{K}^{rep}$}
        \label{fig:Bru_Lc_rep}
    \end{subfigure}
    \caption{Comparison between numerical and analytical results of the spike self-replication thresholds $D_{L, K}^{rep}$ (a) and $L_K^{rep}$ (b) for the Brusselator model \eqref{Brusselator_Lagrange} versus $a$ with $f = 0.8 > f_c \approx 0.769$, $\varepsilon = 0.01$, $D = 2$, and for $K = 1$ and $K = 2$. The solid curves are the asymptotic results for $D_{L, K}^{rep}$ and $L_K^{rep}$ given in \eqref{bruss:thresh_repD} and \eqref{bruss:thresh_repL}, respectively. The red stars are from full numerical simulations of the Brusselator model \eqref{Brusselator_Lagrange} using \textit{FlexPDE} \cite{flexpde2015}. The dashed curves are the closed-form asymptotic approximations for $D_{L, K}^{rep}$ and $L_K^{rep}$ given in \eqref{bruss:DLK_smalla} and \eqref{bruss:small_Lfinal}, respectively, that are valid for $a \ll 1$.}
    \label{fig:Bru_rep_f0.8DL}
\end{figure}

For a fixed $f_c < f < 1$, where $f_c \approx 0.769$, the spike self-replication threshold in \eqref{bruss:thresh_repD} as $a$ is varied is computed numerically from using Newton's method on \eqref{bruss:nas} with $B = B_c(f)$ and $\mu < 2 a$. For $f = 0.8$, in Fig.~\ref{fig:Bru_Dc_rep} we plot $D_{L, K}^{rep}$ versus $a$ for $\varepsilon = 0.01$, $D = 2$, and for $K = 1$ and $K = 2$. The corresponding critical length threshold $L_K^{rep}$ obtained from \eqref{bruss:thresh_repL} is shown in Fig.~\ref{fig:Bru_Lc_rep}. These figures show a very close agreement between our asymptotic prediction for the spike self-replication threshold $D_{L, K}^{rep}$ or $L_K^{rep}$ and corresponding results estimated from time-dependent simulations of \eqref{Brusselator_Lagrange} using \textit{FlexPDE} \cite{flexpde2015}.

In \S \ref{sec:bruss_small_a} below, we will derive closed-form explicit approximations for $D_{L, K}^{rep}$ and $L_K^{rep}$ in the limit $a \ll 1$ that still provide a rather good approximation for the spike self-replication thresholds even for moderately small $a$. These results, given below in \eqref{bruss:DLK_smalla} and \eqref{bruss:small_Lfinal}, are indicated by the green solid curves in Fig.~\ref{fig:Bru_rep_f0.8DL}.

\begin{figure}[h!tbp]
    \centering
    \begin{subfigure}[b]{0.48\textwidth}
        \includegraphics[width=\textwidth, height=6.0cm]{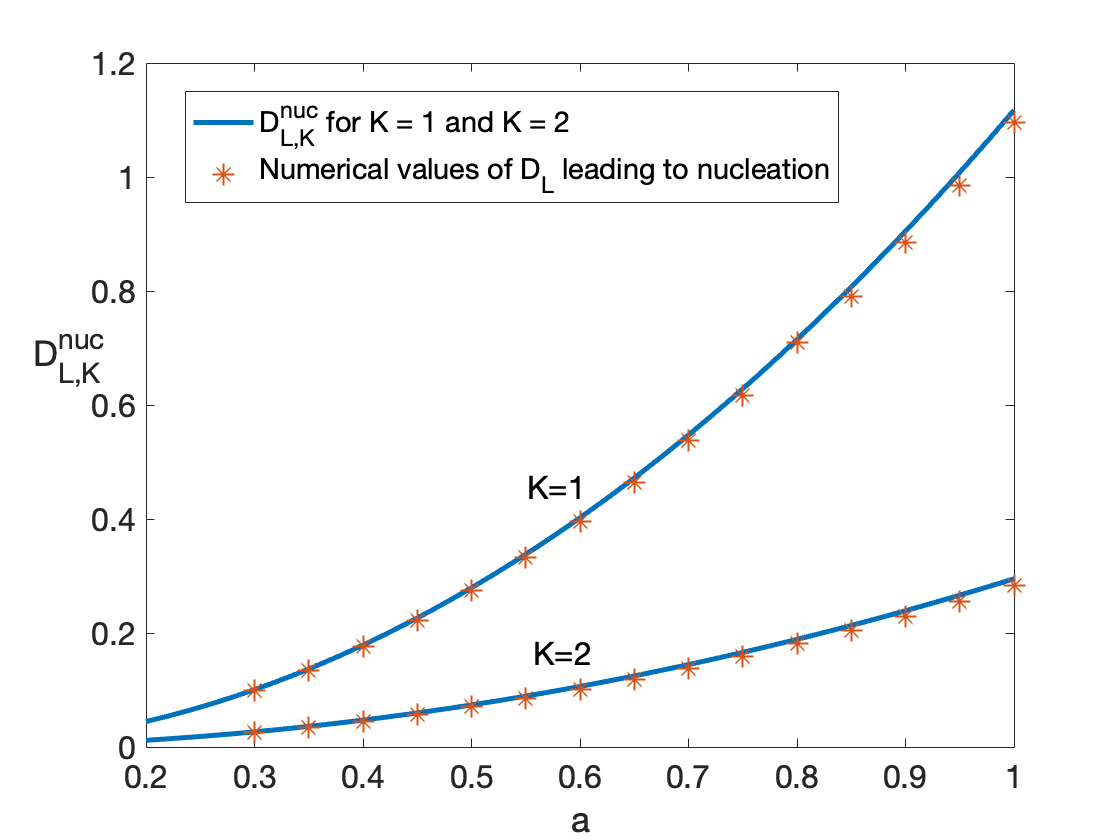}
        \caption{$D_{L,K}^{nuc}$}
        \label{fig:Bru_Dc_nuc}
    \end{subfigure}
    \begin{subfigure}[b]{0.48\textwidth}
        \includegraphics[width=\textwidth, height=6.0cm]{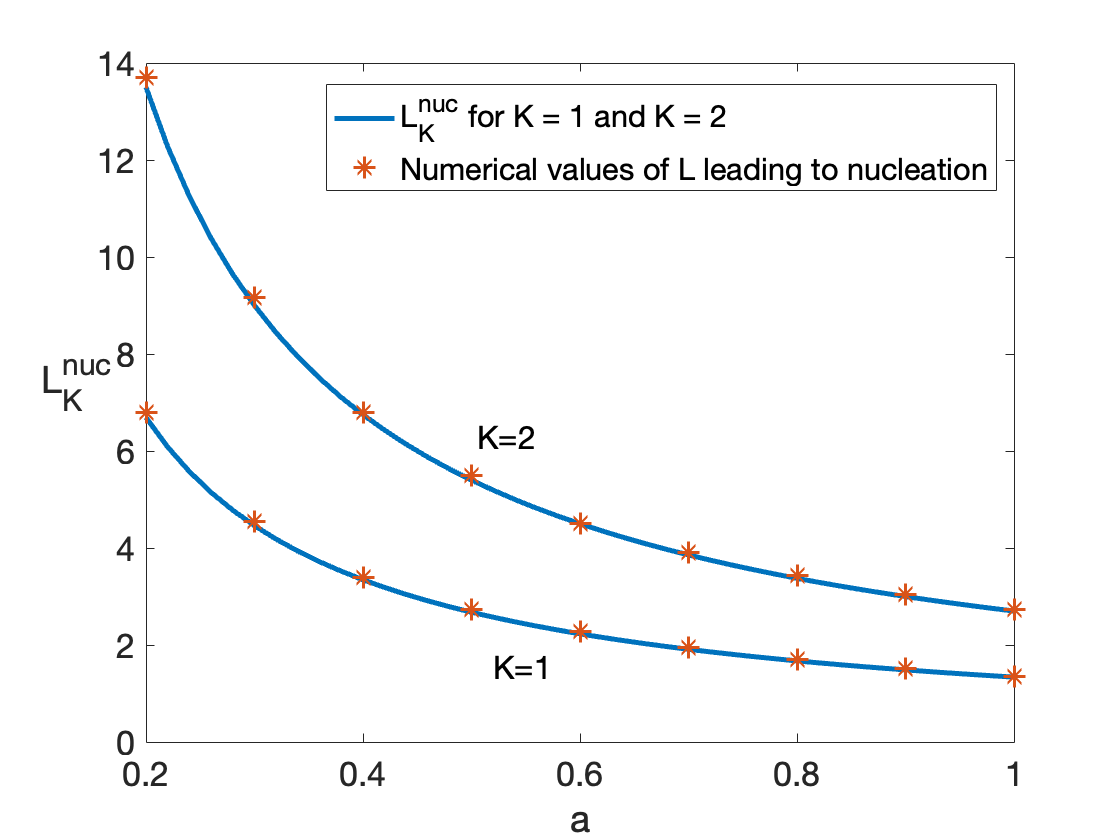}
        \caption{$L_{K}^{nuc}$}
        \label{fig:Bru_Lc_nuc}
    \end{subfigure}
    \caption{Same plot as in Fig.~\ref{fig:Bru_rep_f0.8DL} but for $f = 0.7 < f_c \approx 0.769$, $\varepsilon = 0.01$ and $D = 2$ where spike self-nucleation is the dominant mechanism as $L$ is increased. The solid curves are the asymptotic results for $D_{L, K}^{nuc}$ and $L_K^{nuc}$, as given in \eqref{bruss:thresh_nucD} and \eqref{bruss:thresh_nucL}, respectively. The red stars are obtained from time-dependent numerical simulations of \eqref{Brusselator_Lagrange} using \textit{FlexPDE} \cite{flexpde2015}. }
    \label{fig:Bru_nuc_f0.7DL}
\end{figure}

For the same parameters as in Fig.~\ref{fig:Bru_rep_f0.8DL}, but now with $f = 0.7$ we obtain, from the numerical solution to the algebraic system \eqref{bruss:nas}, that spike nucleation is the dominant mechanism as the domain half-length is increased. In Fig.~\ref{fig:Bru_Dc_nuc} we plot $D_{L, K}^{nuc}$ versus $a$ for $K = 1$ and $K = 2$, obtained from \eqref{bruss:thresh_nucD}, while the corresponding critical length threshold $L_K^{nuc}$, as given in \eqref{bruss:thresh_nucD}, is shown in Fig.~\ref{fig:Bru_Lc_nuc}. From these figures, we observe that these asymptotic predictions for the critical spike-nucleation thresholds agree very closely with corresponding results computed from time-dependent simulations of \eqref{Brusselator_Lagrange}.

\subsection{Asymptotics of the outer solution for $a$ small}
\label{sec:bruss_small_a}

We now calculate an explicit analytical prediction for the spike self-replication threshold valid for $a \ll 1$. In the limit $a \to 0$, spike nucleation behavior no longer occurs.

When $a \ll 1$, we obtain from \eqref{Bruss_out3} that
\begin{equation}\label{bruss:a-small-v}
    v(x) = a + a^2 f u + \ldots \, .
\end{equation}
By substituting \eqref{bruss:a-small-v} into \eqref{Bruss_out4}, and together with the two matching conditions for $u$ in \eqref{bruss:match}, we obtain that $u(x)$ for $a \ll 1$ satisfies the approximating linear problem
\begin{subequations}\label{bruss:nonlinear_asmall}
    \begin{align}
        D_L \, u_{xx} &= - a + a^2 (1 - f) u \, , \quad 0^+ < x < \ell \, ; \quad u_x(\ell) = 0 \, ,
        \\
        u(0^+) &= C_b \frac{\varepsilon}{\sqrt{D}} \, , \quad u_x(0^+) = \frac{B}{\sqrt{D_L}} \, ,
    \end{align}
\end{subequations}
where $C_b = C_b(B, f)$ is defined from the far-field of the core
problem \eqref{bruss:far_u0}.

The solution to \eqref{bruss:nonlinear_asmall} is
\begin{equation}\label{bruss:v1-solve}
    u = \frac{1}{a (1 - f)} + \left(\frac{C_b \varepsilon}{\sqrt{D}} - \frac{1}{a (1 - f)}\right) \frac{\cosh\left[a \sqrt{1 - f} (\ell - |x|)/\sqrt{D}\right]}{\cosh\left(a \sqrt{1 - f} \ell/\sqrt{D_L}\right)} \, .
\end{equation}
Upon satisfying the required condition for $u_{x}(0^{+})$ in \eqref{bruss:nonlinear_asmall} we get that
\begin{equation}\label{bruss:v1-B}
    B \sqrt{1 - f} = \left(1 - \frac{C_b a (1 - f) \varepsilon}{
    \sqrt{D}}\right) \tanh\left(a \sqrt{1 - f} \ell/\sqrt{D_L}\right) \, .
\end{equation}
To determine the approximation for $D_{L, K}^{rep}$ we evaluate \eqref{bruss:v1-B} at the fold point $B = B_c(f)$ of the core problem \eqref{bruss:rep_core} and solve for $D_L$. However, we observe from \eqref{bruss:v1-B} that we must have $a/\sqrt{D_L} = \mathcal O(1)$ when $a \ll 1$ and that $B_c \sqrt{1 - f} < 1$ in order for \eqref{bruss:v1-B} to have a solution. Upon neglecting the $\mathcal O(a \varepsilon)$ term in the prefactor in \eqref{bruss:v1-B}, we set $B = B_c$ and $\ell = 1/K$ and solve for $D_L$ to obtain the approximation
\begin{equation}\label{bruss:DLK_smalla}
    D_{L, K}^{rep} \approx \frac{a^2 (1 - f)}{K^2}
    \left[\tanh^{- 1}\left(B_c \sqrt{1 - f}\right)\right]^{- 2} \, ,
    \quad \mbox{for} \quad a \ll 1 \, ,
\end{equation}
for the spike self-replication threshold when $B_c \sqrt{1 - f} < 1$. Equivalently, the critical domain half-length for spike self-replication when $B_c \sqrt{1 - f} < 1$ is approximated by
\begin{equation}\label{bruss:small_Lfinal}
    L_K^{rep} \sim \frac{\sqrt{D} K}{a \sqrt{1 - f}} \tanh^{- 1}\left(B_c \sqrt{1 - f}\right) \, , \quad \mbox{for} \quad a \ll 1 \, .
\end{equation}
Rather remarkably, in Fig.~\ref{fig:Bru_rep_f0.8DL} we show a very close agreement, even at only moderately small values of $a$, between the approximate results \eqref{bruss:DLK_smalla} and \eqref{bruss:small_Lfinal} and the full simulation results.

\subsection{Global bifurcation diagrams}\label{bruss:global}

\begin{figure}[htbp]
    \centering
    \includegraphics[width=0.99\textwidth, height=5.5cm]{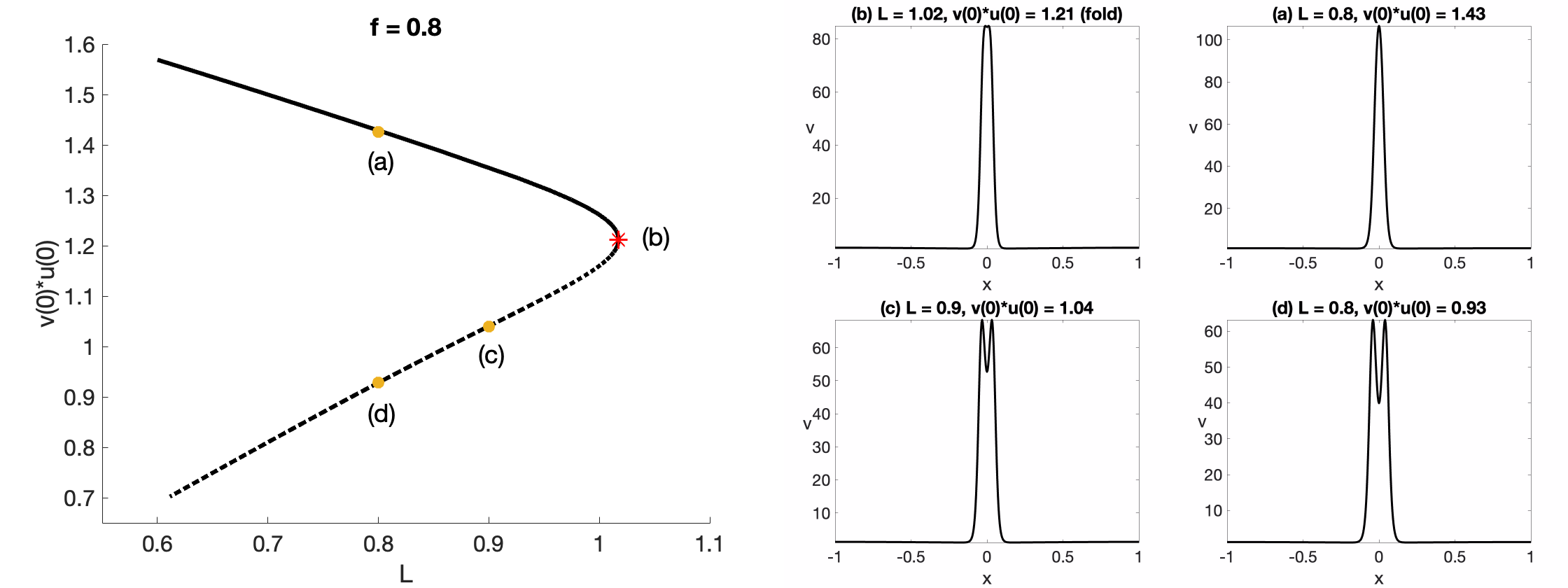}
    \caption{Left panel: Global bifurcation diagram of $v(0) u(0)$ versus $L$ for single-spike steady-states of the Brusselator \eqref{Brusselator_Lagrange} computed using \textit{pde2path} \cite{pde2path} for $f = 0.8$, $\varepsilon = 0.01$, $a = 1$ and $D = 2$. Since $f_c < f < 1$, with $f_c \approx 0.769$, we predict that spike self-replication occurs as $L$ is increased starting from point (a) on the linearly stable upper branch. The red star at (b) is the fold point signifying the onset of replication. The upper branch is connected to an unstable lower branch where the spike profile has a volcano shape. The bifurcation diagram is qualitatively similar to that of the core problem in Fig.~\ref{fig:Bru_core_bif}. Right panel: Spike profile $v(x)$ and bifurcation values (top of subfigures) at the indicated points in the left panel.}
    \label{fig:Bru_bif_rep}%
\end{figure}

In terms of $L$, we use \textit{pde2path} \cite{pde2path} to compute global bifurcation diagrams that path-follow single-spike steady-state solutions for the Brusselator \eqref{Brusselator_Lagrange} with $\varepsilon = 0.01$, $a = 1$ and $D = 2$. For $f = 0.8$, which satisfies $f_c < f < 1$ with $f_c \approx 0.769$, the global bifurcation diagram in Fig.~\ref{fig:Bru_bif_rep} for the case where spike self-replication occurs is qualitatively identical to that for the Schnakenberg model shown in Fig.~\ref{fig:Sch_bif_rep}. The computed saddle-node value $L^{num} \approx 1.02$ in Fig.~\ref{fig:Bru_bif_rep} for the onset of spike self-replication is well-approximated by the critical threshold $L_1^{rep}$ in \eqref{bruss:thresh_repL} for which the core problem \eqref{bruss:rep_core} with far-field behavior \eqref{bruss:far_u0} has a saddle-node point when coupled to the outer solution (see Fig.~\ref{fig:Bru_Lc_rep}).

In contrast, for $f = 0.7$, which satisfies $0.5 < f < f_c \approx 0.769$, the global bifurcation diagram shown in Fig.~\ref{fig:Bru_bif_nuc} for the case where spike nucleation occurs is qualitatively identical to that for the Schnakenberg model shown in Fig.~\ref{fig:Sch_bif_nuc}. The computed saddle-node value $L^{num} \approx 1.35$ in Fig.~\ref{fig:Bru_bif_nuc} for the onset of spike nucleation is well-approximated by the critical threshold $L_1^{nuc}$ in \eqref{bruss:thresh_nucL} for which the outer solution from our asymptotic theory ceases to exist (see Fig.~\ref{fig:Bru_Lc_nuc}).

\begin{figure}[htbp]
    \centering
    \includegraphics[width=0.99\textwidth, height=5.5cm]{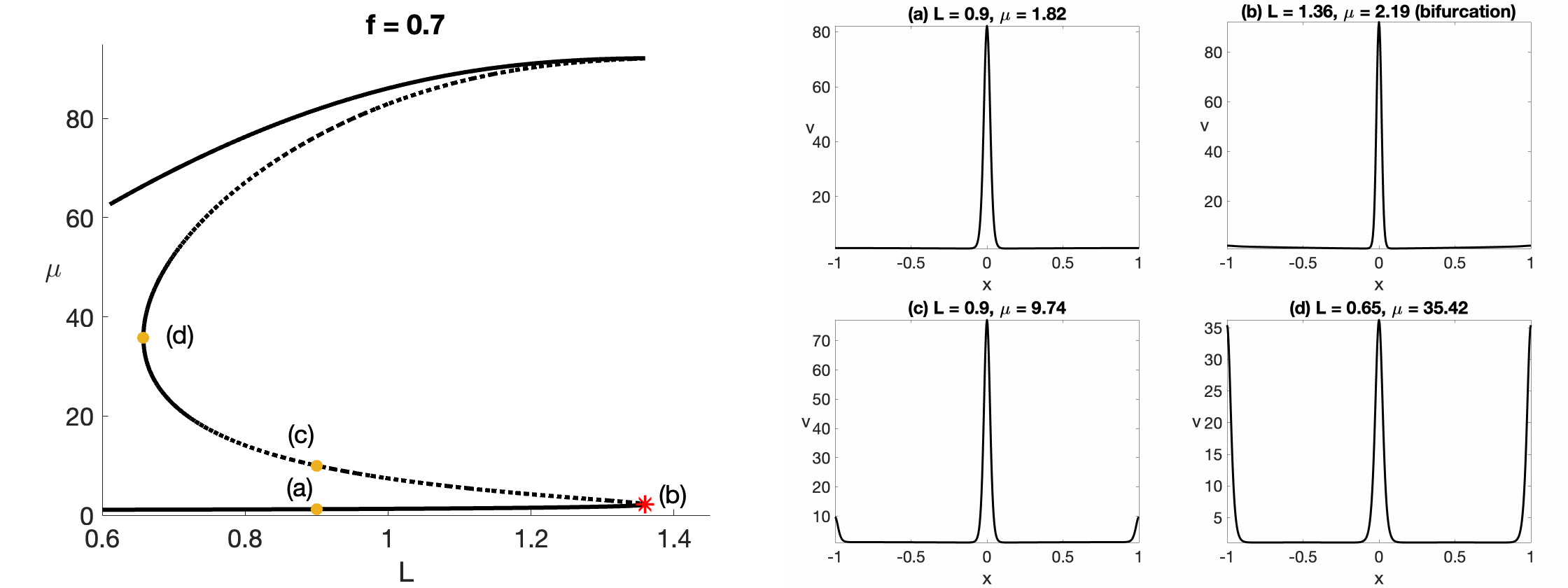}
    \caption{Left panel: Same plot and parameters as in Fig.~\ref{fig:Bru_bif_rep} except that now $f$ is decreased to $f = 0.7 < f_c \approx 0.769$ and the vertical axis is now $\mu = v(1)$. Since $0.5 < f < f_c \approx 0.769$, we predict that spike nucleation will occur as $L$ is increased starting from point (a) on the linearly stable lower branch. The red star is the saddle-node bifurcation point signifying the onset of spike nucleation behavior. On the unstable middle branch, the solution has an interior spike with boundary spikes that emerge at the domain boundaries. Right panel: Spike profile $v(x)$ and bifurcation values at the indicated points in the left panel.}
    \label{fig:Bru_bif_nuc}%
\end{figure}

\subsection{Validation of asymptotic theory: Full PDE
  simulations}\label{sec:Br_fullpde}

\begin{figure}[htbp]
    \centering
    \includegraphics[width=0.95\textwidth, height=6.0cm]{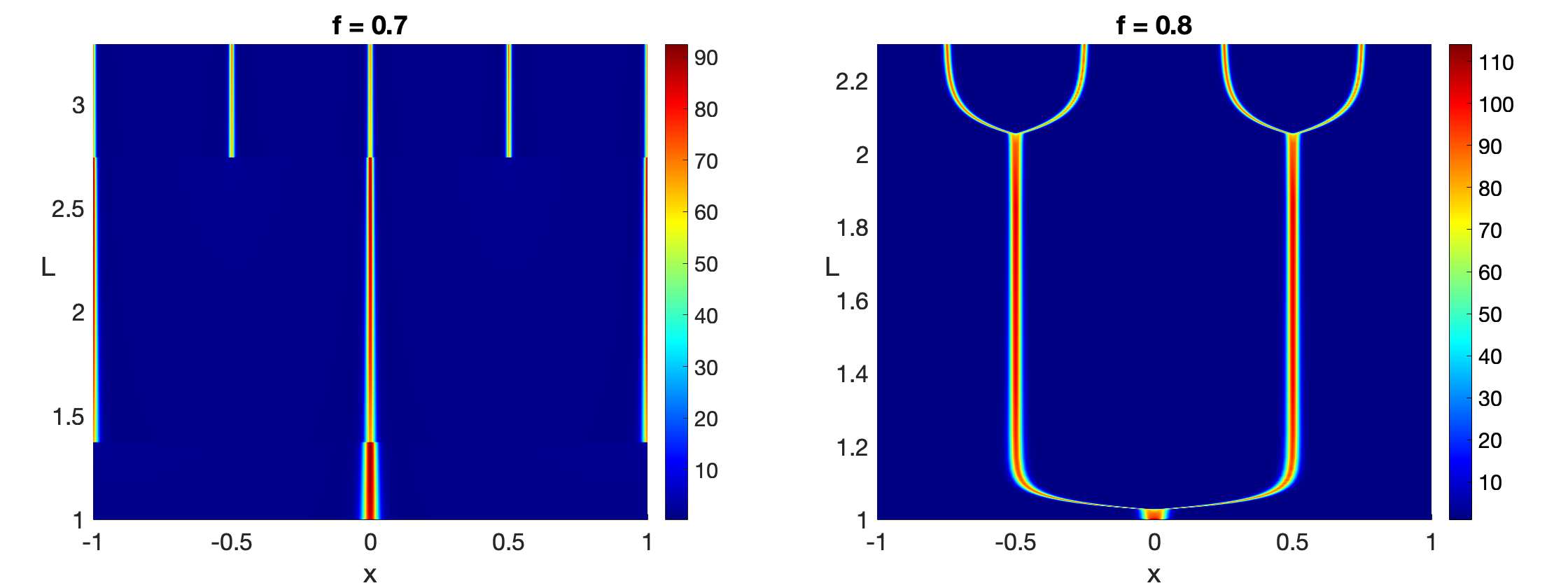}
    \caption{Time-dependent PDE simulations of \eqref{Brusselator_Lagrange} using \textit{FlexPDE} \cite{flexpde2015} as the domain half-length $L$ slowly increases as $L = e^{\rho t}$ with $\rho = 10^{- 4}$. Color plot of the magnitude of the solution $v$ is plotted as a function of $L$, which measures time. Parameters: $\varepsilon = 0.01$, $a = 1$ and $D = 2$. Left panel: For $f = 0.7$, spike nucleation dynamics occur as $L$ increases, as predicted by Fig.~\ref{fig:Bru_nuc_f0.7DL}. Right panel: For $f = 0.8$ spike self-replication dynamics occur as $L$ increases, as predicted by Fig.~\ref{fig:Bru_rep_f0.8DL}.}
    \label{fig:Bru_flexpde_plot}%
\end{figure}

For parameter sets corresponding to either spike self-replication or spike nucleation, in Fig.~\ref{fig:Bru_flexpde_plot} we show full-time-dependent PDE simulations of \eqref{Brusselator_Lagrange} computed using \textit{FlexPDE} \cite{flexpde2015} as the domain half-length $L$ slowly increases in time according to $L=e^{\rho t}$ with $\rho = 10^{- 4}$. For $\varepsilon = 0.01$, $a = 1$ and $D = 2$, in the left panel of Fig.~\ref{fig:Bru_flexpde_plot} we show spike nucleation behavior for $f = 0.7$. The two transition values of $L$ where nucleation events occur are well-approximated by the asymptotic results in \eqref{bruss:thresh_nucL} as plotted in Fig.~\ref{fig:Bru_Lc_nuc}. In contrast, for $f = 0.8$ in the right panel of Fig.~\ref{fig:Bru_flexpde_plot} we show spike self-replication behavior where the two transition values of $L$ where splitting occurs are well-approximated by the results in \eqref{bruss:thresh_repL} as plotted in Fig.~\ref{fig:Bru_Lc_rep}.

\section{The Gierer-Meinhardt model}\label{sec:gm}

In this section, we construct a one-spike solution centered at $x = 0$ to the GM model \eqref{GM_Lagrange} defined on the canonical domain $|x| \leq \ell$. We will show that, although spike self-replication behavior no longer occurs, the GM model admits spike nucleation behavior near $x=\pm \ell$ for $0<\kappa<1$ when $D_{L}$ is below a threshold.

In the inner region, we introduce the inner variables $y$, $A$ and $H$ by
\begin{equation}\label{gm:inn_var}
    y = \frac{x}{\varepsilon_L} \, , \qquad \mathcal A = \frac{A}{\varepsilon_L} \, , \qquad \mathcal H = \frac{H}{\varepsilon_L} \, ,
\end{equation}
so that \eqref{GM_Lagrange} transforms on $y \geq 0$ to
\begin{equation}\label{gm:main}
    A_{yy} - A + \frac{A^2}{H} + \varepsilon_L \kappa = 0 \, , \qquad D_L H_{yy} - \varepsilon_L^2 H + \varepsilon_L A^2 = 0 \, ,
\end{equation}
with $A_y(0) = 0$ and $H_y(0) = 0$. Upon expanding 
\begin{equation}\label{gm:expan}
    A = A_0 + \mathcal O(\varepsilon_L) \, , \qquad H = H_0 + \varepsilon_L H_1 + \ldots \, ,
\end{equation}
and substituting into \eqref{gm:main} we obtain that $H_0$ is a constant, and that $H_1$ satisfies
\begin{equation}\label{gm:h1}
      H_{1yy} = - \frac{A_0^2}{D_L} \, .  
\end{equation}
To account for the non-zero far-field of the activator, we will take $A_0$ to be the homoclinic solution of
\begin{equation}\label{gm:A0}
    A_{0yy} - A_0 + \frac{A_0^2}{H_0} + \varepsilon_L \kappa = 0 \, , \quad y \geq 0 \, ; \quad A_{0y}(0) = 0 \, , \quad A_0(0) > 0 \, .
\end{equation}
The homoclinic solution to \eqref{gm:A0} is
\begin{equation}\label{gm:A0sol}
    A_0(y) = H_0(w_0(y) + \gamma) \, ,
\end{equation} 
where $w_0(y)$ is the unique solution to
\begin{equation}\label{gm:weqn}
    w_{0yy} - (1 - 2 \gamma) w_0 + w_0^2 = 0 \, , \qquad w_0(0) > 0 \, , \quad w_{0y}(0) = 0 \, , \quad \lim_{y \to \infty} w_0 = 0 \, .
\end{equation}
Here $\gamma$ satisfies the quadratic equation
\begin{equation}\label{gm:gamma_eq}
    \gamma^2 - \gamma + \varepsilon_L \kappa H_0 = 0 \, .
\end{equation}
Since we require $\gamma < 1/2$ in \eqref{gm:weqn}, we must take $\gamma$ as the smallest root of \eqref{gm:gamma_eq}, which is given by
\begin{equation}\label{gm:gam}
    \gamma = \frac{1 - \sqrt{1 - 4 \varepsilon_L \kappa/H_0}}{2} \, .
\end{equation}
For $\varepsilon_L \ll 1$, we have $\gamma \sim \varepsilon_L \kappa/H_0 + \mathcal O(\varepsilon_L^2)$. The explicit solution to \eqref{gm:weqn} is readily calculated as
\begin{equation}\label{gm:w0sol}
    w_0(y) = \frac{3}{2}(1 - 2 \gamma) \, \mbox{sech}^2\left(\frac{\sqrt{1 - 2 \gamma}}{2} y\right) \, ,
\end{equation}
which, from \eqref{gm:A0sol}, determines the homoclinic solution to \eqref{gm:A0} up to a constant $H_0$ to be found. By letting $y \to \infty$, we use \eqref{gm:inn_var}, \eqref{gm:expan} and \eqref{gm:A0sol} to conclude that
\begin{equation}\label{gm:A_ff}
    \mathcal A \sim \frac{H_0 \gamma}{\varepsilon_L} \, , \quad \mbox{as} \quad y \to \infty \, .
\end{equation}
Next, we determine the far-field behavior of $H$. Upon substituting \eqref{gm:A0sol} into \eqref{gm:h1}, we obtain that
\begin{equation}\label{gm:h1_1}
  H_{1yy} = - \frac{H_0^2}{D_L}\left(w_0 + \gamma\right)^2 = - \frac{H_0^2\gamma^2}{D_L} - \frac{H_0^2}{D_L}\left(w_0^2 + 2 \gamma w_0\right) \, .
\end{equation}
Upon integrating \eqref{gm:h1_1} using $H_{1y}(0) = 0$, we obtain for any $y > 0$ that
\begin{equation}\label{gm:h1_2}
    H_{1y} = - \frac{H_0^2 \gamma^2}{D_L} y - \frac{H_0^2}{D_L} \left(\int_0^y w_0^2 \, ds + 2\gamma \int_0^y w_0 \, ds \right) \, .
\end{equation}
To determine the limiting behavior as $y \to \infty$, we use \eqref{gm:w0sol} to calculate
\begin{equation}\label{gm:wint}
    \int_0^\infty w_0\, ds = 3 \left(1 - 2 \gamma\right)^{1/2} \, , \qquad \int_0^\infty w_0^2\, ds = 3 \left(1 - 2 \gamma\right)^{3/2} \, .
\end{equation}
By using \eqref{gm:wint} in \eqref{gm:h1_2}, we conclude that
\begin{equation}\label{gm:h1_ff}
    \lim_{y \to \infty} \left(H_{1y} + \frac{H_0^2 \gamma^2}{D_L} y \right) = - \frac{H_0^2}{D_L} \left[3 \left(1 - 2 \gamma\right)^{3/2} + 6 \gamma \left(1 - 2 \gamma\right)^{1/2} \right] = - \frac{3 H_0^2}{D_L} \left(1 - 2 \gamma\right)^{1/2} \, .
\end{equation}
In this way, by using \eqref{gm:h1_ff} together with \eqref{gm:inn_var} and \eqref{gm:expan} we obtain that $\mathcal H$ has the far-field behavior
\begin{equation}\label{gm:H_ff}
    \mathcal H \sim \frac{H_0}{\varepsilon_L} - \frac{H_0^2 \gamma^2}{2 D_L} y^2 - \frac{3 H_0^2}{D_L} \left(1 - 2 \gamma\right)^{1/2} y \, , \quad \mbox{as} \quad y \to \infty \, .
\end{equation}
We now construct the solution to \eqref{GM_Lagrange} in the outer
region $0^+ < x < \ell$. This outer solution satisfies
\begin{subequations}\label{gm:out}
    \begin{align}
        - \mathcal A +  \frac{\mathcal A^2}{\mathcal H} + \kappa &= 0 \, , \label{GM_out1}
        \\
        D_L \mathcal H_{xx} - \mathcal H + \mathcal A^2 &= 0 \, ,  \quad \mathcal H_x(\pm \ell) = 0 \, . \label{GM_out2}
    \end{align}
\end{subequations}
From \eqref{GM_out1}, we obtain that
\begin{equation}\label{gm:hout}
    \mathcal H = \frac{\mathcal A^2}{\mathcal A - \kappa} \, , \quad \mbox{for} \quad \mathcal A > \kappa \, ,
\end{equation}
which implies
\begin{equation}\label{gm:hout_x}
    \mathcal H_x = - \frac{\mathcal A \, \left(2 \kappa - \mathcal A\right)}{\left(A - \kappa\right)^2} \, \mathcal A_x \, , \quad \mbox{for} \quad \mathcal A > \kappa \, .
\end{equation}
Upon substituting \eqref{gm:hout} and \eqref{gm:hout_x} into \eqref{GM_out2} we obtain that the outer problem for $\mathcal A$ is
\begin{subequations}\label{gm:outer_probfull}
    \begin{equation}\label{gm:outer_prob}
        D_L \, \left(f(\mathcal A) \, \mathcal A_x\right)_x = R_g(\mathcal A) \, ; \quad  0^+ < x < \ell \, ; \quad \mathcal A_x(\ell) = 0 \, ,
    \end{equation}
    where
    \begin{equation}\label{gm:outer_def}
        f(\mathcal A) \equiv \frac{\mathcal A \left(2 \kappa - \mathcal A\right)}{\left(\mathcal A - \kappa\right)^2} \, , \qquad R_g(\mathcal A) \equiv \mathcal A^2 - \frac{\mathcal A^2}{\mathcal A - \kappa} \, .
    \end{equation}
\end{subequations}
The problem \eqref{gm:outer_probfull} is well-posed when $\mathcal H > 0$ and $\mathcal H_x > 0$ on $0^+ < x < \ell$. This implies that we must have $\kappa < \mathcal A(x) < 2 \kappa$ on $0^+ < x < \ell$.

Next, we derive the matching conditions between the inner and outer solutions. From \eqref{gm:A_ff}, together \eqref{gm:gam} for $\gamma$, the first matching condition for the outer solution is that
\begin{equation}\label{gm:match_1}
    \mathcal A(0^+) = \frac{H_0}{2 \varepsilon_L} \left(1 - \sqrt{1 - \frac{4 \varepsilon_L \kappa}{H_0}}\right) \, .
\end{equation}
For any $4 \varepsilon_L \kappa/H_0 < 1$, we claim that $\mathcal A(0^+) > \kappa$. To establish this inequality, we define $z \equiv {2 \varepsilon_L \kappa/H_0}$ and readily calculate that
\begin{equation*}
    \frac{\mathcal A(0^+)}{\kappa} = \frac{1 - \sqrt{1 - 2 z}}{z} \, .
\end{equation*}
Since $\sqrt{1 - 2 z} < 1 - z$ on $0 < z < 1/2$, the expression above yields $\mathcal A(0^+) > \kappa$ whenever $4 \varepsilon_L \kappa/H_0 < 1$.

The second condition involves matching the flux $\mathcal H_x$ as $x \to 0^+$. We first observe that the $\mathcal O(y^2)$ term in the far-field behavior \eqref{gm:H_ff} matches exactly with the quadratic term in the near-field behavior of $\mathcal H$ as $x \to 0^+$ that arises from the $\mathcal A^2$ term in \eqref{GM_out2}. From \eqref{gm:H_ff} we conclude that we must have
\begin{equation}\label{gm:hx_match}
    \mathcal H_x = - \frac{3 H_0^2}{\varepsilon_L D_L}
    \left(1 - 2 \gamma\right)^{1/2} \, , \quad \mbox{as} \quad x \to 0^+ \, .
\end{equation}
From \eqref{gm:hout_x}, \eqref{gm:hx_match} provides the second matching condition for the outer solution given in terms of $\mathcal A$ by
\begin{equation}\label{gm:match_2}
    \lim_{x \to 0^+} f(\mathcal A) \mathcal A_x = \frac{3 H_0^2}{\varepsilon_L D_L} \left(1 - 2 \gamma\right)^{1/2} \, .
\end{equation}
Next, we establish the following lemma, which is analogous to Lemmas \ref{lemma:sch} and \ref{lemma:bruss}.

\begin{lemma}\label{lemma:gm}
    Suppose that $0 < \kappa < 1$. Then, on the range of $x$ where $\kappa < \mathcal A(x) < 2 \kappa$, we have $\mathcal R_g(\mathcal A) < 0$ and, consequently, $d \mathcal A/dx > 0$.
\end{lemma}

\begin{proof}
    From \eqref{gm:outer_def} we have
    \begin{equation*}
        \lim_{\mathcal A \to \kappa^+} R_g(\mathcal A) = - \infty
        \, , \qquad R_g(2\kappa)=4\kappa(\kappa-1) \,.
    \end{equation*}
    It follows that $R_g(2 \kappa)<0$ if $0<\kappa<1$. Moreover, we calculate that $R_g^{\prime}({\mathcal A}) = 2 \mathcal A + f(\mathcal A)$, where $f(\mathcal A)$ is defined in \eqref{gm:outer_def}. Since $f(\mathcal A) > 0$ on $\kappa < \mathcal A < 2 \kappa$, it follows that $R_g^{\prime}(\mathcal A) > 0$ on $\kappa < \mathcal A < 2 \kappa$. For $0 < \kappa < 1$, we conclude that $R_g(\mathcal A) < 0$ whenever $\kappa < \mathcal A(x) < 2 \kappa$. Moreover, upon integrating \eqref{gm:outer_prob}, and imposing $\mathcal A_x(\ell) = 0$, we obtain on $0 < x < \ell$ that  \begin{equation}\label{lemma:gm_eq}
        D_L \left[f(\mathcal A) \mathcal A_x\right]\vert_x^\ell = -D_L f(\mathcal A) \mathcal A_x = \int_x^\ell R_g\left[\mathcal A(\eta)\right] \, d\eta < 0 \, ,
    \end{equation}
    whenever $\kappa < \mathcal A < 2 \kappa$ and $0 < \kappa < 1$. Since $f(\mathcal A) > 0$ for $\kappa < \mathcal A < 2 \kappa$, we conclude that $\mathcal A_x > 0$ when $\kappa < \mathcal A(x)< 2 \kappa$ and $0 < \kappa < 1$.
\end{proof}
The remaining steps in the analysis to construct quasi-steady-states are similar to that for the Schnakenberg and Brusselator models analyzed in \S \ref{sec:sch_equil} and \S \ref{sec:bruss_equil}, respectively.

In place of \eqref{Gprime_sch}, we define $\mathcal G_g^{\prime}(\xi)$ by
\begin{equation}\label{Gprime_gm}
    \mathcal G_g^{\prime}(\xi) \equiv - R_g(\xi) \, f(\xi) = \frac{\xi (2 \kappa - \xi)}{(\xi - \kappa)^2} \left(\frac{\xi^2}{\xi - \kappa} - \xi^2\right) > 0 \, , \quad \mbox{on} \quad \kappa < \xi < 2 \kappa \, .
\end{equation}
A first integral of \eqref{Gprime_gm} yields
\begin{equation}\label{G_GM}
    \mathcal G_g(\xi) = -\frac{\kappa^4}{2 (\xi - \kappa)^2} + \frac{\kappa^3 (\kappa - 2)}{\xi - \kappa} - 2 \kappa^3 \log(\xi - \kappa) - \kappa (\kappa + 1) \xi + \frac{1}{3} \xi^3 - \frac{1}{2} \xi^2 \, , \quad \mbox{on} \quad \kappa < \xi < 2 \kappa \, .
\end{equation}
Upon first multiplying \eqref{gm:outer_prob} by $f(\mathcal A) \mathcal A_x$ and then integrating, we obtain using the monotonicity of $\mathcal A(x)$ and \eqref{Gprime_gm} that
\begin{equation}\label{gm:int_1}
    - \frac{D_L}{2} \left[f(\mathcal A) \mathcal A_x\right]^2 = \int_x^\ell f(\mathcal A) R_g(\mathcal A) \mathcal A_x \, dx = - \int_{\mathcal A(x)}^\mu \mathcal G_g^{\prime}(\mathcal A) \, d \mathcal A = \mathcal G_g(\mathcal A) - \mathcal G_g(\mu) \, .
\end{equation}
Here $\mathcal G_g(\xi)$ is given in \eqref{G_GM} and $\mu \equiv \mathcal A(\ell)$ satisfies $\kappa < \mu \leq 2 \kappa$. By taking the positive square root in \eqref{gm:int_1}
\begin{equation}\label{gm:flux}
    f(\mathcal A) \mathcal A_x = \sqrt{\frac{2}{D_L}} \sqrt{\mathcal G_g(\mu) - \mathcal G_g(\mathcal A(x))} \, .
\end{equation}
Letting $x \to 0^+$ in \eqref{gm:flux} and imposing the matching condition \eqref{gm:match_2} we conclude that $H_0$ is related to $\mu$ by the nonlinear algebraic equation
\begin{equation}\label{gm:cond1}
    \frac{3 H_0^2}{\varepsilon_L \sqrt{2 D_L}} \left(1 - 2 \gamma\right)^{1/2} = \sqrt{\mathcal G_g(\mu) - \mathcal G(\mathcal A(0^+))} \, ,
\end{equation}
where $\mathcal A(0^+) > \kappa$ and $\gamma$ are given in terms of $H_0$ by \eqref{gm:match_1} and \eqref{gm:gam}, respectively. By using the scaling relation \eqref{sch:scale} we can re-write \eqref{gm:cond1} as a nonlinear algebraic equation for $H_{0, L} \equiv H_0/\varepsilon_L$ and $\mu$ in the form
\begin{subequations}\label{gm:cond1_new}
    \begin{equation}\label{gm:cond1_new_a}
        \frac{3 H_{0, L}^2}{\sqrt{2}} \left(\frac{\varepsilon}{\sqrt{D}}\right) \left(1 - 2 \gamma\right)^{1/2} = \sqrt{\mathcal G_g(\mu) - \mathcal G(\mathcal A(0^+))} \, ,
    \end{equation}
    where $\mathcal A(0^+)$ and $\gamma$ are given in terms of $H_{0, L}$ by
    \begin{equation}\label{gm:cond1_new_b}
        \mathcal A(0^+) = \gamma H_{0, L} \, , \quad \mbox{where} \quad \gamma = \frac{1 - \sqrt{1 - {4 \kappa/H_{0, L}}}}{2} \, .
    \end{equation}
\end{subequations}
Next, upon integrating \eqref{gm:flux}, and proceeding as for the Schnakenberg and Brusselator models in \S \ref{sec:sch_equil} and \S \ref{sec:bruss_equil}, respectively, we obtain that $\mu = \mathcal A(\ell)$ is related to $H_{0, L}$ when $0 < \kappa < 1$ by
\begin{equation}\label{properchi_gm}
    \chi_g(\mu) \equiv - 2 \frac{\sqrt{\mathcal G_g(\mu) - \mathcal G_g(\mathcal A(0^+))}}{R_g(\mathcal A(0^+))} + 2 \int_{\mathcal A(0^+)}^\mu \, \frac{\sqrt{\mathcal G_g(\mu) - \mathcal G_g(\xi)}}{\left[R_g(\xi)\right]^2} \, R_g^{\prime}(\xi) \, d\xi = \sqrt{\frac{2}{D_L}} \ell \, .
\end{equation}
Here, $R_g(\xi)$ and $\mathcal G_g(\xi)$ are defined in \eqref{gm:outer_def} and \eqref{G_GM}, respectively. The positivity of $\chi_g^{\prime}(\mu)$ on $\kappa < \mathcal A(0^+) < \mu < 2 \kappa$ when $0 < \kappa < 1$ follows since $R_g(\xi) < 0$, $R_g^{\prime}(\xi) > 0$ and $\mathcal G_g^{\prime}(\xi) > 0$ on $\kappa < \xi < 2 \kappa$. The maximum value of $\mu = v(\ell)$ occurs when $\mu = 2 \kappa$, and we define
\begin{equation}\label{gm:xmax}
    \chi_{g, \max} \equiv \chi_g(2 \kappa) \, , \qquad \mu_{\max} \equiv 2 \kappa \, .
\end{equation}
We summarize our asymptotic construction of a one-spike solution to
\eqref{GM_Lagrange} as follows:

\begin{result}
    On the range $0 < \kappa < 1$, our asymptotic construction of a one-spike solution to \eqref{GM_Lagrange} on $|x| \leq \ell$ when $\varepsilon_L \ll 1$ reduces to solving the coupled nonlinear algebraic equations \eqref{gm:cond1_new} and \eqref{properchi_gm} for $H_{0, L} \equiv H_0/\varepsilon_L$ and $\mu = \mathcal A(\ell)$ in terms of $\kappa$, ${\varepsilon/\sqrt{D}}$ and $\ell/\sqrt{D_L}$. For a $K$-spike solution on $- 1 \leq x \leq 1$, where we must set $\ell = 1/K$, the nucleation threshold $D_{L, K}^{nuc}$ is found by first setting $\mu=2\kappa$ and solving \eqref{gm:cond1_new} for $H_{0, L}$. This determines $\mathcal A(0^+)$ from \eqref{gm:cond1_new_b} as needed in calculating $\chi_{g, \max}$ from \eqref{gm:xmax}. In this way, in terms of $\chi_{g, \max}$ as defined in \eqref{gm:xmax}, spike nucleation for a $K$-spike solution is predicted to occur when
    \begin{equation}\label{gm:nuc_D}
        D_L < D_{L, K}^{nuc} \equiv \frac{2}{K^2 \left[\chi_{g, \max}\right]^2} \, .
    \end{equation}
    Equivalently, in terms of the domain half-length $L$, we predict that spike nucleation occurs when
    \begin{equation}\label{gm:nuc_length}
        L > L_K^{nuc} \equiv \sqrt{\frac{D}{2}} K \chi_{g, \max} \, .
    \end{equation}
\end{result}

In Fig.~\ref{fig:GM_nuc_L} we show a very favorable comparison between the asymptotic prediction $L_{K}^{nuc}$ in \eqref{gm:nuc_length} of the critical length for the onset of spike nucleation with the location of the saddle-node bifurcation point as computed from \eqref{GM_Lagrange} using \textit{pde2path} \cite{pde2path} when $\varepsilon = 0.01$ and $D = 1$.

\begin{figure}[htbp]
    \centering
    \includegraphics[width=0.55\textwidth, height=6.0cm]{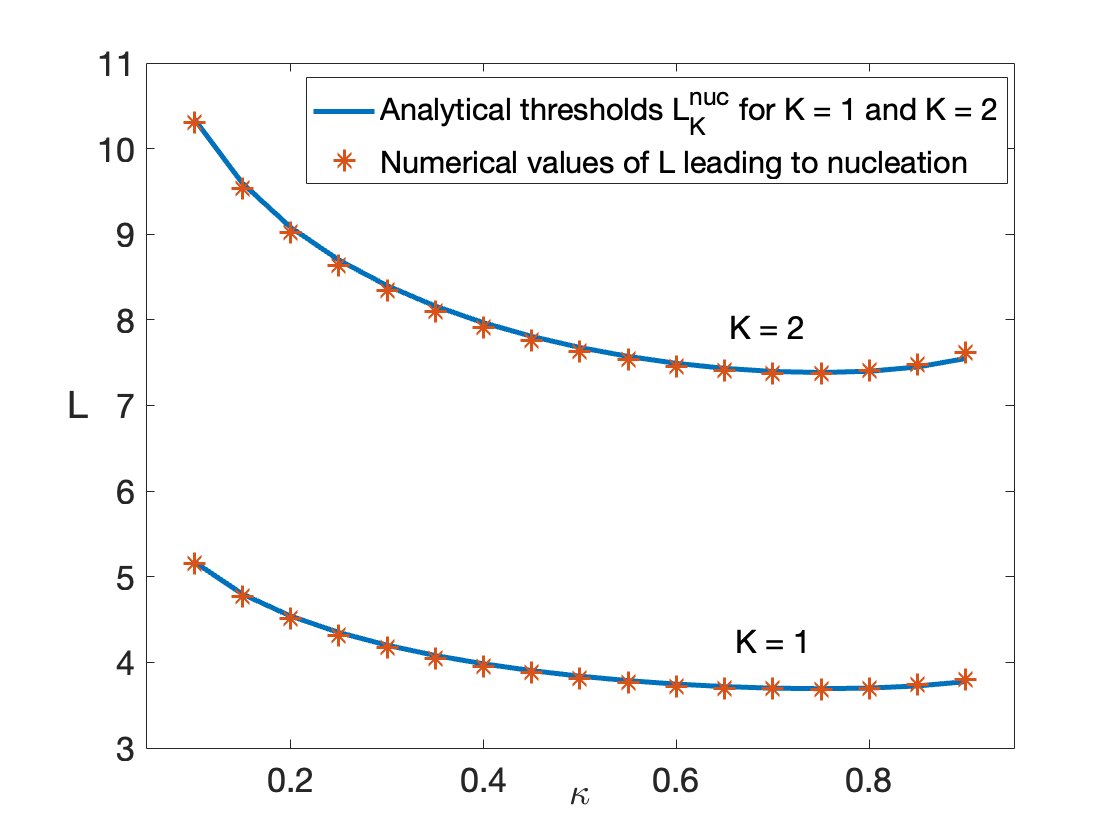}
    \caption{Comparison between asymptotic and numerical results for the nucleation threshold $L_K^{nuc}$ versus $\kappa$ for $K = 1$ and $K = 2$. The solid curves are the asymptotic results given in \eqref{gm:nuc_length}. The red stars are the numerical results for the saddle-node bifurcation point as computed from \eqref{GM_Lagrange} using \textit{pde2path} \cite{pde2path}. Parameters: $\varepsilon = 0.01$ and $D = 1$.}
    \label{fig:GM_nuc_L}
\end{figure}

\subsection{Global bifurcation diagram and full PDE simulations}\label{sec:GM_fullpde}

For $\kappa = 0.5$, $\varepsilon = 0.01$ and $D = 1$, the numerically computed global bifurcation diagram, as obtained by path-following the single-spike branch for the GM model \eqref{GM_Lagrange} using \textit{pde2path} \cite{pde2path}, is shown in Fig.~\ref{fig:GM_nuc_bif}. The saddle-node bifurcation point at point (b) in this figure given by $L \approx 3.81$, signifying the onset of spike nucleation, is well-approximated by the value $L_1^{nuc}$ in \eqref{gm:nuc_length} from our asymptotic theory. This global bifurcation diagram, and the resulting solution behavior on the various branches, is qualitatively identical to that for the Schnakenberg and Brusselator models shown in Fig.~\ref{fig:Sch_bif_nuc} and Fig.~\ref{fig:Bru_bif_nuc}, respectively, for parameter regimes in these models where spike nucleation is the dominant mechanism for creating new spikes as $L$ increases.

\begin{figure}[htbp]
    \centering
    \includegraphics[width=0.95\textwidth, height=6.0cm]{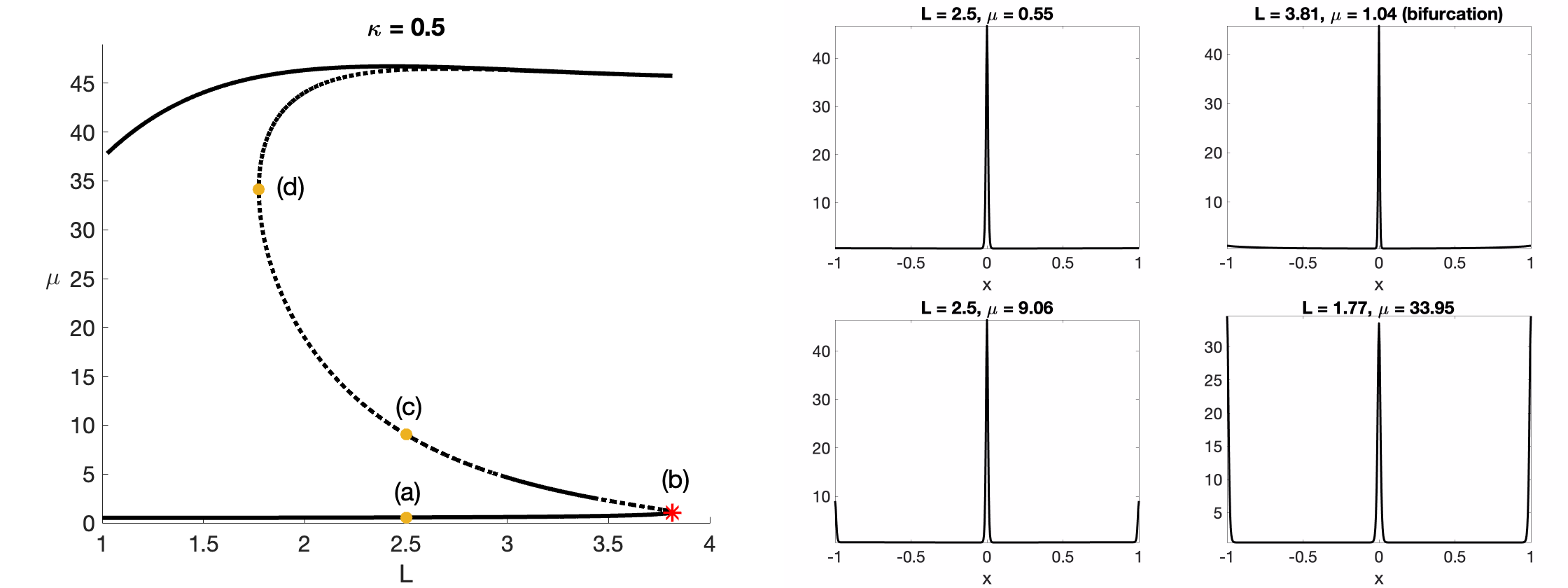}
    \caption{Left panel: Global bifurcation diagram of $\mu=v(1)$ versus $L$ for single-spike steady-states for the GM model \eqref{GM_Lagrange} as computed using \textit{pde2path} \cite{pde2path} for $\kappa = 0.5$, $\varepsilon = 0.01$ and $D = 1$. Since $\kappa<1$, we predict that spike nucleation occurs as $L$ is increased starting from point (a) on the linearly stable lower branch. The red star is the saddle-node point that signifies the onset of spike nucleation behavior. The global bifurcation diagram is similar to that for the Schnakenberg and Brusselator models whenever spike nucleation occurs. Right panel: Spike profile $\mathcal A(x)$ and bifurcation values at the indicated points in the left panel.}
    \label{fig:GM_nuc_bif}%
\end{figure}

\begin{figure}[htbp]
    \centering
    \includegraphics[width=0.66\textwidth, height=6.0cm]{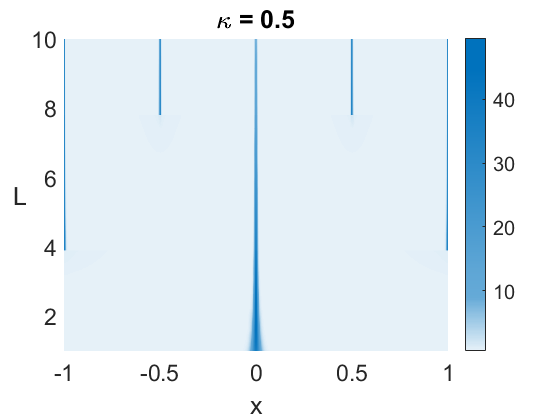}
    \caption{Time-dependent PDE simulations of \eqref{GM_Lagrange} using \textit{FlexPDE} \cite{flexpde2015} illustrating spike nucleation behavior as the domain half-length $L$ slowly increases as $L = e^{\rho t}$ with $\rho = 10^{- 4}$. The transitions where boundary spikes first emerge and then later when spikes are nucleated between the interior and boundary spikes are well-approximated by the critical lengths in Fig.~\ref{fig:GM_nuc_L}. Parameters: $\varepsilon = 0.01$, $D = 1$ and $\kappa=0.5$.}
    \label{fig:GM_flexpde_plot}%
\end{figure}

For $\varepsilon = 0.01$, $D = 1$ and $\kappa = 0.5$, in Fig.~\ref{fig:GM_flexpde_plot} we show full time-dependent PDE simulations of \eqref{GM_Lagrange} computed using \textit{FlexPDE} \cite{flexpde2015} as the domain half-length $L$ slowly increases in time by $L = e^{\rho t}$ with $\rho = 10^{- 4}$. As $L$ increases, boundary spikes first emerge when $L\approx 3.89$. As $L$ is increased further above $L \approx 7.8$, new spikes are nucleated at the midpoint locations between the interior spike and each boundary spike. These two transition values of $L$, estimated from our time-dependent PDE results, are well-approximated by the critical thresholds $L_1^{nuc}$ and $L_2^{nuc}$ with $\kappa = 0.5$ for the non-existence of the outer solution as shown in Fig.~\ref{fig:GM_nuc_L}.

\subsection{Asymptotics of the outer solution for $\kappa$ small} \label{sec:gm_small_a}

In this subsection, we approximate the outer problem \eqref{gm:outer_probfull} by a linear problem when $\kappa\ll 1$. In the limit $\kappa \ll 1$, no spike nucleation behavior occurs.

For $\kappa \ll 1$, we obtain from \eqref{GM_out1} that $\mathcal A \sim \kappa + \mathcal O(\kappa^2)$ in the outer region. As a result, from \eqref{GM_out2}, we obtain that $\mathcal H$ satisfies
\begin{equation}\label{gm_small:outer}
    D_L \mathcal H_{xx} - \mathcal H = \kappa^2 + \mathcal O(\kappa^3) \, , \quad 0^+ < x < \ell \, ; \qquad \mathcal H_x(\ell) = 0 \, ,
\end{equation}
with the matching condition $\mathcal H(0^+) = H_0/\varepsilon_L$. Upon neglecting the $\mathcal O(\kappa^3)$ term in \eqref{gm_small:outer}, we calculate that
\begin{equation}\label{gm:small-solve}
    \mathcal H(x) = \kappa^2 + \left(\frac{H_0}{\varepsilon_L} -
    \kappa^2\right) \frac{\cosh\left[(\ell - |x|)/\sqrt{D_L}\right]}
    {\cosh\left(\ell/\sqrt{D_L}\right)} \, .
\end{equation}
By imposing the second matching condition given by \eqref{gm:hx_match}, we obtain that $H_0$ must satisfy
\begin{equation}\label{gm:small_H0eq}
    \frac{H_0}{\varepsilon_L \sqrt{D_L}} \tanh\left[\frac{\ell}{\sqrt{D_L}}\right] = \frac{3 H_0^2}{D_L \varepsilon_L} \left(1 - 2 \gamma\right)^{1/2} \, .
\end{equation}
For $\kappa \ll 1$, \eqref{gm:cond1_new_b} yields $\gamma = \mathcal O(\kappa \varepsilon_L/H_0) \ll 1$. By setting $\gamma \ll 1$ in \eqref{gm:small_H0eq}, we get that
\begin{equation}\label{gm:small_H0val}
    H_0 \sim \frac{\sqrt{D_L}}{3} \tanh\left(\frac{\ell}{\sqrt{D_L}}
    \right) \, .
\end{equation}
In this way, for $\kappa \ll 1$ our asymptotic result for $\mathcal H(0) = H_0/\varepsilon_L$ is, by using the scaling relation \eqref{scale}, that
\begin{equation}\label{gm:small_H0fin}
    \mathcal H(0) \sim \frac{\sqrt{D}}{\varepsilon}
    \tanh\left[\frac{\ell}{\sqrt{D_L}}\right] \, .
\end{equation}
In contrast, for $\kappa = \mathcal O(1)$, we predict that $\mathcal H(0) \sim H_{0, L}$, where $H_{0, L}$ must be computed from the coupled nonlinear algebraic system \eqref{gm:cond1_new} and \eqref{properchi_gm}.

Although \eqref{gm:small_H0val} and \eqref{gm:small_H0fin} were derived only for the limit $\kappa \ll 1$, in Fig.~\ref{fig:GM_H0} we show that the simple closed-form result \eqref{gm:small_H0val} agrees rather well with full numerical results computed from \eqref{GM_Lagrange} using \textit{pde2path} \cite{pde2path} as $D_L$ is varied for both $\kappa = 0.1$ and $\kappa = 0.5$. In particular, the results in Fig.~\ref{fig:GM_H0} show that $H_0$ has a rather weak dependence on $\kappa$.

\begin{figure}[htbp]
    \centering
    \includegraphics[width=0.55\textwidth, height=6.0cm]{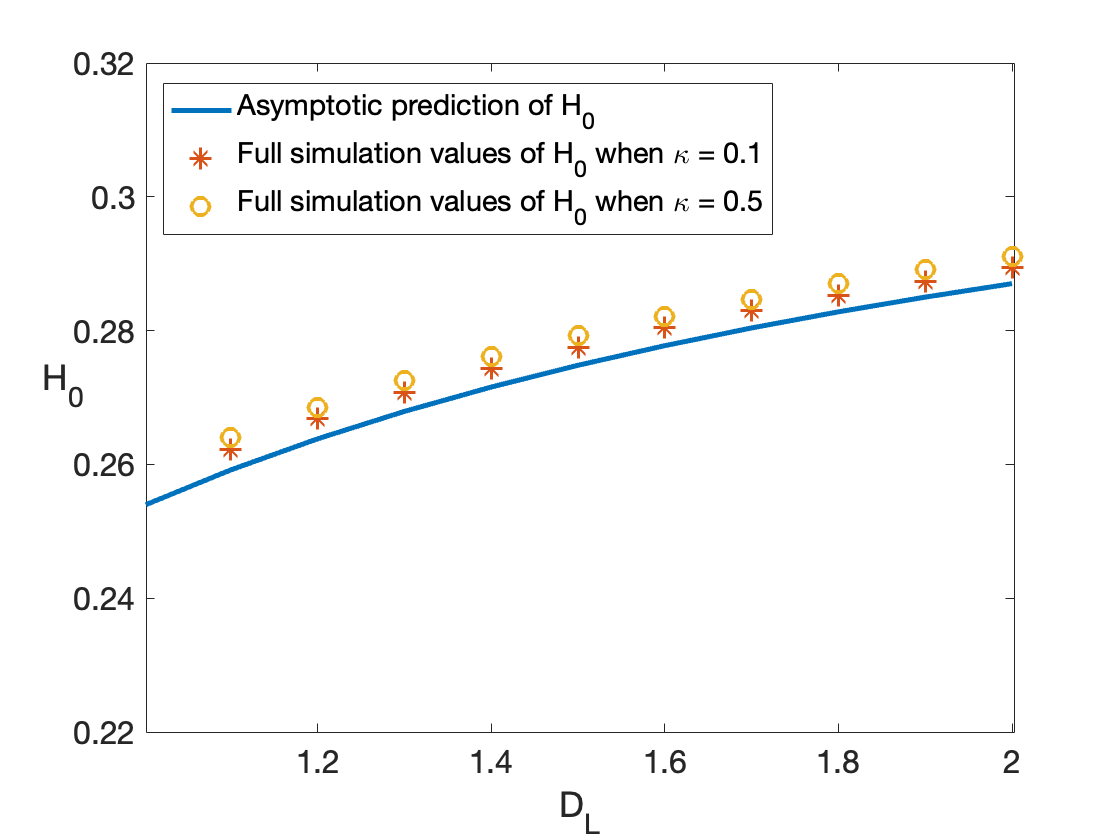}
    \caption{Comparison between asymptotic and numerical results for $H_0$ for two values of $\kappa$ as $D_L$ is varied. The solid curve is the asymptotic result given in \eqref{gm:small_H0val} derived in the limit $\kappa\ll 1$. The red and yellow stars are obtained by full simulations of the GM model \eqref{GM_Lagrange} using \textit{pde2path} \cite{pde2path} with $\kappa = 0.1$ and $\kappa = 0.5$, respectively. Parameters: $\varepsilon = 0.01$ and $\ell = 1$.}
    \label{fig:GM_H0}%
\end{figure}

\section{Conditions for the absence of nucleation instability} \label{sec:no_nucleation}

Our analyses in \S \ref{sec:sch}--\S \ref{sec:gm} have shown that, on certain parameter ranges, spike nucleation behavior can occur as the domain length increases for the Schnakenberg, Brusselator and GM models in the limit $\varepsilon_L\ll 1$. An essential ingredient for the occurrence of spike nucleation behavior was that there is a saddle-node bifurcation point at some finite critical value of $D_L$, at which the outer solution ceases to exist.

One key condition for the existence of this saddle-node point is that the parameters in each RD system are such that there is no spatially homogeneous steady-state solution on the range of well-posedness for the outer problem. More specifically, this is evidenced by our derivation in Lemmas \ref{lemma:sch} and \ref{lemma:bruss} where we require that $R_s(v) < 0$ in \eqref{Sch_fR} and $R_b(v) < 0$ in \eqref{Bruss_fR} for the Schnakenberg and Brusselator models, respectively, in the range $a < v < 2 a$ where the outer problem is well-posed. Under this condition, we obtain, from \eqref{properchi_sch} and \eqref{properchi_bruss}, that the integrals defining $\chi_s(\mu)$ and $\chi_b(\mu)$ are monotone increasing in $\mu$ but have {\em finite values} as $\mu \to 2 a$ from below. These limiting values, based on the non-existence of the outer asymptotic solution, were used in \eqref{sch:thresh_nucD} and \eqref{bruss:thresh_nucD} for the Schnakenberg and Brusselator models, respectively, to obtain a leading-order prediction for the critical threshold of $D_L$ where a saddle-node point along the single-spike solution branch must occur. The existence of such a saddle-node point is the signature for the onset of spike nucleation behavior. A completely similar characterization occurs for the GM model as evidenced by Lemma \ref{lemma:gm} and the corresponding result given in \eqref{gm:nuc_D}.

For parameter ranges where each of our RD systems has a spatially homogeneous steady state, denoted by $v_{\infty}$, on the range of well-posedness of the outer problem, it follows that $R(v_{\infty})=0$ with $R^{\prime}(v_{\infty})>0$, where $v_\infty$ must satisfy the bounds in $v$ for which the outer problem is well-posed. Here, $R(v)$ is the right-hand side of the three outer problems \eqref{Sch_fullouter}, \eqref{Bruss_fullouter} or \eqref{gm:outer_probfull}. More specifically, we calculate that 
%\begin{subequations}
    \begin{align}
        v_\infty &= b + a < 2 a \, , \quad \mbox{when} \quad a > b \, , \quad \mbox{(Schnakenberg model)} \, ,
        \label{eq:vinfty_sch}
        \\
        v_\infty &= \frac{a}{1 - f} < 2 a \, , \quad \mbox{when} \quad 0 < f < 1/2 \, , \quad \mbox{(Brusselator model)} \, ,
        \label{eq:vinfty_bru}
        \\
        v_\infty &= 1 + \kappa < 2 \kappa \, , \quad \mbox{when} \quad \kappa > 1 \, , \quad \mbox{(GM model)} \, .
        \label{eq:vinfty_GM}
    \end{align}
%\end{subequations}
To analyze the outer problem for each of our three RD models, we need only modify our previous analysis by requiring that $\mu \equiv v(\ell) < v_\infty$. More specifically, we observe that $\chi(\mu)$ for either \eqref{properchi_sch}, \eqref{properchi_bruss} or \eqref{properchi_gm} is monotone increasing on $v(0^+) < \mu < v_\infty$ and that $\lim_{\mu \to v_\infty^-} \chi(\mu) = +\infty$ owing to the non-integrability of the integrals at $\mu = v_\infty$ that define $\chi(\mu)$. As a result, we conclude that {\em for any} $D_L > 0$ there is a unique $\mu = \mu^{\star}$ in $v(0^+) < \mu^{\star} < v_\infty$ where either \eqref{properchi_sch}, \eqref{properchi_bruss}, or \eqref{properchi_gm} has a solution.

\begin{figure}[htbp]
    \centering
    \includegraphics[width=0.99\textwidth, height=5.5cm]{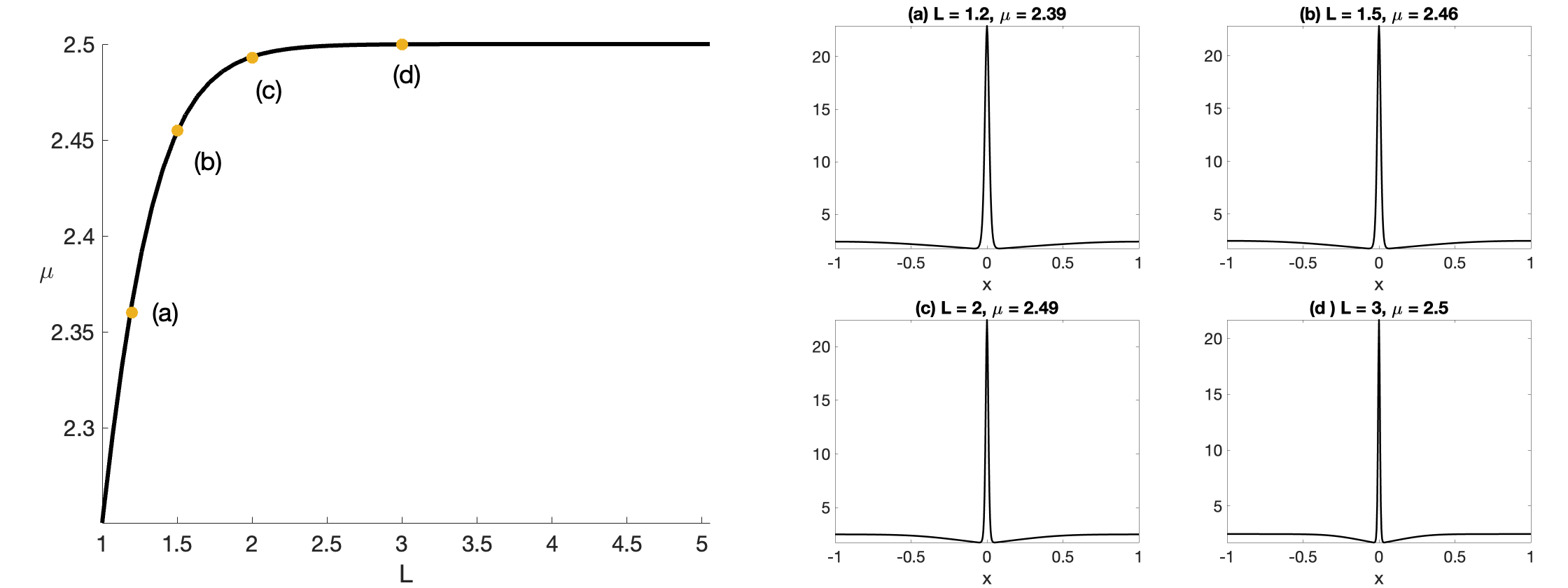}
    \caption{Left panel: Global bifurcation diagram of $\mu = v(1)$ versus $L$ for single-spike steady-state solutions of the Schnakenberg model \eqref{Sch_core} computed using pde2path \cite{pde2path}. Parameters: $a = 1.5$, $b = 1$, $\varepsilon = 0.01$ and $D = 2$. Since $a > b$, a spatially uniform steady-state occurs and there is no longer any saddle-node bifurcation on this branch. Right panel: Spike profile $v(x)$ at the indicated points in the left panel.}
    \label{fig:Sc_novel_bif}
\end{figure}

In Fig.~\ref{fig:Sc_novel_bif} we show a global bifurcation diagram obtained from path-following a single-spike steady-state solution of the Schnakenberg model \eqref{Sch_core} using pde2path \cite{pde2path} for the parameters $a = 1.5$, $b = 1$, $\varepsilon = 0.01$ and $D = 2$. Since $a > b$, we observe that there is no longer any saddle-node bifurcation along this branch, and so we predict that no spike nucleation events will occur. Qualitatively similar global bifurcation diagrams for the Brusselator model \eqref{Brusselator_Lagrange} with $f = 0.3 < 1/2$ and for the GM model \eqref{GM_Lagrange} with $\kappa = 1.5 > 1$ are shown in Fig.~\ref{fig:Br_novel_bif} and Fig.~\ref{fig:GM_novel_bif}, respectively.

In summary, we conclude that the outer problems for our three RD models are uniquely solvable for any $D_L>0$ whenever the RD system has a spatially homogeneous steady-state solution on the range where the outer problem is well-posed. This precludes the existence of a saddle-node bifurcation point in $D_L$ for the outer solution, and, consequently, no spike nucleation events will occur as the domain length increases. Moreover, in these parameter regimes, we remark that there exists a spike solution on the infinite line since we can impose $v \to v_\infty$ as $x \to \pm \infty$. We emphasize that since the outer problems were derived under the assumption of a large diffusivity ratio, the discussion above only pertains to the range where $D_L = \mathcal O(1)$ and not to the weak interaction regime where $D_L = \mathcal O(\varepsilon_L^2)$.

\begin{figure}[htbp]
    \centering
    \includegraphics[width=0.99\textwidth, height=5.5cm]{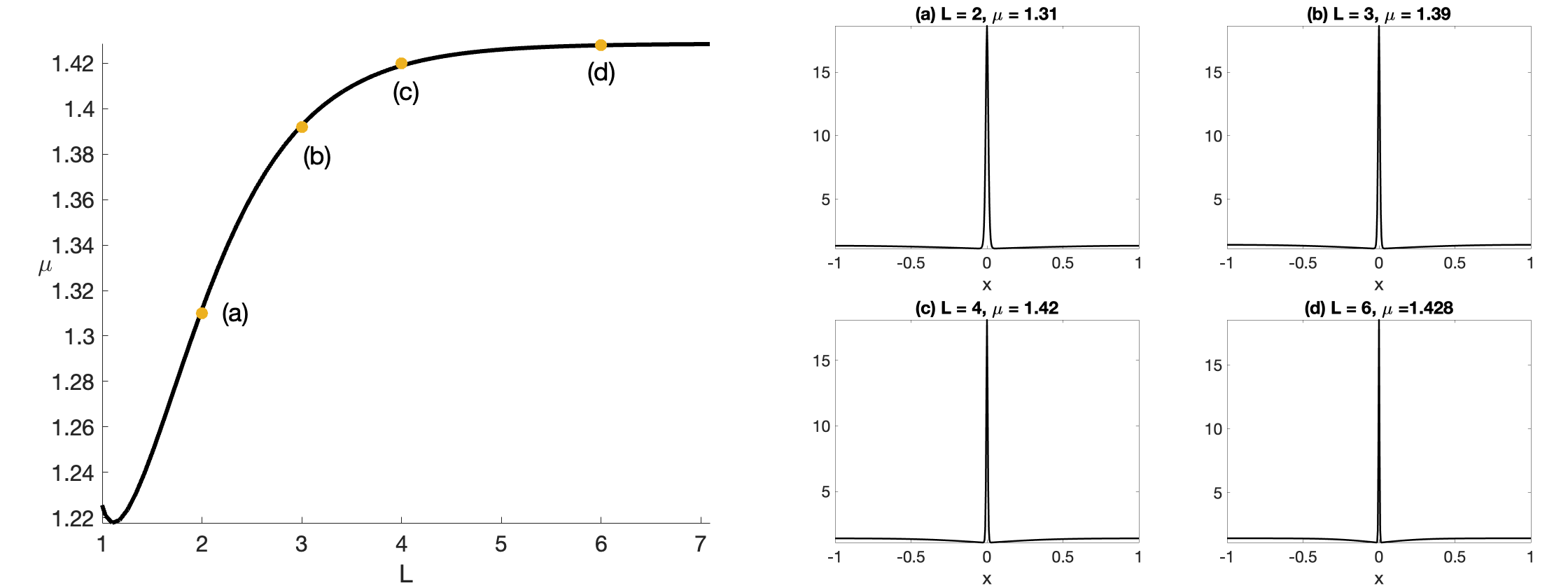}
    \caption{Same plot as in Fig.~\ref{fig:Sc_novel_bif} but for the Brusselator model \eqref{Brusselator_Lagrange} with $f = 0.3$, where a spatially homogeneous steady-state occurs. Remaining parameters: $a = 1$, $\varepsilon = 0.01$ and $D = 2$. Since $f < 1/2$, there is no longer any saddle-node bifurcation and spike nucleation is precluded.}
    \label{fig:Br_novel_bif}
\end{figure}

\begin{figure}[htbp]
    \centering
    \includegraphics[width=0.99\textwidth, height=5.5cm]{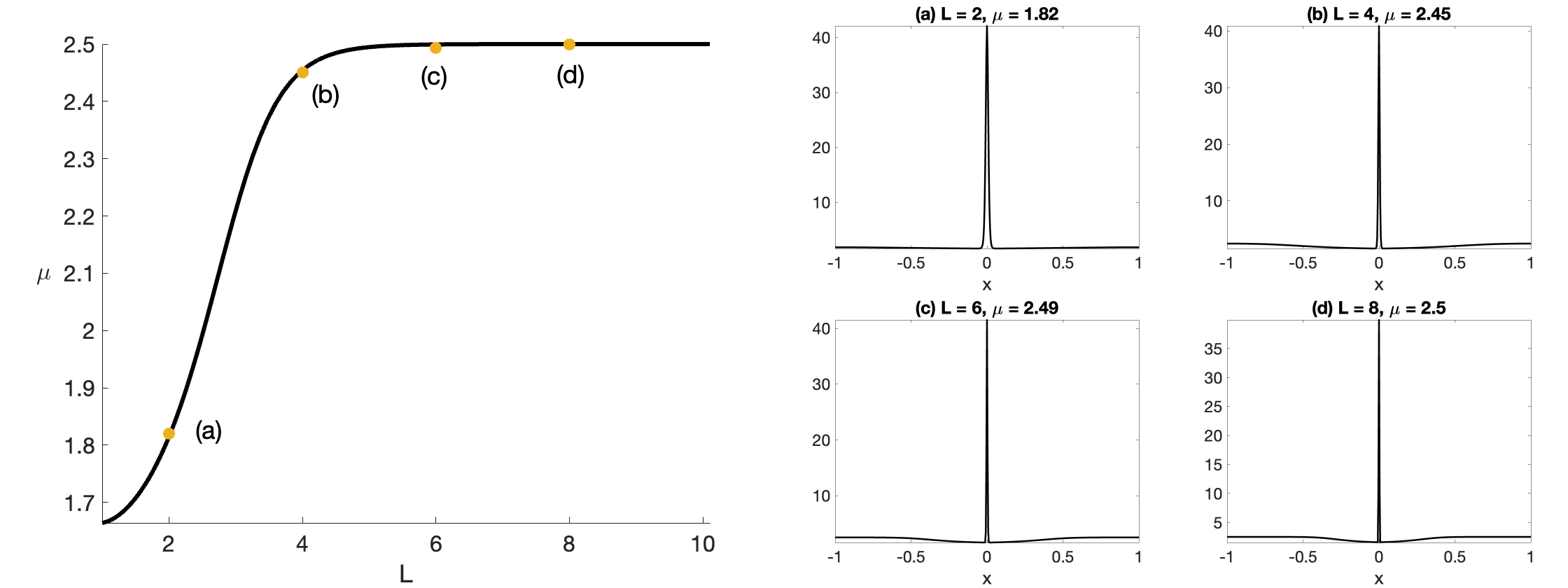}
    \caption{Same plot as in Fig.~\ref{fig:Sc_novel_bif} but for the GM model \eqref{GM_Lagrange} with $\kappa = 1.5$. Remaining parameters: $\varepsilon = 0.01$ and $D = 1$. Since $\kappa > 1$, a spatially uniform steady-state occurs. There is no longer any saddle-node bifurcation and spike nucleation behavior is precluded.}
    \label{fig:GM_novel_bif}
\end{figure}

For $\varepsilon_L\ll 1$, our next result shows that whenever $b < a$, spike self-replication will not occur for the Schnakenberg model \eqref{Sch_core}. As a result, we conclude that there is always a one-spike solution for any $L > 0$.

\begin{lemma}\label{lemma:norep_sch}
    Suppose that $0<b<a$. Then, for $\varepsilon/\sqrt{D} \ll 1$, spike self-replication will not occur for \eqref{Sch_core}.
\end{lemma}

\begin{proof}
    Since $\mathcal G_s^{\prime}(\mu) > 0$ on $a < \mu < v_\infty \equiv a + b$, we have for \eqref{sch:integral_Bfinal} that $B$ is a monotonically increasing function of $\mu$ on this interval. As such, by setting $\mu = v_\infty$ and $v(0^+) = a$, it follows that if we can show that $B = \sqrt{2} \left[\mathcal G_s (a + b) - \mathcal G_s(a)\right]^{1/2}$ satisfies $B < B_c \approx 1.347$ whenever $0 < b < a$, then for any $\mu^{\star}$ in $a < \mu^{\star} < v_\infty$, there is a unique $B < B_c$ on the primary branch for the core problem (see the left panel of Fig.~\ref{fig:bif_core}). As a result, a one-spike solution will always exist and spike self-replication will not occur. To establish that $B < B_c$ for $b < a$, we first use \eqref{G_Sch} for $\mathcal G_s(\xi)$ to obtain after some algebra that
    \begin{equation*}
        \mathcal G_s(v_\infty) - \mathcal G_s(a) = - \frac{b}{a + b} + \log\left(\frac{a + b}{a}\right) \, .
    \end{equation*}
    Upon defining $z \equiv {(a + b)/a}$, where $1 < z < 2$, we conclude that
    \begin{equation}\label{sch:norep}
        B = \mathcal F_s(z) \equiv \sqrt{2} \left(- 1 + \frac{1}{z} + \log{z}\right)^{1/2} \, , \quad \mbox{for} \quad 1 < z < 2 \, .
    \end{equation}
    We readily calculate that $\mathcal F_s(z) > 0$ and $\mathcal F_s^{\prime}(z) > 0$ on $1 < z < 2$, and so the maximum value of $B$ is
    \begin{equation*}
        B_{\max} = \lim_{z \to 2^-} \mathcal F_s(z)= \sqrt{- 1 + 2 \log{2}} \approx 0.386 \, .
    \end{equation*}
    Since $B_{\max} < B_c \approx 1.347$, it follows that $B(z) < B_c$ on $1 < z < 2$. We conclude that spike self-replication can never occur when $b < a$.
\end{proof}

Next, we establish a similar result for the Brusselator model \eqref{Brusselator_Lagrange} that is based, in part, on the range of values of the fold point $B_c(f)$ for the core problem plotted on $0 < f < 1/2$, as seen from Fig.~\ref{fig:Bru_core_bif_c}.

\begin{lemma} \label{lemma:norep_bruss}
    Suppose that $0 < f < 1/2$. Then, for $\varepsilon/\sqrt{D} \ll 1$, spike self-replication can not occur for the Brusselator model \eqref{Brusselator_Lagrange}.
\end{lemma}

\begin{proof}
    The proof is similar to that in Lemma \ref{lemma:norep_sch}. Since $\mathcal G_b^{\prime}(\mu) > 0$ on $a < \mu < v_\infty = a/(1 - f)$ when $0 < f < 1/2$, it follow from \eqref{Bruss:integral_Bfinal} that $B$ is a monotonically increasing function of $\mu$ on this interval in $\mu$ when $0 < f < 1/2$. As such, by setting $\mu=v_{\infty}$ and $v(0^+) = a$, it follows that if we can show that $B = \sqrt{2} f^{- 1} \left[\mathcal G_b(v_\infty) - \mathcal G_b(a)\right]^{1/2}$ satisfies $B < B_c(f)$ on $0 < f < 1/2$, then for any $\mu^{\star}$ in $a < \mu^{\star} < v_\infty$ and any $f$ in $0 < f < 1/2$, there is a unique $B < B_c$ on the primary branch for the Brusselator core problem (see Fig.~\ref{fig:Bru_core_bif}). As a result, a one-spike solution will always exist and spike self-replication will not occur. To establish that $B < B_c(f)$ for $0 < f < 1/2$, we first use \eqref{G_B} for $\mathcal G_b(\xi)$ to obtain after some algebra that
    \begin{equation*}
        \mathcal G_s(v_\infty) - \mathcal G_s(a) = - f (1 - f) - (1 - f) \log(1 - f) \, .
    \end{equation*}
    By using \eqref{Bruss:integral_Bfinal}, we conclude that
    \begin{equation}\label{bruss:norep}
        B = \mathcal F_b(f) \equiv \sqrt{2} \left[\frac{(1 - f)}{f^2} \left(- f - \log(1 - f)\right)\right]^{1/2} \, , \quad \mbox{for} \quad 0 < f < 1/2 \, .
    \end{equation}
    A simple calculation shows that $\mathcal F_b(f) > 0$ and $\mathcal F_b^{\prime}(f) < 0$ on $0 < f < 1/2$. Therefore, by using the Taylor series $\log(1 - f)\sim - f - {f^2/2}$ as $f \to 0$, we calculate that the maximum value of $B$ is
    \begin{equation*}
        B_{\max} = \lim_{f \to 0^+} \mathcal F_b(f) = 1 \, .
    \end{equation*}
    From the numerical results in Fig.~\ref{fig:Bru_core_bif_c}, we observe that $1 = B_{\max} < B_c(f)$ on $0 < f < 1/2$. We conclude that, for any $0 < f < 1/2$, we must have $B < B_c(f)$, and so spike self-replication can never occur.
\end{proof}

Finally, pulling together all the results in this section and \S \ref{sec:bruss} and \S \ref{sec:sch}, we arrive at the two-parameter summaries that were presented in Fig.~\ref{fig:phase_diag}, indicating the type of spike-generating mechanism that occurs as the domain half-length $L$ increases when $\varepsilon/\sqrt{D} \ll 1$.

\section{Dynamic transitions between multi-spike quasi-equilibria}\label{sec:transitions}

In this section, we turn to numerical bifurcation analysis to probe the structure of steady solution branches of our three RD systems on a static domain $|x|\leq L$, with the domain length $L$ as the primary bifurcation parameter. For each system, in Figs.~\ref{fig:bif_Sch_merge}--\ref{fig:bif_GM} we plot bifurcation diagrams of multi-spike equilibria, as computed using the numerical bifurcation and continuation package \textit{pde2path} \cite{pde2path}. In each figure, the vertical axis is the $L_2$ norm of the singularly perturbed solution component. Solution branches in red depict linearly stable $K$-spike equilibria, in which each successive branch to the right corresponds to a half-spike (i.e.~a boundary spike) increment. The blue dashed curves correspond to unstable multi-spike equilibria. In Figs.~\ref{fig:bif_Sch_merge}--\ref{fig:bif_GM}, we have overlaid in yellow the time-dependent numerical results for the $L_2$ norm, as computed from our three RD systems using \textit{FlexPDE} \cite{flexpde2015} with the specific domain growth rate $\rho = \varepsilon^2$. In this way, we can visualize how the quasi-steady state solutions ``jump'' to closely track other steady-state solution branches consisting of more spikes as $L$ slowly increases.

\begin{figure}[htbp]
    \centering
    \includegraphics[width=0.85\textwidth,height=6.0cm]{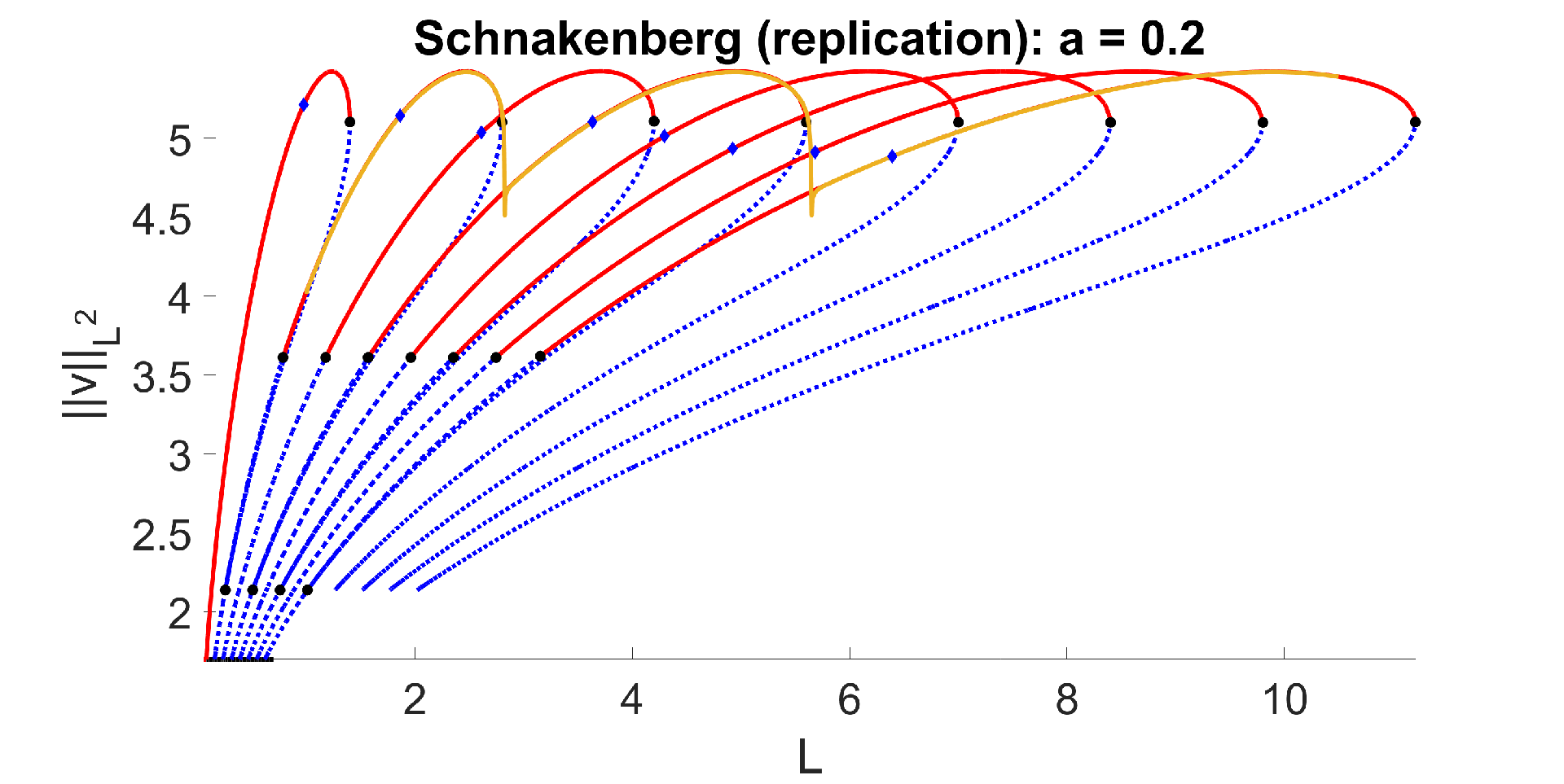}
    \includegraphics[width=0.85\textwidth,height=6.0cm]{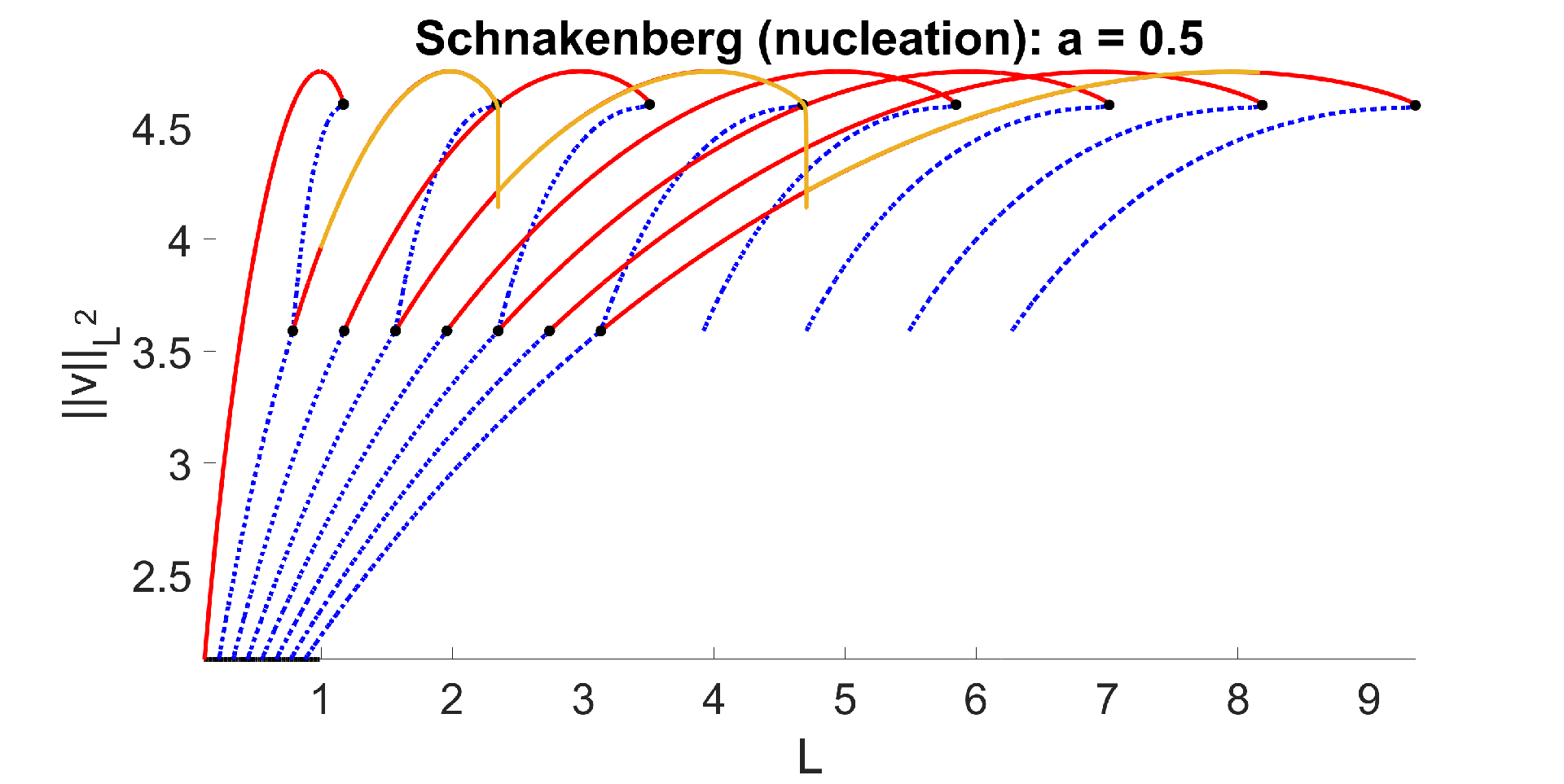}
    \caption{Plots of steady-state solution branches, with successive increments of a half-spike, for linearly stable multi-spike equilibria (red curves) overlaid with the time-dependent solutions (yellow curves) computed from full \textit{FlexPDE} \cite{flexpde2015} simulations of the Schnakenberg model \eqref{Sch_Lagrange}. The leftmost branch corresponds to a half-spike steady-state. The horizontal axis is the bifurcation parameter $L$ and the vertical axis is the $L_2$ norm of $v$. Parameters: $b = 1$, $\varepsilon = 0.04$, $D = 4$ and $\rho = \varepsilon^2 = 0.0016$. Top panel: For $a = 0.2$ spike self-replication occurs along the yellow curve. Bottom panel: For $a = 0.5$ spike nucleation occurs along the yellow curve. Black points represent bifurcation points that could be either a fold point (if the curve folds) or a branch point. Diamonds represent the points from where the solutions were considered for plotting in Figure \ref{fig:bif_Sch_sols}. Furthermore, continuous red (respectively, dotted blue) curves represent stable (respectively, unstable) solutions.}
    \label{fig:bif_Sch_merge}
\end{figure}

\begin{figure}[htbp]
    \centering
    \includegraphics[width=0.85\textwidth,height=6.0cm]{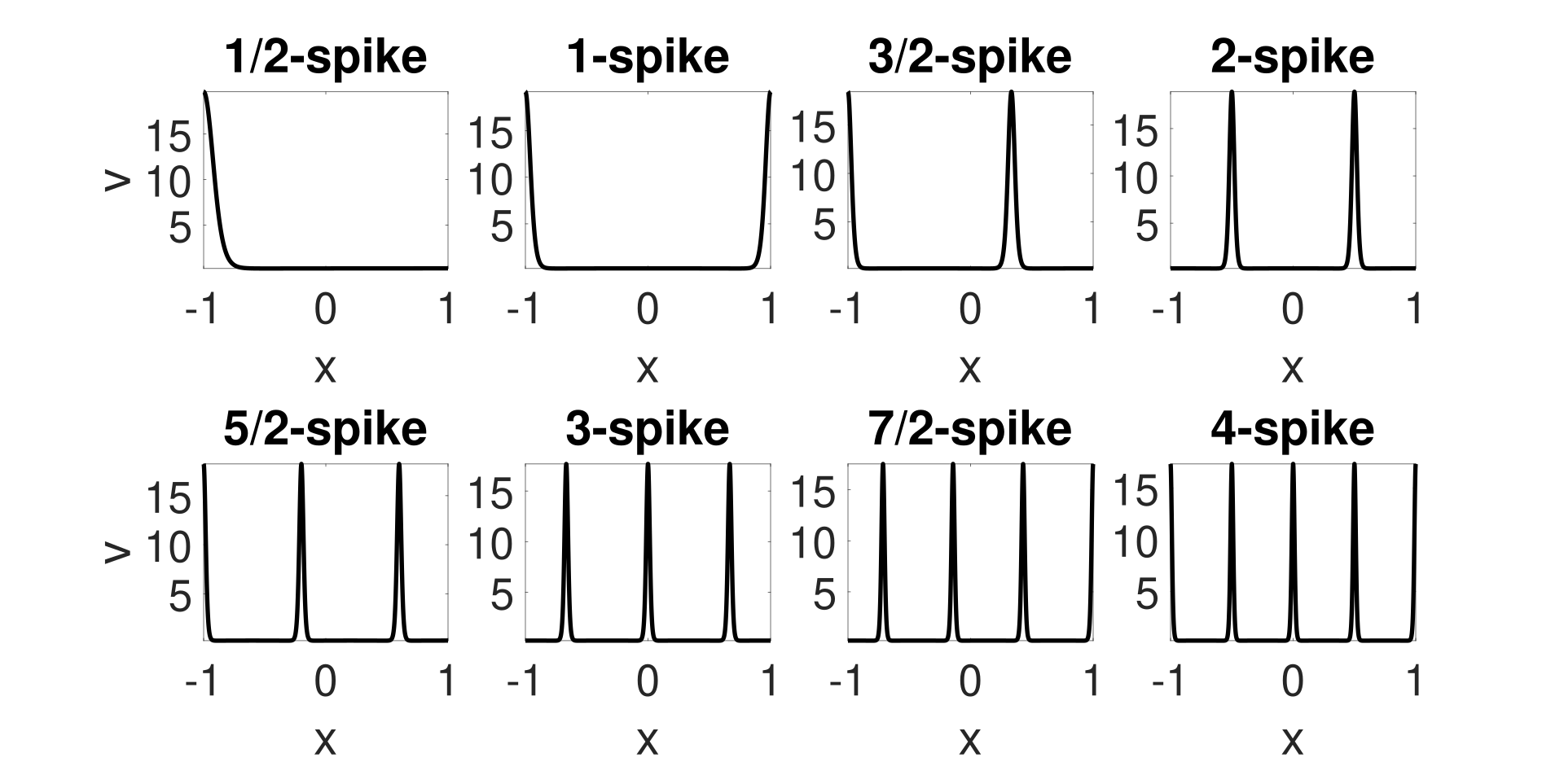}
    \caption{Plots of solutions in each of the branches in the top panel of Figure \ref{fig:bif_Sch_merge}. Each solution corresponds to the diamond-indicated point. Furthermore, the solutions are sorted according to the branch they belong to from left to right. Specifically, the solutions correspond, from left to right and top to bottom, to the parameter values $\left(L, \lVert v \rVert_{L^2}\right) \in \{(0.9729, 5.2111), (1.8597, 5.143), (2.6115, 5.036). (3.6315, 5.1044), (4.2945, 5.0125), (4.9228, 4.9332),
    \\
    (5.678, 4.9125), (6.3921, 4.8852)\}$.}
    \label{fig:bif_Sch_sols}
\end{figure}

\begin{figure}[htbp]
    \centering
    \includegraphics[width=0.85\textwidth,height=6.0cm]{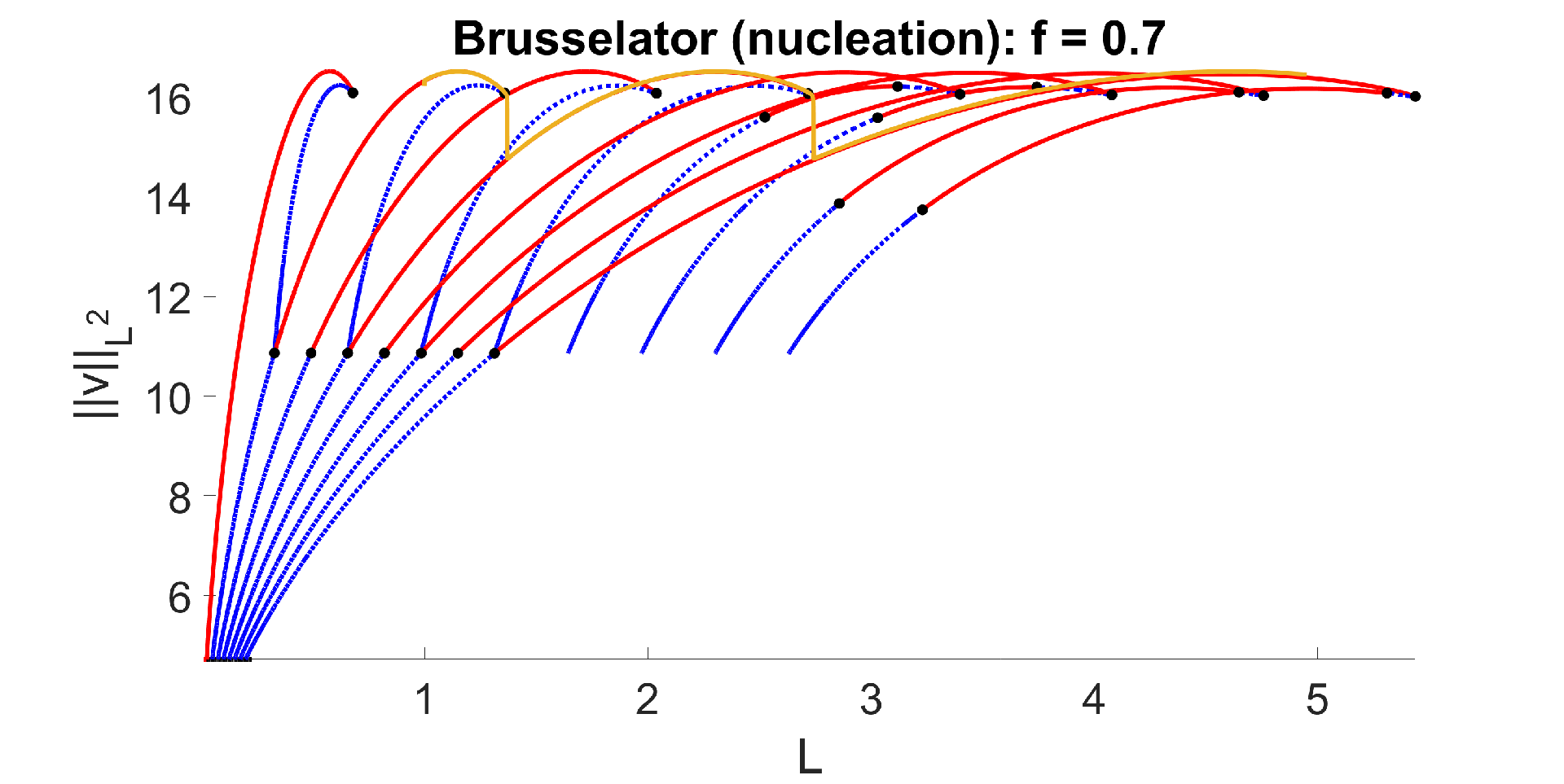}
    \includegraphics[width=0.85\textwidth,height=6.0cm]{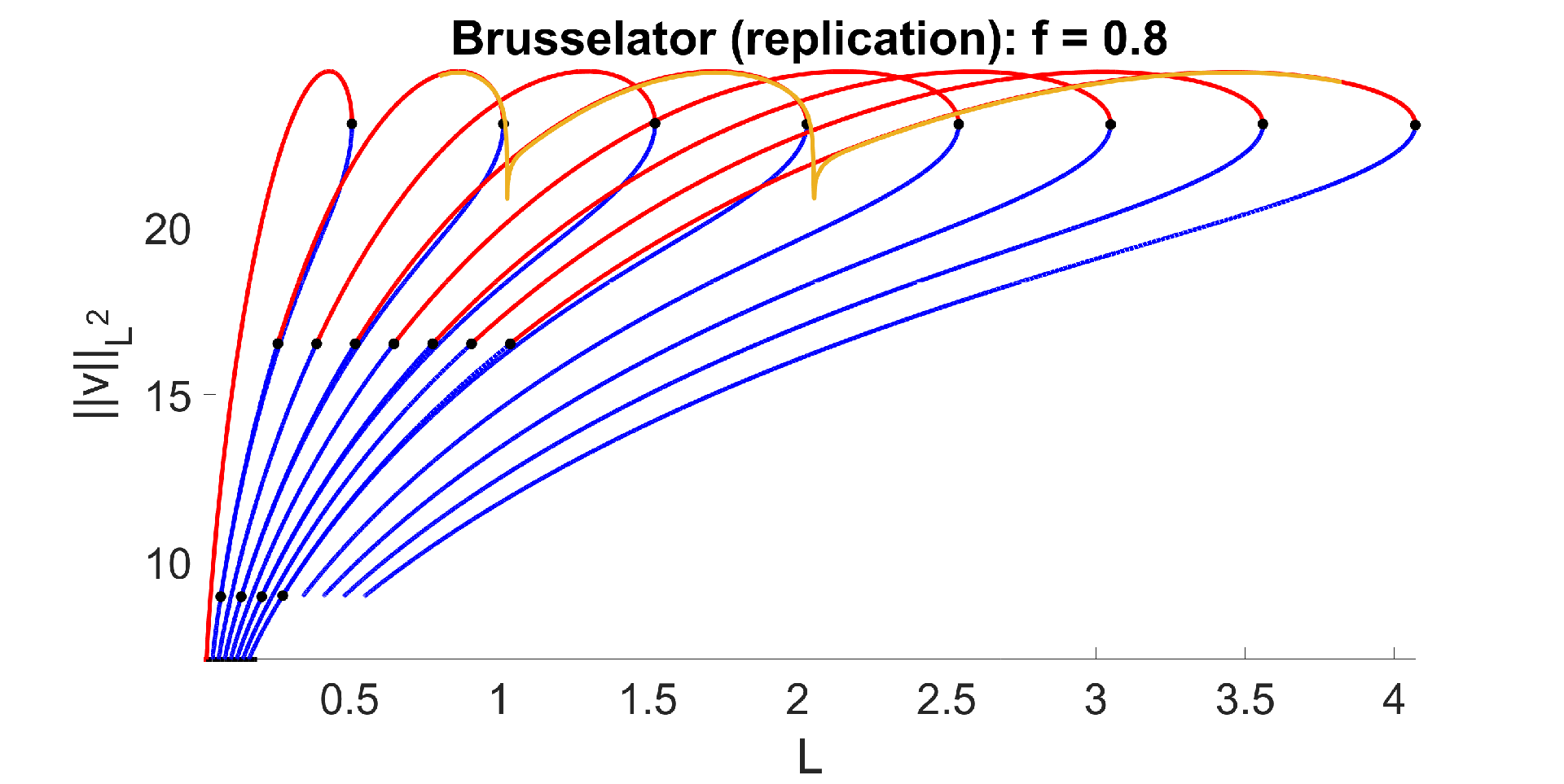}
    \caption{Same caption as in Fig.~\ref{fig:bif_Sch_merge} except now the results are for the Brusselator model \eqref{Brusselator_Lagrange}. Parameters: $\varepsilon = 0.01$, $a = 1$, $D = 2$ and $\rho = \varepsilon^2 = 0.0001$. Top panel: for $f = 0.8$ spike self-replication occurs along the yellow curve. Bottom panel: for $f=0.7$ spike nucleation occurs along the yellow curve.}
    \label{fig:bif_Bruss_merge}
\end{figure}

% \begin{figure}[htbp]
%     \centering
%     \includegraphics[width=0.85\textwidth,height=6.0cm]{Brusf0.8sols.eps}
%     \includegraphics[width=0.85\textwidth,height=6.0cm]{Brusf0.7sols.eps}
%     \caption{Same caption as in Fig.~\ref{fig:bif_Sch_merge} except
%       now the results are for the Brusselator model
%       \eqref{Brusselator_Lagrange}. Parameters: $\varepsilon = 0.01$,
%       $a=1$, $D = 2$ and $\rho=\varepsilon^2=0.0001$.  Top panel: for
%       $f = 0.8$ spike self-replication occurs along the yellow
%       curve. Bottom panel: for $f=0.7$ spike nucleation occurs along
%       the yellow curve.}
%     \label{fig:bif_Bruss_sols}
% \end{figure}

\begin{figure}[htbp]
    \centering
    \includegraphics[width=0.85\textwidth,height=6.0cm]{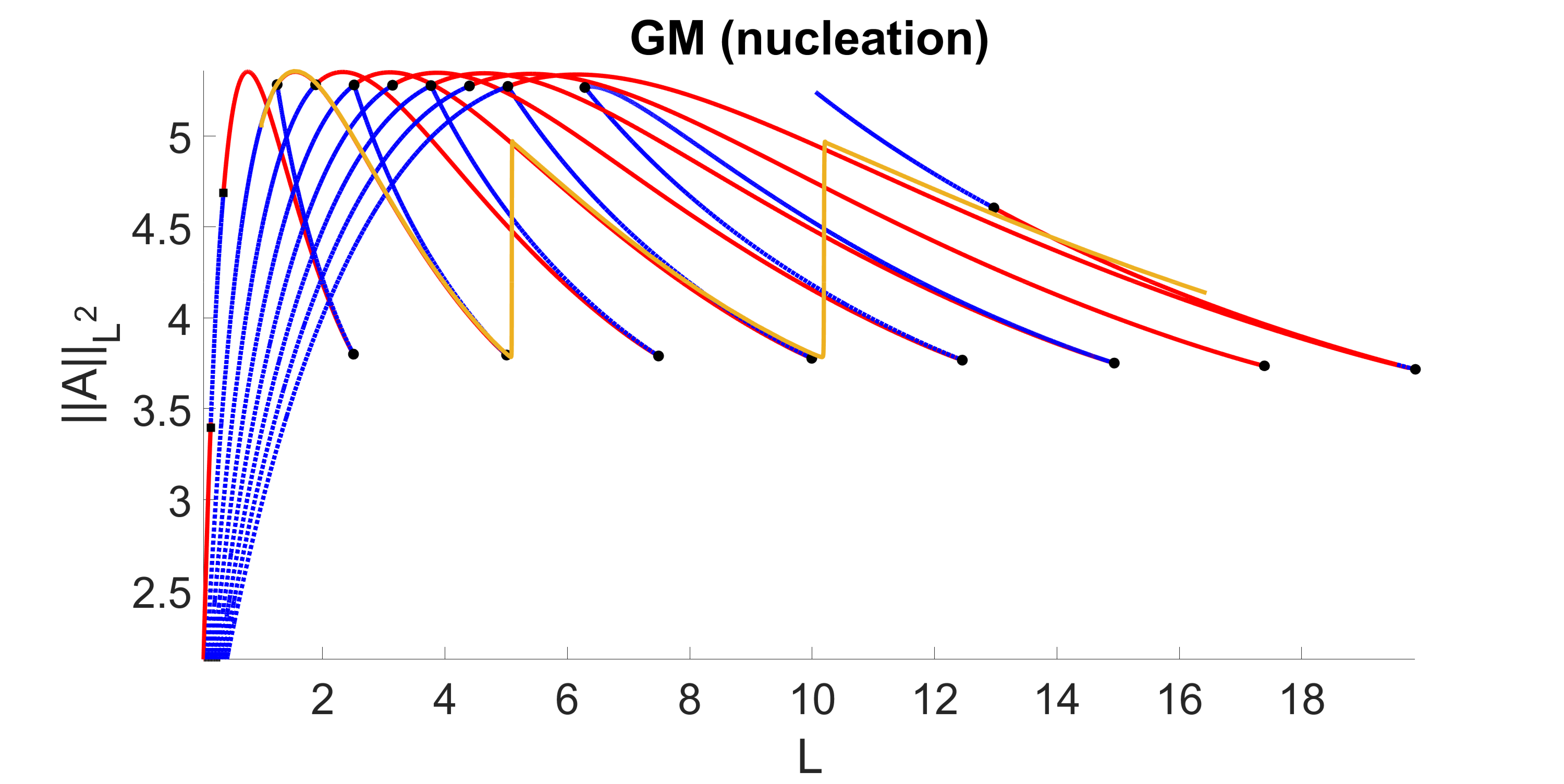}
    \caption{Plots of bifurcation branches for the GM model \eqref{GM_Lagrange}, starting from a half-spike solution and with successive increments of a half-spike, for multi-spike equilibria (red curves are linearly stable) overlaid with the time-dependent solutions (yellow curves), as computed from full \textit{FlexPDE} \cite{flexpde2015} simulations. The vertical axis is the $L_2$ norm of the spike solution $A$. Parameters: $\varepsilon = 0.02$, $D = 2$ and $\rho = \varepsilon^2 = 0.0004$. For $\kappa = 0.5$ spike nucleation occurs along the yellow curve.}
    \label{fig:bif_GM}
\end{figure}

%\begin{figure}[htbp]
%     \centering
%    \includegraphics[width=0.85\textwidth,height=6.0cm]{GMsols.eps}
%    \caption{Similar to Figure \ref{fig:bif_Sch_sols} but the
%      solutions correspond to the branches shown in Figure
%      \ref{fig:bif_GM}.}
%     \label{fig:bif_GM_sols}
% \end{figure}

For the Schnakenberg model \eqref{Sch_Lagrange} with parameters $\varepsilon = 0.04$, $D = 4$, $b = 1$, and growth rate $\rho = \varepsilon^2 = 0.0016$, in Fig.~\ref{fig:bif_Sch_merge} the quasi-steady state solutions computed from time-dependent PDE simulations are superimposed on the steady-state solution branches for $a = 0.2$ and $a = 0.5$ separately, in order to illustrate the two distinct spike-generating mechanisms. The leftmost solution branch is for a half-spike steady-state. In the top panel of Fig.~\ref{fig:bif_Sch_merge} where $a = 0.2$, we observe that the solution jumps from 1 interior spike to 2 interior spikes, and then to 4 interior spikes as $L$ increases. This clearly illustrates spike self-replication behavior as was shown in the left panel of Fig.~\ref{fig:Sch_flexpde_plot} from full PDE simulations. In the bottom panel of Fig. \ref{fig:bif_Sch_merge} where $a = 0.5$, we observe that the quasi-steady solution jumps from 1 interior spike to 2 spikes (1 interior + 2 boundary spikes), and then to 4 spikes (3 interior + 2 boundary spikes) as $L$ increases. The corresponding time-dependent spike-nucleation behavior was shown in the right panel of Fig.~\ref{fig:Sch_flexpde_plot}. For $a = 0.2$, the solutions at the diamond indicated points in the top panel of Fig.~\ref{fig:bif_Sch_merge} are shown in Fig.~\ref{fig:bif_Sch_sols}.

In Fig.~\ref{fig:bif_Bruss_merge}, we consider the Brusselator model with parameters $\varepsilon = 0.01$, $D = 2$, $a = 1$, and growth rate $\rho = \varepsilon^2 = 0.0001$. We have shown that two spike-generating mechanisms can occur depending on the parameter $f$. For $f = 0.8$, where spike self-replication is predicted, we observe from the top panel in Fig.~\ref{fig:bif_Bruss_merge} that the solution jumps from 1 interior spike to 2 interior spikes, and then to 4 interior spikes, as is consistent with time-dependent behavior shown in the left panel of Fig.~\ref{fig:Bru_flexpde_plot}.  In the bottom panel of Fig.~\ref{fig:bif_Bruss_merge} where $f=0.7$, we observe that the solution jumps from 1 interior spike to 2 spikes (1 interior + 2 boundary spikes), and then to 4 spikes (3 interior + 2 boundary spikes). This spike-nucleation behavior was shown in the full-time-dependent simulations in the right panel of Fig.~\ref{fig:Bru_flexpde_plot}.

Finally, we consider the GM model \eqref{GM_Lagrange} with $\varepsilon = 0.02$, $D = 2$ and growth rate $\rho = \varepsilon^2 = 0.0004$. For $\kappa = 0.5$, in Fig.~\ref{fig:bif_GM} the quasi-steady state solutions (yellow curves) computed from time-dependent PDE simulations are superimposed on the solution branches of spike equilibria. The leftmost solution branch in Fig.~\ref{fig:bif_GM} is for a half-spike solution. From this figure, we observe that the quasi-steady solution first jumps from 1 interior spike to 2 spikes (1 interior + 2 boundary spikes), and then to 4 spikes (3 interior + 2 boundary spikes) as $L$ increases. This illustrates the time-dependent spike nucleation behavior shown in Fig. \ref{fig:GM_flexpde_plot}.

\section{Discussion}\label{sec:discussion}

For three RD models in the semi-strong interaction regime, as characterized by an asymptotically large diffusivity ratio, we have analyzed two distinct global bifurcation mechanisms that lead to the generation of spatial patterns of increased complexity as the domain half-length $L$ slowly increases. By treating the domain length as a static parameter, we have shown that spike self-replication can occur from the passage beyond a saddle-node point associated with the local profile of a spike. In contrast, spike nucleation occurs owing to the passage beyond a saddle-node point that can be predicted by the non-existence threshold of a reduced nonlinear problem defined in the outer region away from the core of a spike. Depending on the parameter values in either the Schnakenberg or Brusselator models, we have analyzed in detail which of these two bifurcations will occur first as $L$ increases. For the GM model, we have shown that spike nucleation is the only possible spike-generating mechanism. In parameter regimes where a one-spike solution exists on the infinite line, in \S \ref{sec:no_nucleation} we showed that neither spike self-replication nor spike-nucleation will occur as $L$ increases for our three RD systems in the semi-strong interaction regime. Our theoretical predictions regarding whether self-replication or nucleation will occur are encoded in the parameter phase diagrams in Fig.~\ref{fig:phase_diag}.

For the time-dependent problem on an exponentially slowly growing domain with growth rate $\rho = \varepsilon^2$ we have shown through numerical simulations that our asymptotic theory, in which $L$ is treated as a parameter, can accurately predict critical values of the domain length where either spike self-replication or spike-nucleation will occur. Our choice of $\rho = \mathcal O(\varepsilon^2)$ is motivated by the need to ensure that the domain growth rate is asymptotically smaller than the $\mathcal O(\varepsilon)$ growth rate that is associated with the small eigenvalues of the linearization of spike equilibria (cf.~\cite{kolokolnikov2005existence}). In this way, when $\rho = \varepsilon^2$ the domain growth is sufficiently slow so that, to leading order, the quasi-steady solutions can be parameterized by $L$.

We remark that the necessary conditions for self-replicating patterns, formulated in \cite{skeleton} in the context of exponentially weakly interacting pulses on the infinite line and applied to slow domain growth in \cite{ueda} for the weak interaction regime, are still relevant to our semi-strong interaction setting. In particular, one of the two branches of steady-state solutions with $K$ spikes is linearly stable on an $\mathcal O(1)$ time scale, there is a dimple-shaped neutral eigenfunction for the linearization of $K$-spike equilibria at the saddle-node bifurcation point, and that further branches of multi-spike equilibria have a ``lining-up'' or nearly coincident structure (see \S 1 of \cite{ward_2005} for more details). This global bifurcation mechanism for transitions to patterns of increased spatial complexity was illustrated in \S \ref{sec:transitions} by our superposition of the time-dependent PDE simulation results with growth rate $\rho = \varepsilon^2$ on the solution branches of spike equilibria.

We now apply our phase diagrams to theoretically explain some previous numerical results in \cite{crampin_1999,crampin_2002,methods,tzou_bruss} illustrating either spike self-replication or spike nucleation. For the Schnakenberg model \eqref{Schnakenberg} with $a=0.1$, $b=0.9$, and a diffusivity ratio $\varepsilon^2/D = 0.01$, in Fig.~2 of \cite{crampin_1999} spike self-replication dynamics were observed for an exponentially growing domain with growth rates satisfying $10^{- 6} < \rho < 10^{- 2}$. Since $a < a_c(0.9) \approx 0.441$, as obtained from \eqref{sch:ac_approx} for $a_c(b)$, we conclude that this parameter set is in the blue-shaded region of our phase diagram in Fig.~\ref{fig:Sch_phase_diag} where we predict that spike self-replication behavior will occur as the domain length slowly increases.

Next, we consider the GM model written, as in \eqref{appgm:rd} of Appendix \ref{app:nondim}, in the form
\begin{equation}\label{disc:rd}
    A_T = D_a A_{XX} - \mu_a A + \nu_a \frac{A^2}{H} + \delta_a \,, \qquad H_T = D_h H_{XX} - \mu_h H + \nu_h A^2 \,.
\end{equation}
For the parameter set $\mu_a = \nu_a = 0.01$, $\mu_h = \nu_h = 0.02$, $\delta_a = 0.001$, $D_h = 1$ and $D_a = 0.01$, in Fig.~7b of \cite{crampin_1999} it was shown from full PDE simulations that \eqref{appgm:rd} exhibits spike nucleation behavior as the domain length slowly increases with growth rate $\rho=0.01$. As shown in \eqref{appgm:param} of Appendix \ref{app:nondim}, this parameter set corresponds to setting $\tau = 1/2$ and $\kappa = 0.1$ in \eqref{GM}, where the diffusivity ratio is $\varepsilon^2/D = 0.02$ (see also Fig.~1a of \cite{crampin_2002}). Since $\kappa < 1$, our theoretical analysis in \S \ref{sec:gm} predicts that spike nucleation behavior will occur as the domain length increases. Similarly, in Fig.~5.1b of \cite{methods}, spike nucleation was observed numerically for \eqref{appgm:rd} when $\mu_a = 0.5$, $\nu_a = \mu_h = \nu_h = 1$, $\delta_a = 0.005$, $D_a = 0.01$ and $D_h = 1$.  From \eqref{appgm:param}, this corresponds to setting $\tau = 1/2$ and $\kappa=0.05$ in \eqref{GM} where the diffusivity ratio is ${\varepsilon^2/D}=0.005$. Since $\kappa < 1$, our theoretical analysis again predicts spike nucleation behavior as the domain length slowly increases.

Next, consider the Brusselator model, written in the alternative form of \cite{tzou_bruss} as
\begin{equation}\label{disc:bruss}
  V_t=\varepsilon^2 V_{xx} + \varepsilon - V + fU V^2 \,, \qquad
  \tau U_t = D_u U_{xx} + \frac{1}{\varepsilon}\left(V-U V^2\right) \,.
\end{equation}
In Fig.~23 of \cite{tzou_bruss} it was observed from full PDE simulations on the fixed domain $|x| \leq 1$ that, with a one-spike initial data, self-replication dynamics leading to a four-spike steady-state occurs for the parameter set $\varepsilon = 0.01$, $D_u = 0.02$, $\tau = 0.001$ and $f = 0.95$. To explain this result theoretically, we introduce $V = (\varepsilon \tau)^{1/2} v$ and $U = (\varepsilon \tau)^{- 1/2} u$ in \eqref{disc:bruss} to conclude that $v$ and $u$ satisfy \eqref{Brusselator} on $|x| \leq 1$ where $D = D_u/\tau = 20$ and $a = \sqrt{\varepsilon/\tau} = \sqrt{10} \approx 3.16$ in \eqref{Brusselator}. Since $f = 0.95 > f_c \approx 0.769$, we predict that spike self-replication rather than spike-nucleation will be the pattern-generating mechanism. From Fig.~\ref{fig:Bru_core_bif_c} and Fig.~\ref{fig:Bru_Cb_vs_B} we estimate for $f = 0.95$ that $B_c = 0.245$ and $C_b = 1.36$ at the saddle-node point of the core problem. By solving \eqref{bruss:thresh_repL} numerically, in Fig.~\ref{fig:bruss_tzou} we plot $L_K^{rep}$ versus $a$ for $K = 1, 2, 4$. For $a = \sqrt{10}$ and $L = 1$, we observe from this figure that $L = 1$ is in the range $L_2^{rep} < 1 < L_4^{rep}$, where a four-spike steady-state, such as observed in the PDE simulations of \cite{tzou_bruss}, is theoretically predicted.

\begin{figure}[h!tbp]
    \centering
    \includegraphics[width=0.55\textwidth, height=6.0cm]{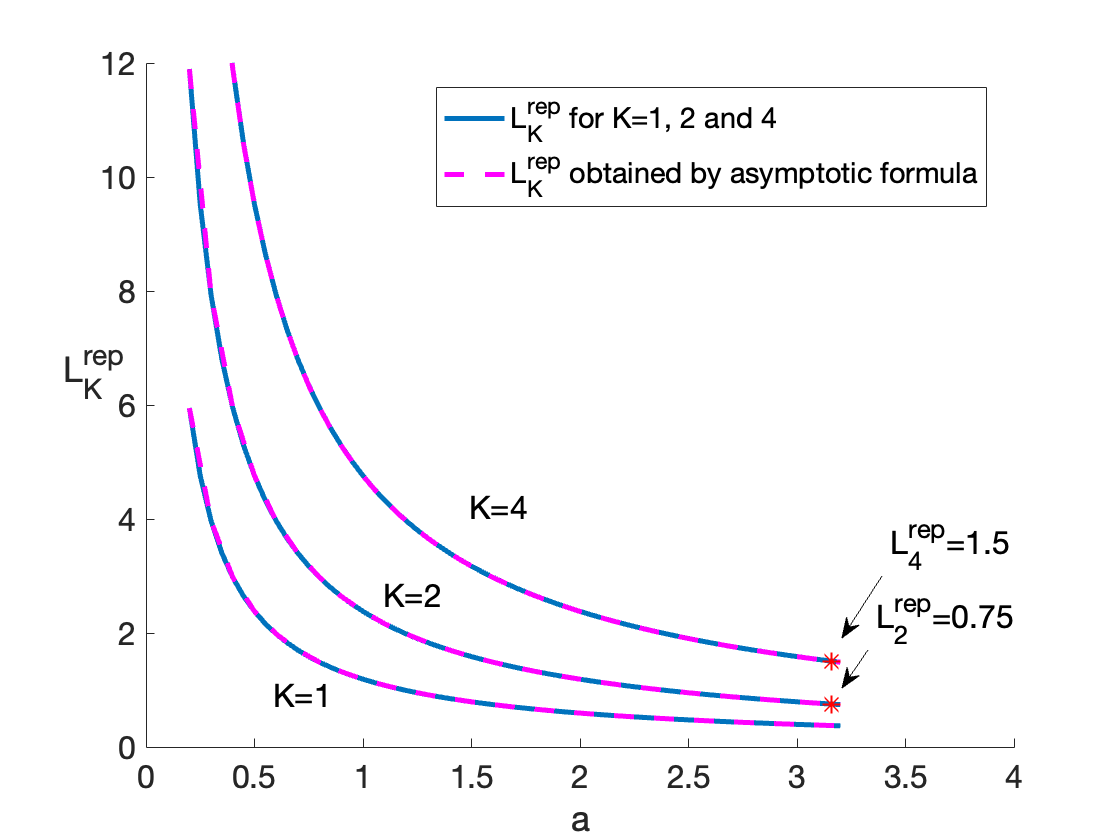}
    \caption{Spike self-replication thresholds $L_K^{rep}$ versus the parameter $a$, as computed numerically from PDE simulations of \eqref{Brusselator_Lagrange} and from the asymptotic formula \eqref{bruss:thresh_repL} for the Brusselator model \eqref{Brusselator_Lagrange} with $f = 0.95$, $\varepsilon = 0.01$ and $D = 20$ for $K = 1, 2, 4$. The spike-splitting dynamics in \cite{tzou_bruss} had $L = 1$, $D = 20$ and $a = \sqrt{10} \approx 3.16$. Since $L_2^{rep} < 1 < L_4^{rep}$, our analysis confirms the four-spike steady-state observed in \cite{tzou_bruss}.}
    \label{fig:bruss_tzou}
\end{figure}

Although our PDE simulations have focused only on slow exponential domain growth, our asymptotic theory can still be applied to other prescribed slow uniform domain growth models. In particular, for slow logistic growth where
\begin{equation*}
    L(t) = \frac{L_\infty e^{\rho t}} {L_\infty + \left(e^{\rho t} - 1\right)}\,,
\end{equation*}
with $\rho = \mathcal O(\varepsilon^2)$ and $L(\infty) = L_\infty$, the dilution terms $- \rho v$ and $- \rho u$ in \eqref{Sch_Lagrange} are replaced by $- L^{\prime} v/L$ and $-L^{\prime} u/L$, which now have time-dependent coefficients. However, when $\rho = \mathcal O(\varepsilon^2)$, in our asymptotic theory we can safely neglect these dilution terms. As a result, for parameter values where the Schnakenberg model admits spike self-replication behavior, we predict that the final steady-state pattern under logistic growth will have $2 K$ spikes when $L_\infty$ is on the range $L_K^{rep} < L_\infty < L_{2 K}^{rep}$,  where $L_K^{rep}$ is defined in \eqref{sch:thresh_repL}. Similar predictions can be made when spike-nucleation occurs. However, it would be worthwhile to determine whether the asymptotic approach developed herein can be extended to problems with nonuniform domain growth such as apical growth (cf.~\cite{crampin_2002}), or where the domain growth is controlled by the concentrations of the two species (cf.~\cite{krause_concengrowth,neville}).

\begin{figure}[htbp]
    \centering
    \includegraphics[width=0.85\textwidth]{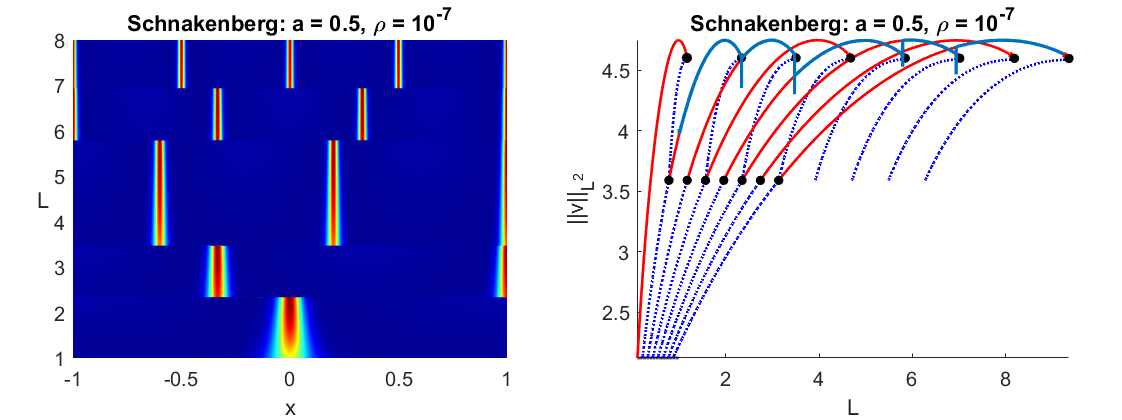}
    \caption{Mode transition failure for spike-nucleation in the Schnakenberg model with the much smaller exponential growth rate $\rho = 10^{- 7}$. Parameters: $b = 1$, $\varepsilon = 0.04$, $D = 4$ and $a = 0.5$. Left panel: time-dependent results from FlexPDE \cite{flexpde2015}. Right panel: time-dependent results superimposed on steady-state bifurcation branches computed from \textit{pde2path} \cite{pde2path}.}
    \label{fig:sch_nuc_fail}
\end{figure}

A further important extension of our study would be to analyze delayed bifurcation phenomena associated with slow domain growth that can lead to ``mode transition'' failure as was shown numerically in \cite{barrass} for the Schnakenberg model. In particular, for the much slower exponential growth rate $\rho = 10^{- 7}$, in Fig.~\ref{fig:sch_nuc_fail} we observe mode transition failure in the spike nucleation regime for \eqref{Sch_Lagrange}. Qualitatively, it is well known that when the rate of passage of a bifurcation parameter is sufficiently slow, there can be a delay in the onset of a bifurcation and that the behavior of the underlying system becomes increasingly sensitive to noise and discretization errors (cf.~\cite{gentz,kuehn_2011}). Although relatively well understood in an ODE setting, there have been relatively few analytical studies of the effect of delayed bifurcation for pattern-forming PDE systems in the nonlinear regime (cf.~\cite{knobloch_2015,knobloch_2024,tzou_2015}). Alternatively, with a significantly larger growth rate for which the quasi-steady state assumption is no longer approximately satisfied, in Fig.~\ref{fig:sch_rep_fail} we observe that mode transition failure can also occur for the Schnakenberg model in the regime where spike self-replication behavior should occur. As a result, we expect that our theoretical predictions for spike self-replication and spike nucleation should be realized numerically only in some
intermediate range of growth rates.

\begin{figure}[htbp]
    \centering
    \includegraphics[width=0.85\textwidth]{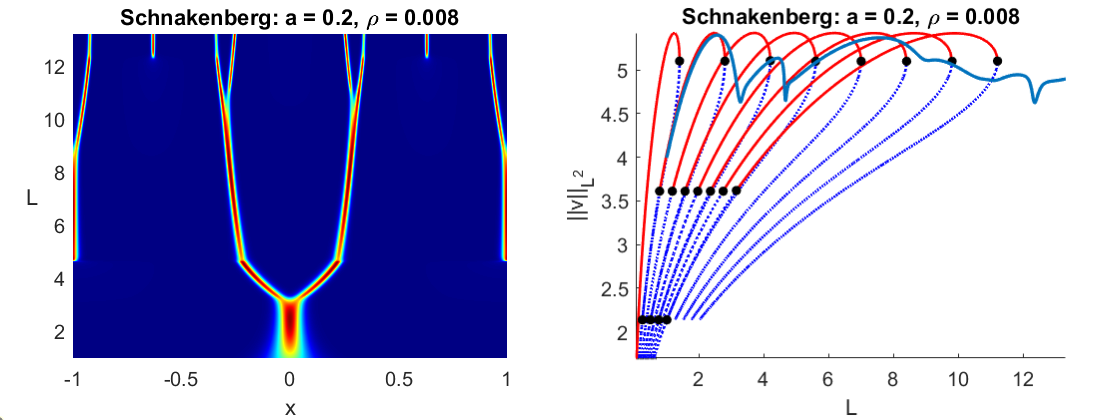}
    \caption{Mode transition failure for spike self-replication in the Schnakenberg model with the larger exponential growth rate $\rho = 0.008$. Parameters: $b = 1$, $\varepsilon = 0.04$, $D = 4$ and $a = 0.2$. Left panel: time-dependent results from FlexPDE \cite{flexpde2015}. Right panel: time-dependent results superimposed on steady-state bifurcation branches computed from \textit{pde2path} \cite{pde2path}.}
    \label{fig:sch_rep_fail}
\end{figure}

We emphasize that our analysis has focused only on pattern-generating dynamics for spike-type solutions of a few RD systems on slowly growing domains in the semi-strong interaction limit of a large diffusivity ratio. For other 1-D RD systems with bistable behavior, such as the GM model with activator saturation among others, mesa-type solutions characterized by a plateau region connected by transition layers can be constructed in the semi-strong limit. Such mesa solutions can also exhibit mesa-splitting dynamics as the domain length slowly increases (see \cite{mesa_2007}, Fig.~28 of \cite{stripe_mjw}, and the piecewise nonlinear RD model in \cite{crampin_mode}).

\begin{figure}
    \centering
    \includegraphics[width = 0.6\textwidth]{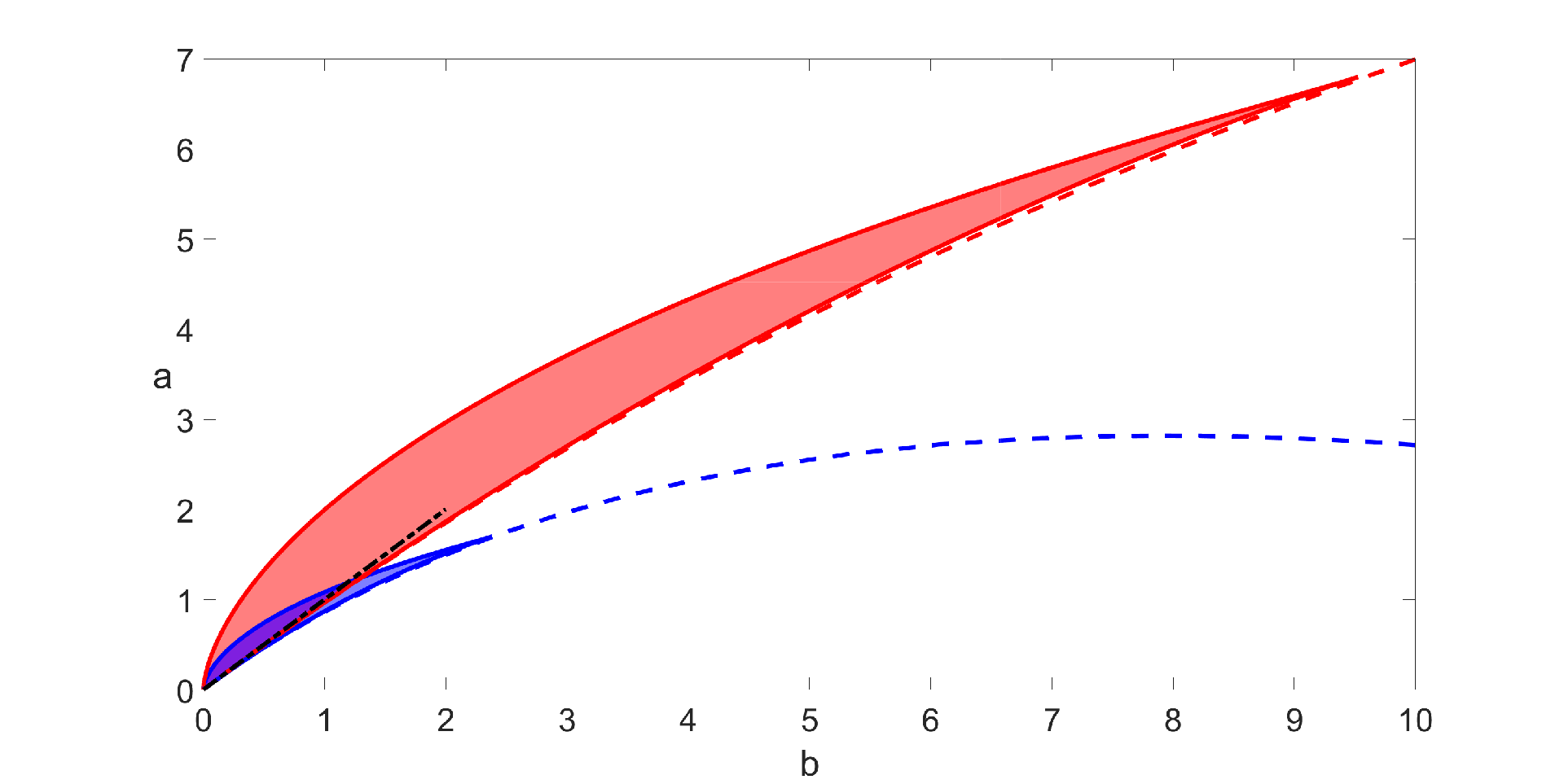}
    \caption{Numerically computed existence region in the $(b,a)$-plane for stationary localized solutions to the Schnakenberg model \eqref{Schnakenberg} on the infinite domain $L\to \infty$ for $D = 1$ and $\varepsilon = 0.04$ (blue lines, color version online)  and $\varepsilon = 0.01$ (red). The existence region is bounded by solid lines corresponding to the fold curves of single-spike solutions. The dashed lines correspond to the Turing bifurcation of the spatially homogeneous solution on the real line for the corresponding $\varepsilon$-values and the (black) dot-dashed curve is the line $a=b$ which is included for reference only. } 
    \label{fig:snaking_region_schn}
\end{figure}

It is also interesting to note the correspondence of the results obtained here for slowly growing domain length $L$ to what is known about the infinite-domain problem, see \cite{FahadWoods,FahadGrayScott,Gai,Degenerate}. Through a combination of asymptotic theory, so-called spatial dynamics arguments, and numerical continuation, it is known for each of the models \eqref{Schnakenberg}--\eqref{GM} considered here, that the Turing bifurcation of the infinite-domain homogeneous equilibrium is sub-critical for small $\varepsilon$. This leads to a region of parameters where there is bi-stability between the homogeneous equilibrium and a spatially periodic pattern, within which there is a subregion in which there are stationary localized stationary solutions, which for small-$\varepsilon$ solutions resemble the single spike solutions studied here for fixed finite $L$. 

To delineate this correspondence, we consider the Schnakenberg model (\ref{Schnakenberg}) in detail (similar results pertain for the Brusselator and GM models, but are omitted for brevity). Moreover, without loss of generality, for the infinite domain problem, we set $D = 1$. The tear-shaped region in the $(a, b)$-plane for which infinite-domain localized solutions exist is shown in Fig.~\ref{fig:snaking_region_schn} for two different $\varepsilon$-values. These localized-solutions regions were computed by numerical continuation of folds (saddle-node instabilities) of localized solutions for a large finite $L$. (Note that standard theorems from spatial dynamics show that for each localized solution on the infinite domain, there is a unique solution on $[- L, L]$ with Neumann boundary conditions that converges uniformly to it as $L\to \infty$). The requirement that solutions are exponentially localized means that the localized solutions must lie above the curve of Turing bifurcations of the homogenous equilibrium on the infinite domain, which is readily computed to be \cite{Gai}
\begin{equation}
    b - \left(a + b\right)^3{\varepsilon}^2 - a = 2 \, \left(a + b\right)^2 \varepsilon, \label{eq:Sch_Turing}
\end{equation}
and is depicted using dashed lines in Fig.~\ref{fig:snaking_region_schn}. The pinch point at the upper end of the tear is the codimension-two point at which the Turing bifurcation switches from super- to sub-critical, the precise location of which was computed in \cite{Degenerate}. Close to here, localized solutions have oscillatory tails and are connected to stable localized multi-pulse solutions through a so-called homoclinic snaking bifurcation diagram \cite{Beck,kno}. In fact, the tear-shaped region is divided in two by the line \cite{Gai}
\begin{equation}
    b - \left(a + b\right)^3{\varepsilon}^2 - a = - 2 \, \left(a + b\right)^2 \varepsilon. \label{eq:Sch_BD}
\end{equation}
(curve not shown in the figure) which delineates whether the far field approaching the homogeneous equilibrium has oscillatory or monotone decay. As one transitions from the oscillatory to the monotone region, a so-called Belyakov-Devaney transition occurs where infinitely many multi-pulse solutions are destroyed through a non-local bifurcation mechanism \cite{Nico}.

It is interesting to note how the lower part of this tear (for small $a$ and $b$) appears to have two clear asymptotes as $\varepsilon \to 0$. The upper-$a$ fold was computed asymptotically in \cite{Gai} to be 
$$
    b^2 = 12 \varepsilon^3 \left(a + b \right)^3
$$
which tends to $b = 0$ as $\varepsilon \to 0$. A Hopf bifurcation curve can be shown to exist close to this fold, which is numerically observed to be sub-critical.

In contrast, and most relevant to the present study, the lower-$a$ boundary appears to asymptote to the line $a=b$ in the limit $\varepsilon \to 0$. This corresponds exactly to the condition \eqref{eq:vinfty_sch} which is the same as the nucleation instability calculated in \S~\ref{sec:no_nucleation}. Note also that both the Turing bifurcation \eqref{eq:Sch_Turing} and Belyakov-Devaney transition \eqref{eq:Sch_BD} also asymptote to the same curve as $\varepsilon \to 0$. 

It is interesting then to consider what happens in a related problem than that studied in this paper, namely fixing $D$ and $a$ in the Schnakenberg problem for a large fixed $L$, and slowly increasing $b$ from an initial value close to a one-spike solution for small $b$. As one crosses $b = a$, one would expect to see additional spikes appear via nucleation. Then for yet-higher $b$-values, as we cross the $b$-value corresponding to $a = a_c(b)$ in Fig.~\ref{fig:sch:ac_vs_b}, further spikes should appear via spike-splitting. Figure \ref{fig:nuc-repl} shows just such a computation, where we notice that the single spike at $x=0$, first destabilizes via nucleation (from around $b=1.5$) before proceeding to generate new spikes via splitting (from around $b=4$). 

\begin{figure}
    \centering
    \includegraphics[width = \linewidth]{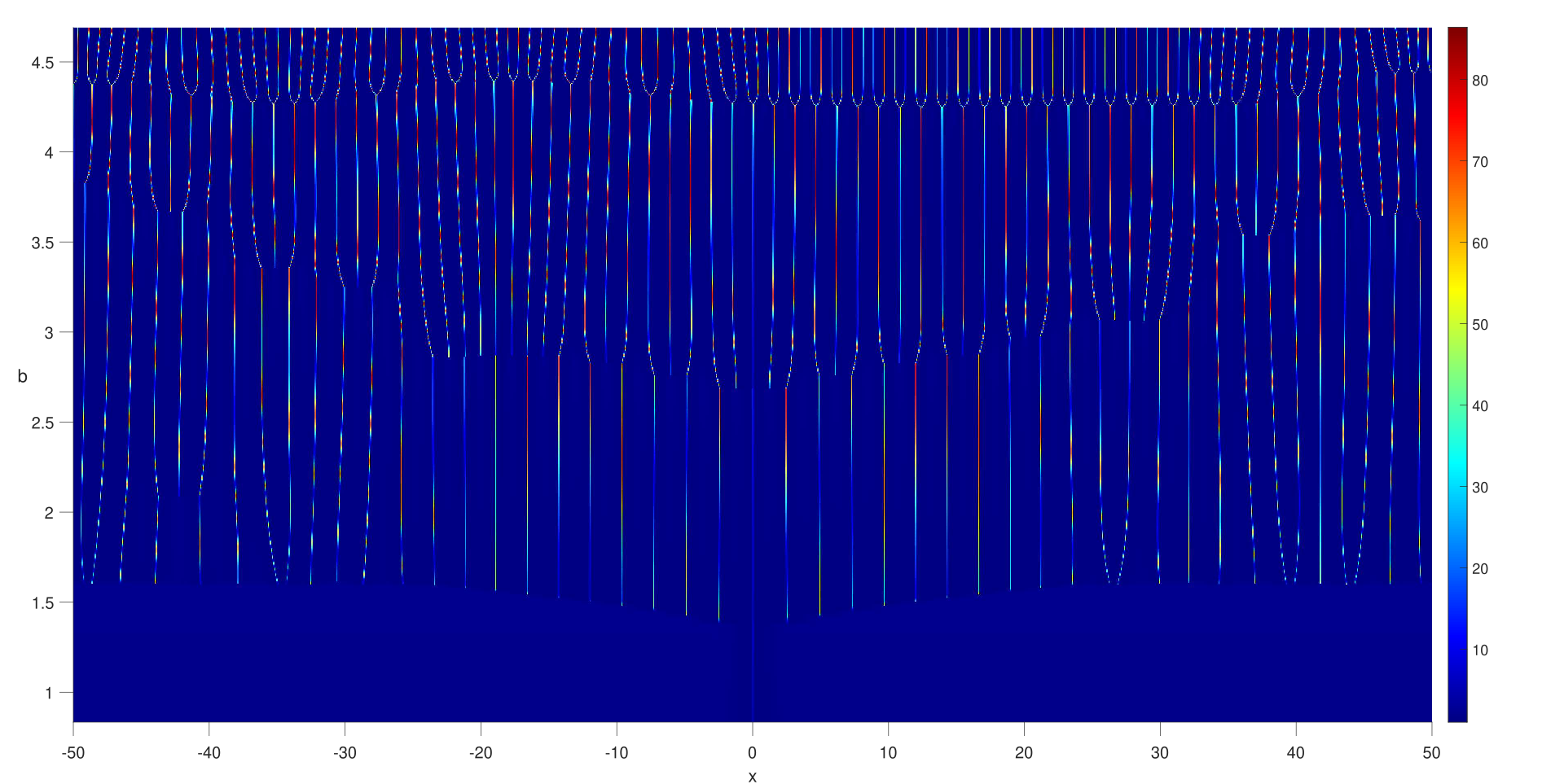}
    \caption{Time-integration of the Schnakenberg system \eqref{Schnakenberg} for $a = 1$, $\varepsilon = 0.01$, $D = 4$, $L = 50$ and $b = b_0 \, e^{\rho t}$, where $b_0 = 0.834227$ and $\rho = 10^{- 4}$, with the spike at $b = b_0$ as an initial condition. The vertical axis is the time $t$ but we show the corresponding value of $b$ at each time step to see how said value determines the phenomenon that occurs to the spikes that exist from the beginning of the integration.}
    \label{fig:nuc-repl}
\end{figure}

Finally, it is useful to discuss a few open problems in higher spatial dimensions that warrant further investigation. Firstly, in a rectangular domain that is slowly elongated horizontally, it would be interesting to analyze whether a stripe pattern, consisting of 1-D quasi-steady spike solutions trivially extended in the direction transverse to the domain growth, is linearly stable to break up into spots and to zigzag instabilities. For the GM model with $\kappa = 0$, it is well-known that for such a homoclinic stripe there is always an unstable band of transverse wavelengths for the linearized problem that leads to the breakup of a stripe into a collection of spots unless the rectangular domain is asymptotically thin (cf.~\cite{stripe_doel,stripe_mjw}). With a nontrivial background $\kappa>0$ for the GM model, and where the slow domain growth introduces the possibility of delayed bifurcations and tracking close to unstable solution branches, it would be interesting to determine if transverse instabilities can be eliminated. Secondly, for a radially symmetric domain with a single spot centered at the origin, it would be interesting to analyze the effect of slow radial domain growth for each of our three RD systems. In particular, for the Schnakenberg and Brusselator models, it would be worthwhile to determine a phase diagram in parameter space where either spot self-replication, owing to a symmetry-breaking bifurcation, or spot-nucleation initiated near the domain boundary, owing to the secondary breakup of a boundary layer solution into spots, will occur first as the domain radius slowly increases. The theoretical challenge with this analysis will be to analyze a nonlinear radially symmetric outer problem that no longer has a first integral.

%\bibliographystyle{plain} 
%\bibliography{refs} 

\section*{Acknowledgments}
Edgardo Villar-Sepúlveda was funded by ANID, Beca Chile Doctorado en el extranjero, number 72210071. Michael J. Ward gratefully acknowledges the support of the NSERC Discovery Grant Program.

\begin{appendices}
\renewcommand{\theequation}{\Alph{section}.\arabic{equation}}
\setcounter{equation}{0}

\section{The dimensionless Brusselator and GM models}
\label{app:nondim}

In this appendix, we first non-dimensionalize the Brusselator RD system of \cite{prig}, given by for $E > 0$ and $B > 0$ by
\begin{equation}\label{appb:rd}
    \mathcal V_T = D_v \mathcal V_{XX} - (B + 1) V + E + \mathcal U \mathcal V^2 \, , \qquad \mathcal U_T = D_u \mathcal U_{XX} + B \mathcal V - \mathcal U \mathcal V^2 \,.
\end{equation}
Upon introducing the scaling $\mathcal U = u_c u$, $\mathcal V = v_c v$, $t = T \sigma$ and $X = L x$ into \eqref{appb:rd} we obtain
\begin{align*}
    \frac{1}{\sigma (B + 1)} v_t &= \frac{D_v}{(B + 1) L^2} v_{xx} - v + \frac{E}{(B + 1) v_c} +  \frac{u_c v_c}{(B + 1)} u v^2 \, ,
    \\
    \frac{1}{v_c^2 \sigma} u_t &= \frac{D_u}{v_c^2 L^2} u_{xx} + \frac{B}{u_c v_c} v - u v^2 \, .
\end{align*}
By choosing $u_c v_c = B$, $\sigma = (B + 1)^{- 1}$ and $v_c = \sqrt{B + 1}$, we get
\begin{align*}
    v_t &= \frac{D_v}{(B + 1) L^2} v_{xx} - v +  \frac{E}{(B + 1)^{3/2}} + \frac{B}{B + 1} u v^2 \, ,
    \\
    u_t &= \frac{D_u}{(B + 1) L^2} u_{xx} + v - u v^2 \, .
\end{align*}
Finally, we identify the key non-dimensional parameters $f$, $a$, and $D$, which we define by
\begin{equation}\label{appd:nondim}
    a \equiv \frac{E}{(B + 1)^{3/2}} \, , \qquad f \equiv \frac{B}{B + 1} \, , \qquad D \equiv \frac{D_u}{(B + 1) L^2} \, ,
\end{equation}
and where we observe that $0<f<1$ must hold. In this way, we obtain the dimensionless form \eqref{Brusselator} for the Brusselator when we assume that $\varepsilon^2 \equiv D_v/{(B + 1) L^2} \ll 1$ and $D = \mathcal O(1)$.

In a similar way, we non-dimensionalize the Gierer-Meinhardt model of \cite{gierer}, given by
\begin{equation}\label{appgm:rd}
    A_T = D_a A_{XX} - \mu_a A + \nu_a \frac{A^2}{H} + \delta_a \,, \qquad H_T = D_h H_{XX} - \mu_h H + \nu_h A^2 \,.
\end{equation}
Upon introducing the scalings $A = A_c \mathcal A$, $H = H_c \mathcal H$ and $T = t/\mu_a$, \eqref{appgm:rd} transforms into
\begin{align*}
    \mathcal A_t &= \frac{D_a}{\mu_a} \mathcal A_{XX} - \mathcal A + \left(\frac{\nu_a A_c}{\mu_a H_c}\right) \frac{\mathcal A^2}{\mathcal H} + \frac{\delta_a}{\mu_a A_c} \,,
    \\
    \frac{\mu_a}{\mu_h} \mathcal H_t &= \frac{D_h}{\mu_h} \mathcal H_{XX} - \mathcal H + \left(\frac{\nu_h A_c^2}{\mu_h H_c}\right) \mathcal A^2\,.
\end{align*}
By choosing $A_c = \mu_h \nu_a/(\nu_h \mu_a)$ and $H_c = \mu_h \nu_a^2/(\nu_h \mu_a^2)$, we obtain the non-dimensional form given in \eqref{GM}, where we identify $\kappa$, $\tau$, $\varepsilon$ and $D$ in \eqref{GM} by
\begin{equation}\label{appgm:param}
    \tau \equiv \frac{\mu_a}{\mu_h} \,, \quad D \equiv \frac{D_h}{\mu_h} \,, \quad \varepsilon^2 \equiv \frac{D_a}{\mu_a} \,, \quad \kappa \equiv \frac{\delta_a \nu_h}{\mu_h \nu_a} \,.
\end{equation}
In \eqref{GM}, we will assume a large diffusivity ratio in the sense that
\begin{equation}\label{appgm:rat}
  \frac{\varepsilon^2}{D}=\frac{D_a\mu_h}{D_h\mu_a}\ll 1\,.
\end{equation}

\section{The eigenvalue problem for the core solution}\label{app:stab}

In this appendix, we derive the eigenvalue problem \eqref{schnak:eig_prob} that determines instabilities on an $\mathcal O(1)$ time-scale for the Schnakenberg model. With $u_e$, $v_e$ denoting the steady-state solution to \eqref{Sch_core}, we introduce the perturbation
\begin{equation}\label{apps:eigen}
    v = v_e + e^{\lambda t} \phi \,, \qquad u = u_e + e^{\lambda t} \eta\,,
\end{equation}
into \eqref{Sch_core}. Upon linearizing, we obtain the eigenvalue problem
\begin{equation}\label{apps:eigprob}
    \varepsilon_L^2 \phi_{xx} - \phi + 2 u_e v_e \phi + v_e^2 \eta = \lambda \phi \,, \qquad D_L \eta_{xx} - 2 u_e v_e \phi - v_e^2 \eta = \lambda \eta\,.
\end{equation}

In the inner region, we recall from \eqref{sch:inner} and \eqref{sch:expan} that, with $\sqrt{D_L}/\varepsilon_L = D/\sqrt{\varepsilon}$ from \eqref{sch:scale}, we have $y = x/\varepsilon_L$, $v_e \sim \sqrt{D} V_0/\varepsilon$ and $u_e \sim \varepsilon U_0/\sqrt{D}$. As a result, in the inner region \eqref{apps:eigprob} becomes
\begin{equation}\label{apps:eigprob_inner}
    \Phi_{yy} - \Phi + 2 U_0 V_0 \Phi + \frac{D}{\varepsilon^2} V_0^2 N = \lambda \Phi \, , \qquad \frac{D}{\varepsilon^2} N_{yy} - 2 U_0 V_0 \Phi - \frac{D}{\varepsilon^2} V_0^2 N = \lambda N
\end{equation}
on $- \infty < y < \infty$, where $U_0$ and $V_0$ satisfy the core problem \eqref{sch:rep_core}. Then, upon introducing the scaling
\begin{equation}\label{apps:N0}
    N = \frac{\varepsilon}{D} \left(N_0(y) + \ldots\right) \,, \qquad \Phi = \frac{1}{\varepsilon} \left(\Phi_0(y) + \ldots\right) \,,
\end{equation}
into \eqref{apps:eigprob_inner} we obtain, to leading-order, that for $\lambda = \mathcal O(1)$,
\begin{equation}\label{apps:eigprob_inner0}
    \Phi_{0yy} - \Phi_0 + 2 U_0 V_0 \Phi_0 + V_0^2 N_0 = \lambda \Phi \,; \qquad N_{0yy} - 2 U_0 V_0 \Phi_0 - V_0^2 N_0 = 0 \,.
\end{equation}
Since we will seek even eigenfunctions we impose that $\Phi_0^{\prime}(0) = N_0^{\prime}(0) = 0$ and consider \eqref{apps:eigprob_inner0} on the half-line $y\geq 0$. The natural far-field behavior of \eqref{apps:eigprob_inner0} is that $\Phi_0\to 0$ exponentially as $y\to +\infty$, and that
\begin{equation}\label{apps:eig_ff}
    N_0(y) \sim c_0 y + d_0 \,, \quad \mbox{as} \quad y\to +\infty\,, \qquad \mbox{where} \quad c\equiv \int_0^\infty
   \left(2 U_0 V_0 \Phi_0 + V_0^2 N_0\right)\,dy\,.
\end{equation}
By writing $y = x/\varepsilon_L$, where $\varepsilon_L = \varepsilon/L$, \eqref{apps:N0} and \eqref{apps:eig_ff} yield the matching conditions $\eta_x(0^+) = c_0 L/D$ and $\eta(0^+) = \mathcal O(\varepsilon)$ for the outer eigenfunction. To complete the derivation of \eqref{schnak:eig_prob} we now show that we must impose $c_0 = 0$ in finding any instabilities with $\mbox{Re}(\lambda) > 0$.
 
In the outer region $0^+ < x < \ell$, we have to leading order that $u_e v_e^2 = v_e - a$ from \eqref{Sch_out3} and $(\lambda + 1 - 2 u_e v_e) \phi = v_e^2 \eta$ from the first equation in \eqref{apps:eigprob}. Upon solving for $\phi$ and eliminating $u_e$, we conclude from the second equation in \eqref{apps:eigprob} that in the outer region $\eta(x)$ satisfies to leading-order
\begin{subequations}\label{apps:etaout}
    \begin{equation}\label{apps:etaout_prob}
        D_L \eta_{xx} - f(x, \lambda)\eta = 0 \,, \quad 0 < x < \ell \,; \quad \eta(0^+) = 0 \, , \quad \eta_x(\ell) = 0 \,,
    \end{equation}
    with $\eta_x(0^+) = c_0 L/D$. Here $f(x,\lambda)$ is defined in terms of the solution $v_e(x)$ to the outer problem \eqref{Sch_fullouter} by
    \begin{equation}\label{apps:etaout_f}
        f(x, \lambda) = \frac{(\lambda + 1) v_e^2}{\lambda - 1 + 2 a/v_e} + \lambda \,.
    \end{equation}
\end{subequations}
Since $a < v_e(x) < 2 a$, we observe that $f(x, \lambda)$ is analytic in $\mbox{Re}(\lambda) \geq 0$ for each $x$ in $0 \leq x \leq \ell$.

We now establish that, for any $\lambda$ satisfying $\mbox{Re}(\lambda) \geq 0$, we must have $\eta(x) \equiv 0$ on $0 \leq x \leq \ell$. To establish this, we first let $\lambda = \lambda_R + i \lambda_I$ with $\lambda_R \geq 0$ and from \eqref{apps:etaout_f} we calculate that
\begin{equation}\label{apps:re_f}
    \mbox{Re}(f) = \frac{\left(\lambda_R + 1\right) v_e^2 \left(\lambda_R - 1 + 2a/v_e\right) + \lambda_I^2 v_e^2}{\left(\lambda_R - 1 + 2a/v_e\right)^2 + \lambda_I^2} + \lambda_R \,.
\end{equation}
Since $a < v_e(x) < 2 a$, we conclude that $\mbox{Re}(f) > 0$ on $0 < x < \ell$ whenever $\lambda_R \geq 0$. Next, we take the conjugate of \eqref{apps:etaout_prob} and add the resulting equation to \eqref{apps:etaout_prob}. Upon integrating by parts over $0 \leq x\leq \ell$, we get
\begin{equation}
    D_L \left(\bar{\eta} \eta_x \vert_0^\ell + \eta \bar{\eta}_x\vert_0^\ell\right) - 2 \int_0^\ell \left(|\eta_x|^2 + \mbox{Re}(f) |\eta|^2\right) \, dx = 0 \,.
\end{equation}
By using $\eta(0) = \bar{\eta}(0) = \eta(\ell) = \bar{\eta}_x(\ell) = 0$, we conclude that the identity $\int_0^\ell \left(|\eta_x|^2 + \mbox{Re}(f) |\eta|^2\right) \, dx = 0$ must hold. Since $\mbox{Re}(f) > 0$ on $0 < x < \ell$, we must have that $\eta(x) \equiv 0$ and so $\eta_x(0) = 0$.

In summary, when seeking instabilities on an $\mathcal O(1)$ time scale where $\mbox{Re}(\lambda) \geq 0$, we must set $c_0 = 0$ in \eqref{apps:eig_ff} so that $N_{0y} \to 0$ as $y \to \infty$. This completes the derivation of the eigenvalue problem \eqref{schnak:eig_prob} associated with the Schnakeberg core problem.

\end{appendices}    

\bibliographystyle{plain} 
\bibliography{manuscript} 

\end{document}